%% file: DRLProgramming.tex
\documentclass[12pt]{article}
\input{tex/packages.tex}
\usepackage[many]{tcolorbox}
\newtcolorbox{revise}{colback=red!5!white,
colframe=red!75!black}
\begin{document}

\title{Deep ReLU Programming}
\author{Peter Hinz\\
	Sara van de Geer
}
\maketitle
\input{tex/abstract.tex}
\section{Introduction}
\label{sec:introduction}
\subsection{Feed-forward ReLU neural networks}
A real-valued ReLU feed-forward neural network $f:\mathbb{R}^{n_0}\to \mathbb{R}$ with input layer dimension $n_0\in\mathbb{N}$ is a composition of $L\in\mathbb{N}$ layer transition functions of possibly different output dimensions $n_1,\ldots,n_{L+1}\in\mathbb{N}_+$ with $n_{L+1}=1$ and a subsequent affine mapping, i.e.
\begin{equation}
\label{eq:f}
f:
\begin{cases}
\mathbb{R}^{n_0}&\to \mathbb{R}\\
x&\mapsto W_{L+1}g^{(L)}\circ\dots\circ g^{(1)}(x)+b_{L+1}
\end{cases},
\end{equation}
where $g^{(l)}:\mathbb{R}^{n_{l-1}}\to\mathbb{R}^{n_l}, x\mapsto \textnormal{ReLU}.(W_lx+b_l)$ with the activation function $\textnormal{ReLU}:\mathbb{R}\to\mathbb{R},  x\mapsto \max(x,0)$, matrices $W_l\in \mathbb{R}^{n_{l}\times n_{l-1}}$ and bias vectors $b_l\in\mathbb{R}^{n_l}$, $l\in\left\{ 1,\ldots,L \right\}$. The dot ``$.$'' denotes element-wise application of the $\textnormal{ReLU}$ function and the restriction $n_{L+1}=1$ reflects the fact that we consider real-valued neural networks. 
We denote the class of all such functions by 
\begin{equation}
	\label{eq:F}
	\mathcal{F}:=\left\{ f:\mathbb{R}^{n_0}\to\mathbb{R}\middle\vert\; f\textnormal{ is of the form \eqref{eq:f}} \right\},
\end{equation}where the number of layers $L$ and their widths $n_0,\ldots,n_L$ are implicit. 
The functions $f\in \mathcal{F}$ are piece-wise affine and continuous. They can be represented as a sum 
\begin{equation}
\label{eq:partition}
f(x)=\sum_{i=1}^{M}\mathds{1}_{A_i}(x)\langle x,c_i\rangle +d_i
\end{equation}
with vectors $c_1,\dots,c_{M+1}\in\mathbb{R}^{n_0}$ and real numbers $d_1,\dots,d_M\in\mathbb{R}$ and a disjoint partition of the input space $\mathbb{R}^{n_0}=\bigcup_{i=1}^MA_i$ into convex sets $A_i$, $i\in\{1,\ldots,M\}$ each being the solution of linear inequalities. The works ~\cite{Montufar:2014:NLR:2969033.2969153}, \cite{DBLP:BoundingCounting} and \cite{hinz2018framework} derive upper bounds on $M\in\mathbb{N}$ such that this representation is still possible for all $f\in \mathcal{F}$, for example for $M\le\prod_{i=1}^{L}\sum_{j=0}^{\min\left( n_0,\dots,n_{i-1} \right)}{n_i\choose j}$. However, in this work we rely on the direct canonical representation~\eqref{eq:f} specified by the weight and bias parameters.
\subsection{Deep ReLU programming}
\label{sec:drlpproblem}
Linear programming plays an important role in the field of mathematical optimization and is used in operations research as a main tool to solve many large-scale real-world problems. A linear programming problem is an optimization problem restricted by linear conditions. In its canonical form, for $n\in\mathbb{N}$, $c\in \mathbb{R}^{n_0}$, a matrix $A\in\mathbb{R}^{n\times n_0}$ and a vector $b\in\mathbb{R}^{n}$ the goal is to
\begin{equation}
\label{eq:linearProgramming}
\textnormal{minimize }\langle x,c\rangle\textnormal{ such that } x\in\mathbb{R}_{\ge0}^{n_0}\textnormal{ and } (Ax)_j\le b_j\textnormal{ for }j\in\left\{ 1,\ldots,n \right\}.
\end{equation}
This means that linear programming aims to minimize an affine map over a convex set of feasible arguments that is determined by a number of linear inequalities. 

Deep ReLU programming problems are minimization problems of functions $f\in \mathcal{F}$. Some of the functions $f\in\mathcal{F}$ are not bounded from below, i.e. attain arbitrarily small values and in general multiple local minima may exist. Hence, in contrast to linear programming problems which are convex optimization, deep ReLU programming problems are non-convex in general. Deep ReLU programming can be seen as a generalization of linear programming in two ways:

\begin{itemize}
	\item 
In a linear program we can restrict to vertices as possible candidates for the minimization. Intuitively these are the intersection points, as depicted in Figure~\ref{fig:linearProgram}.
\begin{figure}[htpb]
\centering
\includegraphics[width=0.5\textwidth]{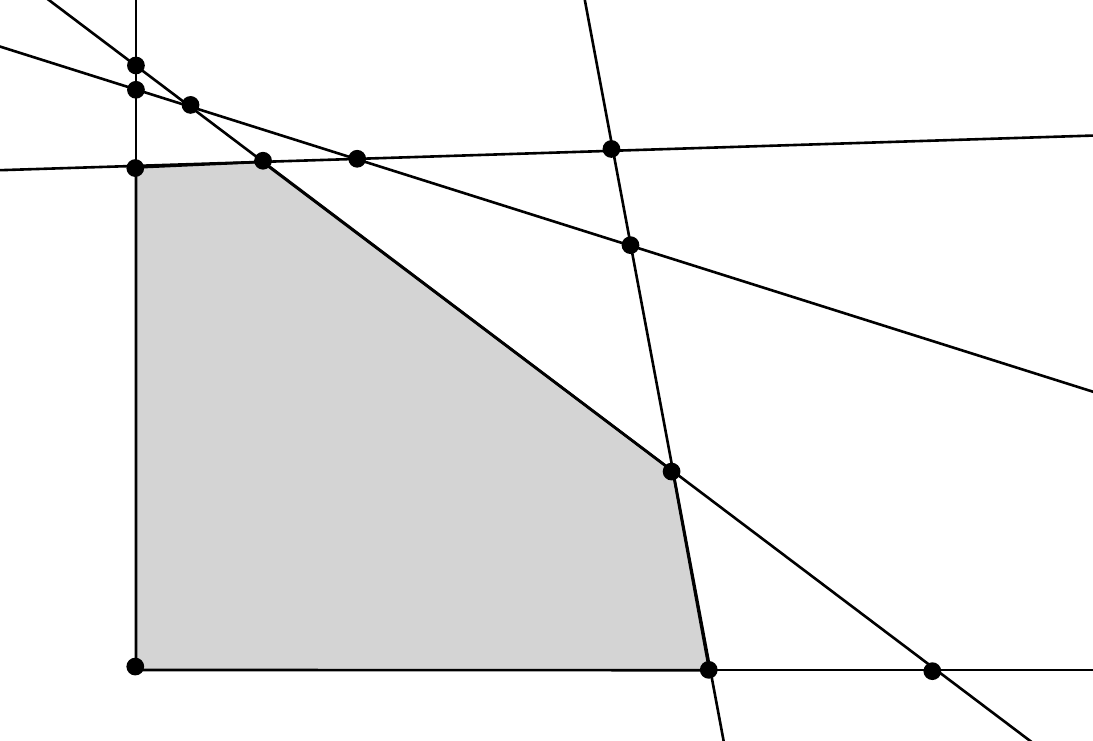}
\caption{Qualitative illustration of the argument space of a linear program with $n_0=2$. The vertices are depicted as dots and the feasible region  is represented by the gray polygon.}
\label{fig:linearProgram}
\end{figure}
If it exists, the global minimum of the objective function $x\mapsto \langle x,c\rangle$, $c\in\mathbb{R}$ will be attained in one of the vertices, more precisely in one of those vertices that are contained in the feasible region. For a given linear program as in equation~\eqref{eq:linearProgramming} and a parameter $\alpha>0$ the neural network $f_{\alpha}:\mathbb{R}^{n_0}\to\mathbb{R}$,
\begin{equation*}
	f_{\alpha}(x)=
		  \textnormal{ReLU}(\langle c,x\rangle)-\textnormal{ReLU}(\langle -c,x\rangle)+\alpha\sum_{i=1}^n\textnormal{ReLU}(\langle A_{i\cdot},x\rangle-b_i)+\alpha\sum_{i=1}^{n_0}\textnormal{ReLU}(-x_i)
\end{equation*}
is of the form~\eqref{eq:f} for $L=1$ and appropriate matrices and bias vectors. Furthermore $f_\alpha(x)=\langle c,x\rangle$ whenever $x$ is in the feasible region, i.e. for $x\ge 0$ and $Ax\le b$ element-wise. For any of the finitely many vertices and for any vertex $x^*\in\mathbb{R}^{n_0}$ that is not contained in the feasible region, the function value of $f_\alpha(x^*)$ can be made arbitrarily large by increasing $\alpha$. Hence for $\alpha$ large enough the original linear program and the problem of minimizing $f_\alpha$ have the same solutions. This shows that for every linear program, there exists a deep ReLU program that has the same vertices and global minimum.
\item Moreover, by equation \eqref{eq:partition} minimizing the restriction $f|_{A_i}$ of $f$ onto the domain $A_i$ for some $i\in\left\{ 1,\ldots,M \right\}$ yields a linear programming problem since the restriction is affine, i.e. $f|_{A_i}(x)=\langle x, c_i\rangle +d_i$ for some $c_i\in\mathbb{R}^{n_0}$ and $d_i\in\mathbb{R}$. In this setting $A_i$, which is given by linear inequalities as noted above, takes the role of the feasible region over which the function $x\mapsto \langle x,c_i\rangle=f|_{A_i}(x)-d_i$ shall be minimized. In this sense linear programming can be seen as a minimization problem of a ReLU feed-forward neural net for inputs restricted to one of its affine regions, and by dropping this condition and allowing to consider other affine regions of the neural net in the minimization problem we arrive at deep ReLU programming.
\end{itemize}
The loss functions of several machine learning methods can be written in the form of a feed-forward ReLU neural network such that parameter training of such methods has the form of deep ReLU programming. For example Linear Absolute Deviation (LAD) regression and Censored Linear Absolute Deviation (CLAD) regression require the optimization of a one and two-hidden-layer neural network respectively. More generally, in neural network training, the L1 training loss for fixed training data also has this form when considered as a function of the weight and bias parameters of a specific layer. This relevance in practical applications gives rise to a deeper analysis of the structure and solvers for deep ReLU programming problems. 

\subsection{Deep ReLU Simplex Algorithm}
A natural idea for an iterative solver for deep ReLU programming problems is the iteration on adjacent vertices of the piece-wise affine objective function. Similarly to the simplex algorithm in linear programming, the iteration should continue until we arrive at a local minimum. A modification of the simplex algorithm to convex separable piece-wise affine functions was presented 1985 in \cite{Fourer1985}. However, the functions $\mathcal{F}$ we are considering~\eqref{eq:F} are not necessarily convex functions such that our approach is different.

As explained in the previous section, in contrast to linear programming we do not have a specifically specified feasible region. Instead multiple regions of different affine behaviour are adjacent to every vertex such that one might propose an algorithm that iterates like the simplex algorithm in linear programming with such an adjacent region selected as the feasible region. Furthermore we need to change the selected feasible region to a different adjacent region of the current vertex position if necessary. In this context, the following questions arise:
\begin{itemize}
	\item How can vertices of the neural network $f$ be defined and characterized?
	\item How can the selected feasible region be specified? 
	\item How can it be changed to an adjacent affine region of $f$? Can we reuse previously computed quantities during this process?
	\item What is an optimality criterion characterizing a local minimum of $f$?
	\item What is the computational complexity of this procedure per step?
\end{itemize}

We will introduce precise mathematical definitions and develop a theory to describe the structure of ReLU neural networks and answer these questions. This allows us to formalize the above described iterative optimization process in our \emph{Deep ReLU Simplex} (DRLSimplex) algorithm, an iterative solver applicable to functions $f\in\mathcal{F}$ that are firstly regular and secondly have vertices. We will define these concepts in Section~\ref{sec:analysis}. For such functions $f$, our algorithm finds a local minimum or detects that the function $f$ is not bounded from below after a finite number of steps. 

Seen as a variant of gradient descent, our method automatically selects the step size maximally such that the affine behaviour of the target function does not change. In particular by design, the value of objective function $f$ is guaranteed to decrease in every step up to numerical uncertainty. For a given position and direction, the computation of this maximal step size of constant affine behaviour can be realized by first solving a linear equation for each neuron to find the critical step size at which its activation changes, and then taking a simple minimum over these numbers. We exploit this fact and the structure of the objective function to construct a number of algorithmic building blocks that are then composed in our DRLSimplex algorithm. Similarly to the simplex algorithm in linear programming, it traverses vertices and keeps track of a local basis corresponding to the axis directions but in contrast to the standard simplex algorithm, these axes do not correspond to edges of the feasible region. Instead they are edges separating regions of different affine behaviour of the objective function $f$. 

The essential extension of our DRLSimplex over the standard simplex algorithm is firstly the ability to change the considered affine region, thus allowing it to be used for deep ReLU programming, and secondly to efficiently update the set of axes during this process. A careful analysis shows that only one axis and the new gradient needs to be updated during this process of changing the region such that the overall runtime for each iteration of our presented algorithm is only $\mathcal{O}(\sum_{i=1}^Ln_in_{i-1}+n_0^2)$. Note that the number of weight and bias parameters of the considered neural network is $\sum_{i=1}^{L+1} n_in_{i-1}$ and $\sum_{i=1}^{L+1}n_i$ respectively with output dimension $n_{L+1}=1$ which is the reason for the order $\mathcal{O}(\sum_{i=1}^Ln_in_{i-1})$ in the runtime per step. But also the complexity term $\mathcal{O}(n_0^2)$ is unavoidable for a simplex-like algorithm because a simplex tableau has to be considered. In this sense, the complexity per step is optimal and the extension of being able to change the feasible region in our algorithm compared to simplex iterations in a fixed feasible region of the objective function comes at zero cost concerning the computational complexity per step.

Our DRLSimplex algorithm allows to strictly decrease the objective function value in every iteration. Paired with the insight that several machine learning optimization problems can be written as deep ReLU programming problems, this provides a new interesting training algorithm with a reasonable computational complexity per iteration. The iteration on vertices guarantees that exact positions of local minima are found on convergence. Furthermore, our algorithm could be used to perform novel empirical studies on the structure of piece-wise affine loss functions such as the vertex density analysis around local minima.
\subsection{Notation and structure}
We consider a vector $a\in\mathbb{R}^{n}$, $n\in\mathbb{N}$ also as a matrix with $n$ rows and one column where appropriate. Similarly the transpose $a^\intercal$ is treated as a matrix with one row and $n$ columns. Where appropriate, we define matrices block-wise based on smaller matrices and zero-filled matrices such as $
C=\tiny
\begin{pmatrix}
A& 0\\
0& B
\end{pmatrix}$ for $A\in\mathbb{R}^{a\times b}$, $B\in\mathbb{R}^{c\times d}$, $a,b,c,d\in\mathbb{N}_{>0}$. In this case, $C$ would have $a+c$ rows and $b+d$ columns. For $n\in\mathbb{N}$, the $n\times n$ identity matrix is denoted by $\textnormal{Id}_n$ and for $a\in\mathbb{R}^n$ the diagonal matrix with values $a$ on its diagonal is denoted by $\textnormal{diag}(a)$.
The dimension of a vector space $V$ and the linear span of vectors $v_1,\ldots,v_k\in V$ are denoted by $\textnormal{dim}(V)$ and $\textnormal{span}(v_1,\ldots,v_k)$ respectively.

When composing functions $\varphi^{(l)}$, $l\ge1$ we use the convention that $\varphi^{(\ell)}\circ\cdots\circ \varphi^{(1)}(x) $ means $\varphi^{(1)}(x)$ for $\ell=1$ and $x$ for $\ell=0$. 
Furthermore for a univariate function  $\varphi:\mathbb{R}\to\mathbb{R}$ the dot ``.'' denotes element-wise application. For example for $x\in\mathbb{R}^{p_1}\times\cdots\times\mathbb{R}^{p_K}$, $K,p_1,\ldots,p_K\in\mathbb{N}$, write $\varphi.(x)$ for $((\varphi(x_{1,1}),\ldots,\varphi(x_{1,p_1})),\ldots,(\varphi(x_{K,1}),\ldots,\varphi(x_{K,p_K})))$.

This paper is structured as follows: In Section~\ref{sec:analysis} we analyze mathematical properties of ReLU feed-forward neural networks and derive results needed in the subsequent sections. We then focus on algorithmic building blocks specifically designed to exploit the layer-wise structure of such networks in Section~\ref{sec:algPrimitives}. These are then combined in Section~\ref{sec:drls} with our mathematical results to construct our DRLSimplex algorithm, an adaptation of the simplex algorithm to deep ReLU programming. Here, we first, we give a description and pseudo-code in Section~\ref{sec:description}, provide a simple implementation in the Julia programming language in Section~\ref{sec:juliaImplementation} and present examples of deep ReLU programming problems that can be minimized with our algorithm in Section~\ref{sec:examples} such as L1 training of a neural network's first layer parameters.  Then we discuss our DRLSimplex algorithm in the context of linear programming and gradient descent-like procedures in Section~\ref{sec:discussion} and highlight its benefits. Finally, we summarize our findings in Section~\ref{sec:summary}. The proofs are deferred to the Appendix~\ref{app:proofs}.

\section{Mathematical Analysis}
\label{sec:analysis}
In this section we introduce the theoretical foundations that are needed in the subsequent sections. Recall that $f$ is a feed-forward neural network with ReLU activations with $L\in\mathbb{N}_+$ hidden layers and in particular a piece-wise affine function. We index the neurons by tuples $(l,j)$ where the $l\in\left\{ 1,\ldots L\right\}$ specifies the layer number and $j\in \left\{ 1,\ldots,n_l \right\}$ the neuron index within the $l$-th layer. The index set of all neurons is denoted by 
\begin{equation}
\label{eq:I}
\mathcal{I}=\left\{ (l,j)\in\mathbb{N}^2\middle\vert l\in\left\{ 1,\ldots,L \right\}, j\in\left\{ 1,\ldots,n_l \right\} \right\}.
\end{equation} 
For our analysis we need to consider the activation of the neurons of the neural network $f$ at a given input $x\in\mathbb{R}^{n_0}$. To this end we define a function that computes the arguments of the ReLU activation functions of the individual neurons. More precisely we define the \emph{ReLU arguments} by
\begin{equation}
\label{eq:args}
A:
\begin{cases}
\mathbb{R}^{n_0}&\to\mathbb{R}^{n_1}\times\cdots\times\mathbb{R}^{n_L}\\
x&\mapsto (W_1x+b_1,W_2g^{(1)}(x)+b_2,\ldots,W_Lg^{(L-1)}\circ\cdots\circ g^{(1)}(x)+b_L)
\end{cases}
\end{equation}using the notation from equation~\eqref{eq:f}.
This is a vector of vectors such that for every $(l,j)\in \mathcal{I}$ $A(x)_{lj}$ is the input of the activation function of the $j$-th neuron in the $l$-th layer.
\subsection{Hyperplane and Activation Patterns}
For every $(l,j)\in\mathcal{I}$ and $x\in\mathbb{R}^{n_{l-1}}$, the value $(W_lx+b_l)_j$ can be positive, zero or negative. If the $j$-th row of $W_l$ is non-zero, $\left\{ x\in\mathbb{R}^{n_{l-1}}\mid (W_lx+b_l)_j=0 \right\}$ defines a hyperplane in $\mathbb{R}^{n_{l-1}}$ and the three aforementioned possibilities correspond to $x$ lying on the positive half-space of this hyperplane, on the hyperplane itself or on its negative half-space. In order to describe the hyperplane relation of the ReLU arguments from equation~\eqref{eq:args} we define the \emph{hyperplane pattern} by
\begin{equation}
	\label{eq:hyperplanePattern}
H:
\begin{cases}
\mathbb{R}^{n_0}&\to\mathcal{H}\\
x&\mapsto \textnormal{sign.}(A(x))
\end{cases}
\end{equation}
with $\mathcal{H}=\left\{ -1,0,1 \right\}^{n_1}\times\cdots\times\left\{ -1,0,1 \right\}^{n_L}$ and element-wise application of the signum function defined by $\textnormal{sign}(t)=t/|t|$ for $t\neq 0$ and \textnormal{sign}(0)=0. The $\textnormal{ReLU}$-function is the identity for arguments greater that $0$, else it maps to $0$. This is why we call a neuron \emph{active} if its argument is positive and otherwise \emph{inactive}. In particular, for $(l,j)\in\mathcal{I}$ the $j$-th neuron in the $l$-th layer is active at input $x\in\mathbb{R}^{n_0}$ if $H(x)_{lj}=1$ and inactive if $H(x)_{lj}\in\left\{ -1,0 \right\}$. We encode this information in the \emph{activation pattern} defined by
\begin{equation}
\label{eq:activationPattern}
S:
\begin{cases}
\mathbb{R}^{n_0}&\to\mathcal{S}\\
x&\mapsto \textnormal{ReLU}.(H(x))
\end{cases}
\end{equation}
with $\mathcal{S}=\left\{ 0,1 \right\}^{n_1}\times\cdots\times\left\{ 0,1 \right\}^{n_L}$ and again element-wise application of the ReLU activation function.  Furthermore we define the \emph{attained activation patterns} by $\mathcal{S}_f=\big\{ S(x)\mid x\in\mathbb{R}^{n_0} \big\}$. Note that for $(l,j)\in\mathcal{I}$, the $j$-th neuron in the $l$-th layer is active at input $x$ if and only if $S(x)_{lj}=1$.
\begin{example}
\label{ex:activationPattern}
Let $L=2$, $(n_0,n_1,n_2,n_3)=(2,2,1,1)$, with weight matrices
\begin{equation*}
W_1=
\begin{pmatrix}
1&-1\\
1&1
\end{pmatrix},
W_2=
\begin{pmatrix}
1&1
\end{pmatrix},
W_3=(1),
b_1=\binom{0}{0},
b_2=-1,
b_3=0.
\end{equation*}
The neurons are indexed by $\mathcal{I}=\left\{ (1,1),(1,2),(2,2) \right\}$, the input space $\mathbb{R}^{n_0}$ is partitioned into $|\mathcal{S}_f|=7$ regions and the attained activation patterns are $\mathcal{S}_f=\left\{ 0,1 \right\}^2\times\left\{ 0,1 \right\}\setminus\big\{(\binom{0}{0},1)\big\}$ as depicted in Figure~\ref{fig:exampleSignature}. The ReLU arguments at $x=(x_1,x_2)\in\mathbb{R}^{2}$ are
\begin{equation*}
A(x)=\left( \binom{x_1-x_2}{x_2+x_2},\textnormal{ReLU}(x_1-x_2)+\textnormal{ReLU}(x_1+x_2)-1  \right)\in\mathbb{R}^2\times\mathbb{R}.
\end{equation*}
It follows that $H(x)_{1,1}=\textnormal{sign}(x_1-x_2)$ such that $\big\{x\in\mathbb{R}^2\mid H(x)_{1,1}=0\big\}=\big\{ x\in\mathbb{R}^2\mid x_1=x_2\big\}=\mathbb{R}\binom{1}{1}$ is a line or more generally, a hyperplane in $\mathbb{R}^2$. Similarly, $\big\{ x\in\mathbb{R}^2\mid H(x)_{1,2}=0 \big\}=\mathbb{R}\binom{1}{-1}$. However $H(x)_{2,1}=0$ if and only if $\textnormal{ReLU}(x_1-x_2)+\textnormal{ReLU}(x_1+x_2)=1$ such that $\big\{x\in\mathbb{R}^2\mid H(x)_{2,1}=0  \big\}$ is not a hyperplane due to the nonlinearity induced by the $\textnormal{ReLU}$-function. Instead it is is a line with two kinks, see Figure~\ref{fig:exampleSignature}.
\begin{figure}[htpb]
\centering
\includegraphics[width=0.5\textwidth]{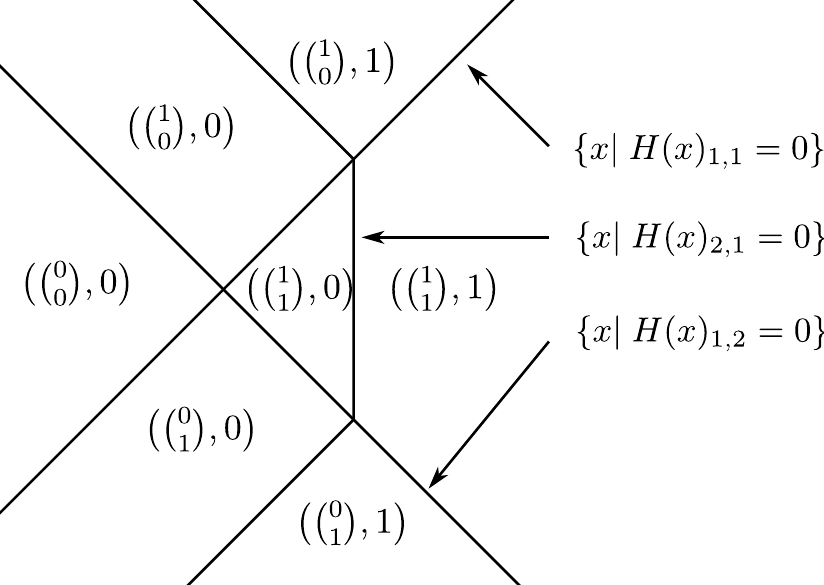}
\caption{Qualitative illustration of the partition of the input space  $\mathbb{R}^{2}$ into $7$ regions with different activation patterns as explained in Example~\ref{ex:activationPattern}.}
\label{fig:exampleSignature}
\end{figure}
\end{example}

\subsection{Objective and subjective quantities}
For the mathematical analysis of theoretical algorithmic foundations for ReLU feed-forward neural networks $f$ it is important to elaborate relations between quantities of the true function $f$ and its induced affine mappings that arise when all neuron activations are frozen, i.e. fixed to specific values. We call quantities \emph{objective} if they are based on the network $f$ and \emph{subjective} if they are based on a fixed-activation counterpart. Subjective quantities are specific to an activation pattern $s\in\mathcal{S}$ that is imposed on the network's neurons and will be consistently denoted with a tilde ``\~{}''. For example for $s\in\mathcal{S}$ we define the \emph{subjective network} $\tilde f_s$ by
\begin{equation}
	\label{eq:tildef}
\tilde{f}_{s}:
\begin{cases}
\mathbb{R}^{n_0}&\to\mathbb{R}\\
x&\mapsto 
W_{L+1}\tilde{g}^{(L)}_{s_L}\circ\cdots\circ \tilde{g}^{(1)}_{s_1}(x)+b_{L+1}
\end{cases}
\end{equation}
with 
\begin{equation}
\tilde{g}^{(l)}_{s_l}:
\begin{cases}
\mathbb{R}^{n_{l-1}}&\to\mathbb{R}^{n_{l}}\\
y&\mapsto\textnormal{diag}(s_l)\left( W_ly+b_l \right)
\end{cases}\quad \textnormal{ for }l\in\left\{ 1,\ldots,L \right\}.
\end{equation}
Note that for all inputs $x\in\mathbb{R}^{n_0}$ it holds that $\tilde f_{S(x)}(x)=f(x)$ since each neuron's $\textnormal{ReLU}$-mapping can be represented by a multiplication of $1$ if the neuron is active or by a multiplication of $0$ if the neuron is inactive, which is precisely encoded in $S(x)$. 

Similar to equation~\eqref{eq:args} we can define the \emph{subjective ReLU arguments} by
\begin{equation}
\label{eq:subjectiveArgs}
\tilde A_s:
\begin{cases}
\mathbb{R}^{n_0}&\to \mathbb{R}^{n_1}\times\cdots\times \mathbb{R}^{n_L}\\
x&\mapsto (W_1x+b_1,W_2\tilde g_{s_1}^{(1)}(x)+b_2,\ldots,W_L \tilde g_{s_{L-1}}^{(L-1)}\circ\cdots\circ \tilde g_{s_1}^{(1)}(x)+b_L)
\end{cases}
\end{equation} for an activation pattern $s\in \mathcal{S}$. The corresponding \emph{subjective hyperplane pattern} is given by
\begin{equation}
\label{eq:subjectiveHPattern}
\tilde H_s:
\begin{cases}
\mathbb{R}^{n_0}&\to\mathcal{H}\\
x&\mapsto \textnormal{sign.}(\tilde A_s(x))
\end{cases}.
\end{equation}
\subsection{Compatibility and Stability}
In equation~\eqref{eq:activationPattern} we arbitrarily defined a neuron to be inactive if its input is exactly $0$. This is a borderline case and it could just as well be defined active for this input. To reflect this fact we introduce the concept of compatibility. We say that a hyperplane pattern $h\in\mathcal{H}$ and an activation pattern $s\in\mathcal{S}$ are \emph{compatible} if 
\begin{equation}
\label{eq:compatible}
\forall (l,j)\in\mathcal{I}\quad h_{lj}(s_{lj}-0.5)\ge 0.
\end{equation}
In other words $s$ and $h$ are compatible if every neuron that is considered active has a non-negative input and every neuron that is considered inactive has a non-positive input. This seemingly simple concept will play an important role in our theory. For convenience we define
\begin{align}
\label{eq:compatibleSignatures}
\mathcal{S}^C(x)&= \big\{ s\in\mathcal{S}\mid\;s\textnormal{ is compatible with } H(x) \big\},\\
\label{def:representingSignaturesFixed}
\tilde{\mathcal{S}}^{C}(x)&=  \big\{ s\in\mathcal{S}\mid\; s\textnormal{ is compatible with }\tilde H_s(x) \big\}
\end{align}
and say that $s\in \mathcal{S}$ is compatible with $x$ if it is compatible with $H(x)$, i.e. if $s\in\mathcal{S}^C(x)$. The first quantity $\mathcal{S}^C(x)$ above is simply the set of all activation patterns that are compatible with the hyperplane pattern $H(x)$ at the input $x\in\mathbb{R}^{n_0}$. The second quantity $\tilde{\mathcal{S}}^C(x)$ is a bit more involved because $s$ is appearing twice in the set condition. It is the set of all activation patterns that are compatible with the their induced subjective hyperplane pattern at $x$. Luckily, both quantities are always equal as the following theorem states.
\begin{theorem}
\label{thm:argumentRepresentation}
For $x\in\mathbb{R}^{n_0}$ the following are true:
\begin{enumerate}
\item $\mathcal{S}^C(x)=\tilde{\mathcal{S}}^C(x)$
\item For $s\in\mathcal{S}^C(x)$ or $s\in\tilde{\mathcal{S}}^C(x)$ it holds that $A(x)=\tilde{A}_s(x)$.
\end{enumerate}
\end{theorem}
The above theorem is an important tool for the subsequent sections. Below we present two immediate consequences.
\begin{corollary}
\label{cor:representation}
For all $x\in\mathbb{R}^{n_0}$ and compatible activation patterns $s\in\mathcal{S}^C(x)$, it holds that $\tilde{f}_s(x)=f(x)$.
\end{corollary}
\begin{corollary}
\label{cor:sameHyperplanePatternSet}
For $x\in\mathbb{R}^{n_0}$ and a compatible activation pattern $s\in\mathcal{S}^{C}(x)$, it holds that
\begin{equation*}
\big\{ y\in\mathbb{R}^{n_0}\mid H(y)=H(x) \big\}=\big\{ y\in\mathbb{R}^{n_0}\mid \tilde{H}_s(y)=\tilde{H}_s(x) \big\}.
\end{equation*}
\end{corollary}
For our theory and our DRLP-algorithm it is essential to understand the relation of objective quantities specific to the true network $f$ and their subjective counterparts specific to a fixed activation. For this purpose we call a quantity specific to some input $x\in\mathbb{R}^{n_0}$ \emph{stable} if its objective and subjective version with compatible activation $s\in\mathcal{S}^C(x)$ are equal. For example the function value of the network at $x\in\mathbb{R}^{n_0}$ is stable by Corollary~\ref{cor:representation} because $f(x)=\tilde f_s(x)$ for compatible $s\in\mathcal{S}^C(x)$. Similarly, Corollary~\ref{cor:sameHyperplanePatternSet} shows that at $x\in\mathbb{R}^{n_0}$ the set of other input values with identical hyperplane pattern is stable.

In our DRLSimplex algorithm we will iterate on vertices of a ReLU feed-forward network which we will define later. These are the intersection points of several affine regions of the network and in particular the activation pattern corresponding to a vertex is not canonically given. Instead we will keep track of the activation pattern in a dedicated state parameter. As long as this activation pattern $s\in\mathcal{S}$ is compatible with the current vertex position, for stable quantities we automatically obtain their objective version by computing the subjective version corresponding to $s$. To exploit this fact it is essential to prove stability of all quantities used in our algorithm. Simply put, our algorithm sees the neural network with imposed compatible activation $s$ at input $x$ and to be sure that derived quantities are not depending on the choice of $s$ we need the stability property of all such relevant quantities.

\subsection{Local behaviour}
We now want to focus on the local behaviour of feed-forward neural networks with ReLU activation functions. For $x\in\mathbb{R}^{n_0}$ and $\varepsilon>0$ we denote the $\varepsilon$-\emph{ball} around $x$ by $B_\varepsilon(x):=\left\{ y\in\mathbb{R}^{n_0}\middle\vert\;\|x-y\|_2<\varepsilon \right\}$. 
\subsubsection{Critical Indices}
\label{sec:criticalIndices}
At every input $x\in\mathbb{R}^{n_0}$ the hyperplane $H(y)$ pattern might or might not be constant for $y$ in an infinitely small neighbourhood around $x$. If it is not constant, there are one or more indices $(l,j)\in\mathcal{I}$ describing the positions of corresponding neurons whose activation is not constant. We call these indices the \emph{critical indices} and denote them by
\begin{equation}
\label{eq:C}
C(x)= \left\{ (l,j)\in \mathcal{I} 
\middle\vert\;\forall \varepsilon>0\;|
\left\{ H(y)_{l,j}\middle\vert\; y\in B_\varepsilon(x) \right\}
|>1
\right\}\subset \mathcal{I},
\end{equation}
where $\mathcal{P}$ denotes the power set. Similarly we define the \emph{subjective critical indices} of $x\in\mathbb{R}^{n_0}$ by
\begin{equation}
\label{eq:subjectiveC}
\tilde{C}_s(x)=
\{(l,j)\in\mathcal{I}\vert \;\forall \varepsilon>0\;\vert\{ \tilde{H}_s(y)_{l,j}\vert\; y\in B_\varepsilon(x) \}\vert>1  \}\quad\textnormal{ for }s\in\mathcal{S}.
\end{equation}
\begin{lemma}
\label{lem:zero}
For $x\in\mathbb{R}^{n_0}$ and all $(l,j)\in C(x)$, it holds that $H(x)_{l,j}=0$ and for all $s\in\mathcal{S}$ and $(l,j)\in \tilde C_s(x)$ it holds that $\tilde H_s(x)_{l,j}=0$.
\begin{proof}
If  $H(x)_{l,j}=c\in\left\{ -1,1 \right\}$ for some $x\in\mathbb{R}^{n_0}$ and $(l,j)\in \mathcal{I}$, then by continuity of $y\mapsto A(y)_{l,j}$ for $\varepsilon$ small enough $|\left\{ H(y)_{l,j}\middle\vert\; y\in B_{\varepsilon} \right\}|=|\left\{ c \right\}|=1\not > 1$ such that $(l,j)\not \in C(x)$. The case for $\tilde H_s$ is along similar lines.
\end{proof}
\end{lemma}
Unfortunately, the critical indices are not stable, i.e. for the subjective and objective versions equality $C(x)=\tilde C_s(x)$ does not hold for all $x\in\mathbb{R}^{n_0}$ and compatible $s\in\mathcal{S}^C(x)$ as the following two examples demonstrate.
\begin{example}
\label{ex:CSub}
Let $L=2$, $(n_0,n_1,n_2,n_3)=(1,2,1,1)$, $b_1=(0,0)^\intercal$, $W_1=(1,1)^\intercal$, $W_2=(1,-1)$, $b_2=0$, $W_3=(1)$, $b_3=0$.
It holds that
\begin{align*}
C(0)&= ((1,1), (1,2))\subsetneq \mathcal{I}=\left\{ (1,1),(1,2),(2,1) \right\},\\
\mathcal{S}^C(0)&= \left\{ s\in\left\{ 0,1 \right\}^2\times \left\{ 0,1 \right\}\mid\; s\textnormal{ is compatible with }H(0)=
\left( 
\tbinom{0}{0},0 \right)
\right\}\\
&= \left\{ 0,1 \right\}^2\times \left\{ 0,1 \right\},\\
\tilde A_{\left( \binom{0}{1},1 \right)}(x)&= \left( \tbinom{x}{x},-x \right),\\
\tilde C_{\left( \binom{0}{1},1 \right)}(0)&= \left\{ (1,1),(1,2)),(2,1) \right\}=\mathcal{I}.
\end{align*}
In particular $C(0)\subsetneq \tilde C_{s}(0)$ for $s=\big( \binom{0}{1},1 \big)\in\mathcal{S}^{c}(0)$
\end{example}
\begin{example}
\label{ex:CSup}
Let $L=2$, $(n_0,n_1,n_2,n_3)=(1,2,1,1)$, $b_1=(0,0)^\intercal$, $W_1=(1,-1)^\intercal$, $W_2=(1,-1)$, $b_2=0$, $W_3=(1)$, $b_3=0$. It holds that
\begin{align*}
C(0)&= ((1,1), (1,2), (2,2))= \mathcal{I}\\
\mathcal{S}^c(0)&= \left\{ s\in\left\{ 0,1 \right\}^2\times \left\{ 0,1 \right\}\mid\; s\textnormal{ is compatible with }H(0)=
\left( 
\tbinom{0}{0},0 \right)
\right\}\\
&= \left\{ 0,1 \right\}^{2}\times \left\{ 0,1 \right\},\\
\tilde A_{\left( \binom{0}{0},1 \right)}(x)&= \left( \tbinom{x}{x},0 \right),\\
\tilde C_{\left( \binom{0}{0},1 \right)}(0)&= \left\{ (1,1),(1,2))\right\}\subsetneq\mathcal{I}.
\end{align*}
In particular $C(0)\supsetneq \tilde C_s(0)$ for $s=\big( \tbinom{0}{0},1 \big)\in\mathcal{S}^{c}(0)$.
\end{example}
Later in Section~\ref{sec:regularity} we will introduce a \emph{regularity} condition for which Theorem~\ref{thm:regularity} proves stability of the critical indices.
\subsubsection{Critical Kernel}
We now want to focus on how the hyperplane pattern changes locally. To describe the directions in which it does not change, we define the \emph{critical kernel} at $x\in\mathbb{R}^{n_0}$ by
\begin{equation}
\label{eq:hyperKernel}
\textnormal{Ker}^C(x)=\left\{ v\in\mathbb{R}^{n_0}\mid \exists \varepsilon^*>0\;\forall t\in(-\varepsilon^*,\varepsilon^*)\;H(x+tv)=H(x) \right\}.
\end{equation}
Similarly, for a specific activation pattern $s\in\mathcal{S}$ we define the \emph{subjective critical kernel} by
\begin{equation}
\tilde{\textnormal{Ker}}_s^C(x)=\left\{ v\in\mathbb{R}^{n_0}\mid \exists \varepsilon^*>0\;\forall t\in(-\varepsilon^*,\varepsilon^*)\;\tilde H_s(x+t v)=\tilde H_s(x) \right\}.
\end{equation}
The critical kernel at some input $x\in\mathbb{R}^{n_{0}}$ is stable as the following lemma shows.
\begin{lemma}
\label{lem:sameKernel}
For all $x\in\mathbb{R}^{n_0}$ and $s\in\mathcal{S}^{C}(x)$ it holds that $\textnormal{Ker}^C(x)=\tilde{\textnormal{Ker}}{}^C_s(x)$.
\begin{proof}
This follows immediately from Corollary~\ref{cor:sameHyperplanePatternSet}.
\end{proof}
\end{lemma}
The above result states that we can find directions in which the hyperplane pattern of the neural network $f$ does not change locally by only considering the subjective affine functions $\tilde A_s(x)_{lj}$, $(l,j)\in\mathcal{I}$ for some compatible $s\in\mathcal{S}^C(x)$. In other words, for $s\in\mathcal{S}(x)$ the subjective quantity $\tilde{\textnormal{Ker}}{}_s^C(x)$ is a universal property and not specific to $s$. The following reformulation makes this result more useful.
\begin{lemma} For all $x\in\mathbb{R}^{n_0}$ and $s\in\mathcal{S}$ it holds that
\label{lem:tildeKernel}
\begin{equation}
	\label{eq:tildeKernel}
	\tilde{\textnormal{Ker}}{}_s^C(x)=\bigcap_{(l,j)\in\tilde C_s(x)} \left\{ v\in\mathbb{R}^{n_0}\mid \tilde A_s(x+v)_{lj}=0 \right\}.
\end{equation}
\end{lemma}
This lemma shows that for $s\in\mathcal{S}$ the subjective critical kernel is a linear subspace of $\mathbb{R}^{n_0}$. In connection with Lemma~\ref{lem:sameKernel} it states that for every $x\in\mathbb{R}^{n_0}$ the critical kernel $\textnormal{Ker}^C(x)$ is the solution set of all $v\in \mathbb{R}^{n_0}$ satisfying the linear equation system $\tilde A_s(x+v)_{lj}=0$ for all $(l,j)\in\tilde C_s(x)$. It is remarkable that this holds irrespective of the choice of $s\in\mathcal{S}^C(x)$. Despite the fact that firstly for different $s,s'\in\mathcal{S}^{C}(x)$ the subjective critical indices might not have the same cardinality $|\tilde C_s(x)|\neq|\tilde C_{s'}(x)|$ and secondly, the affine functions $\tilde A_s$ and $\tilde A_{s'}$ can differ, the subjective critical kernels are the same, i.e. $\tilde{\textnormal{Ker}}{}_s^C(x)=\tilde{\textnormal{Ker}}{}_{s'}^C(x)$.
\begin{lemma}
\label{lem:hyperKernelEquality}
For $x\in\mathbb{R}^{n_0}$ there exists an $\varepsilon^*>0$ such that
\begin{equation}
B_{\varepsilon^*}(0)\cap \left\{ v\in\mathbb{R}^{n_0}\middle\vert H(x+v)=H(x) \right\}=B_{\varepsilon^*}(0) \cap \textnormal{Ker}^C(x)
\end{equation}
\end{lemma}
The above lemma states that locally around $x\in\mathbb{R}^{n_0}$ the solutions $x'\in\mathbb{R}^{n_0}$ to $H(x')=H(x)$ are precisely the affine subspace $x+\textnormal{Ker}^C(x)$. We call its dimension the \emph{critical degrees of freedom}
\begin{equation}
\label{eq:df}
\textnormal{df}^C(x)=\textnormal{dim}\left( \textnormal{Ker}^C(x) \right).
\end{equation}
Informally, this is the dimension of the vector space of directions which do not change the sign of any neuron output before the ReLU activation function is applied. We have shown that we can compute this dimension by picking an arbitrary $s\in\mathcal{S}^C(x)$ and considering the kernel of a linear equation system~\eqref{eq:tildeKernel}.  
Note that by Lemma~\ref{lem:tildeKernel} we have the following lower bound:
\begin{lemma}
\label{lem:dfge}
For $x\in\mathbb{R}^{n_0}$ and $s\in\mathcal{S}^C(x)$ it holds that $\textnormal{df}^C(x)\ge n_0-|\tilde C_s(x)|$.
\end{lemma}
\subsection{Normal vectors}
For $s\in\mathcal{S}$ the subjective neural network $\tilde f_s$ is an affine function. The same holds for the subjective ReLU arguments $x\mapsto \tilde A_s(x)_{lj}$ for every neuron position $(l,j)\in \mathcal{I}$. By equation~\eqref{eq:subjectiveArgs} it follows that 
\begin{equation}
\label{eq:explFormulation}
\tilde A_s(x)_{lj}=\langle x, \tilde v_{s,l,j}\rangle+\tilde w_{s,l,j}
\end{equation}with 
$\tilde w_{s,l,j}=b_{lj} + \left( \sum_{\ell=1}^{l-1} W_l\textnormal{diag}(s_{l-1})\cdots W_{\ell+1}\textnormal{diag}(s_\ell)b_\ell \right)_j $ and
\begin{equation}
	\label{eq:v}
\tilde v_{s,l,j}=\left( W_l\textnormal{diag}(s_{l-1})W_{l-1}\cdots \textnormal{diag}(s_1)W_1 \right)_{j\cdot}
\end{equation}
for $s\in\mathcal{S}$ and $(l,j)\in\mathcal{I}$. We call $(\tilde v_{s,l,j})_{(l,j)\in\mathcal{I}}$, the \emph{normal vectors} for $s\in\mathcal{S}^C(x)$. Figure~\ref{fig:normalVector} illustrates this concept. For $(l,j)\in\mathcal{I}$, the set $\big\{ x\in\mathbb{R}^2\mid \tilde A_s(x)_{lj}=0 \big\}$ is a $1$-dimensional hyperplane if and only if the normal vector $\tilde v_{s,l,j}$ is non-zero. In contrast, the objective version without a fixed activation $\big\{ x\in\mathbb{R}^2\mid A(x)_{lj}=0 \big\}$ is in general not a hyperplane as explained previously in Example~\ref{ex:activationPattern}.
\begin{figure}[htpb]
\centering
\includegraphics[width=0.6\linewidth]{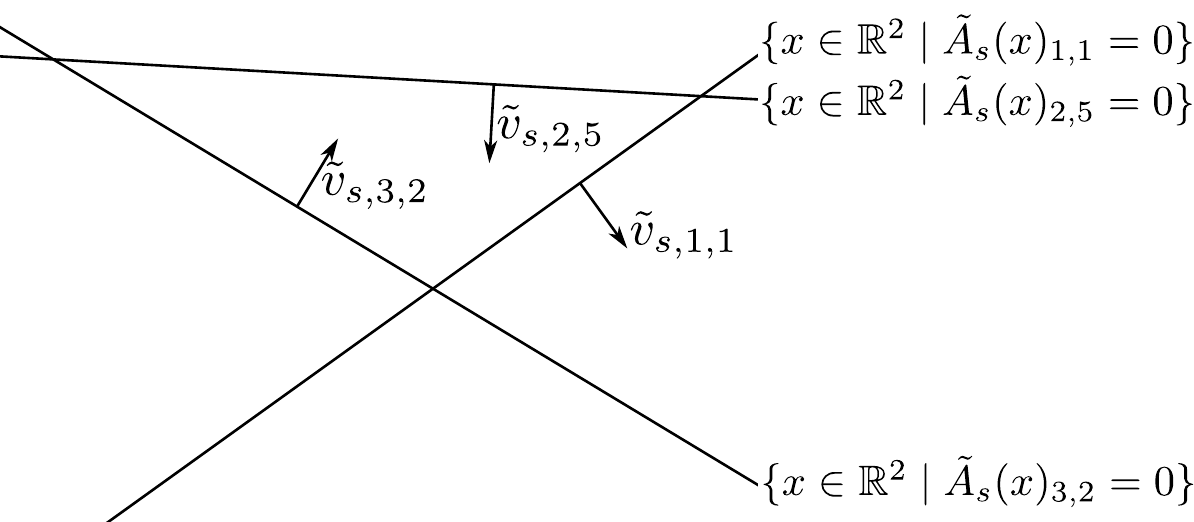}
\caption{Generic qualitative illustration of normal vectors for $n_0=2$. The non-zero normal vectors $\tilde v_{s,l,j}$ are precisely the normal vectors of the hyperplanes $\{x\in\mathbb{R}^2|\; \tilde A_{s}(x)_{l,j}=0\}$ for $(l,j)\in\mathcal{I}$.}
\label{fig:normalVector}
\end{figure}
Below we present two results about these normal vectors.
\begin{lemma}
\label{lem:tildeV}
For $x\in\mathbb{R}^{n_0}$, $s\in\mathcal{S}^C(x)$ and $(l,j)\in \tilde C_s(x)$ and $v\in\mathbb{R}^{n_0}$ it holds that
$\tilde A_s(x+v)_{lj}=\langle v,\tilde v_{s,l,j}\rangle.$
\begin{proof}
Equation~\eqref{eq:explFormulation} and Lemma~\ref{lem:zero} imply $\tilde A_s(x+v)_{lj}=\tilde A_s(x)_{lj}+\langle v,\tilde v_{s,l,j}\rangle=\langle v,\tilde v_{s,l,j}\rangle$.
\end{proof}
\end{lemma}
\begin{lemma}
Let $x\in\mathbb{R}^{n_0}$ and $s\in\mathcal{S}$. For all $(l,j)\in \tilde C_s(x)$ it holds that $\tilde v_{s,l,j}\neq 0$.
\label{lem:nonzeroChanging}
\begin{proof}
By equation~\eqref{eq:subjectiveC}, $x\mapsto \tilde A_{s}(x)_{lj}$ cannot be constant. The statement now follows from equation~\eqref{eq:explFormulation}.
\end{proof}
\end{lemma}

\begin{figure}[htpb]
\centering
\includegraphics[width=0.4\linewidth]{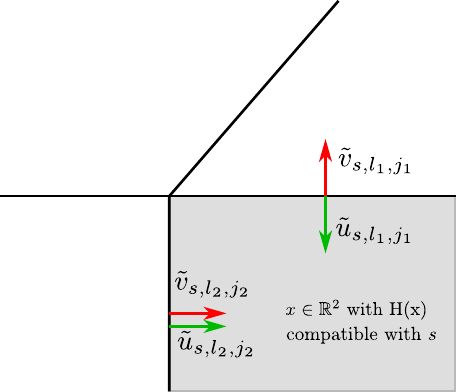}
\caption{Intuitively, for $x\in\mathbb{R}^{n_0}$, $s=S(x)$ and $(l,j)\in\mathcal{I}$, the direction of non-zero oriented normal vectors $\tilde u_{s,l,j}$ (green) specifies on which side of the hyperplane $\{ x'\in\mathbb{R}^{n_0}|\;\tilde A(x')_{lj}=0\}$ the argument $x$ lies. For this purpose, the normal vector $\tilde v_{s,l,j}$ (red) is flipped if necessary.}
	\label{fig:orientedNormal}
\end{figure}
The description of our algorithm furthermore requires a slightly modified version. For $s\in\mathcal{S}$ and $(l,j)\in\mathcal{I}$ we define the \emph{oriented normal vector} $\tilde u_{s,l,j}\in\mathbb{R}^{n_0}$ by
\begin{equation}
	\label{eq:orientedNormal}
	\tilde u_{s,l,j}=
	\begin{cases}
		\tilde v_{s,l,j}&\quad \textnormal{ if }s_{lj}=1\\
		-\tilde v_{s,l,j}&\quad \textnormal{ if }s_{lj}=0
	\end{cases}.
\end{equation}
For each activation pattern $s\in\mathcal{S}$ and neuron position $(l,j)\in\mathcal{I}$, the direction of the oriented normal vector $\tilde v_{s,l,j}$ indicates on which side of the hyperplane $\{ x\in\mathbb{R}^{n_0}| \tilde A_s(x)_{lj}=0 \}$ the points $x\in\mathbb{R}^{n_0}$ with activation pattern $S(x)=s$ are located, given that such points exist, see Figure~\ref{fig:orientedNormal}. In fact, in terms of linear programming, they can be used to define the ``feasible region'' $\{x\in\mathbb{R}^{n_0}|S(x)=s\}$ and closely related, we will base our definition of the feasible directions on the oriented normal vectors, see Section~\ref{sec:feasibleDirections}.
\subsection{Regularity}
\label{sec:regularity}
We call $x\in\mathbb{R}^{n_0}$ a \emph{vertex} if $\textnormal{df}^C(x)=0$. Furthermore, $x\in\mathbb{R}^{n_0}$ is a \emph{regular point} if 
the normal vectors $\left( \tilde v_{s,l,j} \right)_{(l,j)\in \tilde C_s(x)}$ are linearly independent for all compatible $s\in \mathcal{S}^C(x)$. If both conditions hold, $x$ is a \emph{regular vertex}.  Furthermore,  we call the ReLU feed-forward neural network $f$ \emph{regular} if all points $x\in\mathbb{R}^{n_0}$ are regular points. 
In Section~\ref{sec:criticalIndices} we saw that the critical indices are not stable, i.e. that $C(x)=\tilde C_s(x)$ does not hold for all $x\in\mathbb{R}$ and compatible activation patterns $s\in\mathcal{S}^{C}(x)$. However, if we restrict to regular arguments $x$, stability does hold.
\begin{theorem}
	\label{thm:regularity}
	For a regular point $x\in\mathbb{R}^{n_0}$ it holds that $C(x)=\tilde C_s(x)$ for all $s\in \mathcal{S}^C(x)$.
\end{theorem}
The above result is an important statement required to understand correctness of our proposed DRLSimplex algorithm in Section~\ref{sec:drls} because it implies that for every regular vertex $x\in\mathbb{R}^{n_0}$ the subjective critical indices $\tilde C_s(x)$ are always identical for different compatible activation patterns $s\in\mathcal{S}^C(x)$. Intuitively, for a compatible $s\in\mathcal{S}^C(x)$ if we find the critical indices $\tilde C_s(x)$ i.e. those $n_0$ neuron positions $(l,j)\in\mathcal{I}$ for which the set $\{x'\in\mathbb{R}^{n_0}|\; \tilde A_s(x')_{lj}=0\}$ is a hyperplane containing the regular vertex $x$, we can be sure that for a different compatible activation pattern $s'\in\mathcal{S}^{C}(x)$ exactly the same indices $\tilde C_{s'}(x)=\tilde C_s(x)$ will be critical. In particular in an algorithm, if we update the state parameter representing a compatible activation pattern, another state parameter representing the critical indices will still contain the correct value and is not required to be recomputed or updated to keep the state variables consistent. We will exploit this fact in Section~\ref{sec:updatingFeasibleAxes}.

Below we present a probabilistic result which guarantees that a random neural network $f$ is regular almost surely for specific distributional assumptions on its parameters.
\label{sec:regularVertices}
\begin{theorem}
	\label{thm:regularNetwork}
Assume the $\sum_{i=2}^{L+1}n_ln_{l-1}$ weight parameters of the matrices $W_1,\ldots,W_{L}$ and the $\sum_{i=2}^{L+1}n_i$ bias parameters of the vectors $b_1,\ldots,b_L$ of the neural network $f$ are sampled from a distribution such that conditionally on the weight parameters, the bias parameters are independent with a conditional marginal distribution that assigns probability zero to all finite sets. Then 
\begin{equation*}
	\mathbb{P}\left[ f \textnormal{ is regular}\right]=1,
\end{equation*}i.e. almost surely all arguments $x\in\mathbb{R}^{n_0}$ of $f$ are regular points.
\end{theorem}
The condition is satisfied for example if all parameters of the neural network are independent random variables each with a continuous distribution or if the weight matrices are deterministic and the bias parameters are independent random variables with a continuous distribution. In such cases, the above result states that all points $x\in\mathbb{R}^{n_0}$ in the input space are regular points. Our DRLSimplex algorithm we describe in Section~\ref{sec:drls} can only iterate on regular points and the above result demonstrates that this is no severe restriction.
\subsection{Feasible directions}
\label{sec:feasibleDirections}
For each input value $x\in\mathbb{R}^{n_0}$ and compatible $s\in\mathcal{S}^C(x)$ we consider directions $v\in\mathbb{R}$ such that locally $s$ is still compatible for arguments in direction $v$ from $x$ i.e. such that for some $\varepsilon>0$, $s\in\mathcal{S}^C(x+\varepsilon v)$. We call the class of such directions \emph{feasible directions}, formally defined by
\begin{equation}
\tilde{D}_s(x):=\big\{ v\in\mathbb{R}^{n_0}\mid\;\exists\varepsilon>0\quad s\in\mathcal{S}^C(x+\varepsilon v) \big\}.
\end{equation}
The following lemma provides a different formulation in terms of the subjective ReLU arguments $\tilde A_s$.
\begin{lemma}
\label{lem:feasibleFormulation}
For $x\in\mathbb{R}^{n_0}$ and $s\in\mathcal{S}^C(x)$ it holds that 
\begin{equation}
\label{eq:feasibleAlternativeDef}
\tilde{D}_s(x)=\big\{ v\in\mathbb{R}^{n_0}\mid\; \forall (l,j)\in \tilde C_s(x)\quad \tilde A_s(x+v)_{lj}(s_{lj}-0.5)\ge 0\big\}.
\end{equation}
\end{lemma}
In particular for each $x\in\mathbb{R}^{n_0}$ and $s\in\mathcal{S}^C(x)$, the feasible directions form a \emph{convex cone}.
\subsubsection{Feasible axes and region separating axes for regular vertices}
\label{sec:feasibleAxes}
In the special case where $x\in\mathbb{R}^{n_0}$ is a regular vertex, for $s\in\mathcal{S}^C(x)$ the feasible directions $\tilde D_s(x)$ are the positive span of $|C(x) |$ linear independent vectors as we will show below.
By regularity and Theorem~\ref{thm:regularity}, the critical oriented normal vectors $(\tilde u_{s,l,j})_{(l,j)\in C(x)}$ are linearly independent. In particular, when put as columns in a matrix, the resulting square matrix $A\in\mathbb{R}^{n_0\times n_0}$ has an inverse $A^{-1}$ with rows denoted by $\tilde a_{x,s,l,j}\in\mathbb{R}^{n_0}$, $(l,j)\in C(x)$ satisfying
\begin{equation}
	\label{eq:axisCondition}
	\langle \tilde a_{x,s,l,j},\tilde u_{s,l',j'}\rangle=
	\begin{cases}
		1\quad&\textnormal{if }(l,j)=(l',j')\\
		0\quad&\textnormal{if }(l,j)\not =(l',j')
	\end{cases}\quad \textnormal{ for }(l,j),(l',j')\in C(x).
\end{equation}

For convenience we define the \emph{set of feasible axes} $\tilde{\mathcal{A}}_s(x)$ by
\begin{equation}
\label{eq:feasAx}
\tilde{\mathcal{A}}_s(x):= \left\{ \tilde a_{s,x,l,j}\middle\vert\; (l,j)\in C(x) \right\}.
\end{equation}

We now show that indeed the feasible directions are the positive span of the feasible axes.
\begin{lemma}
\label{lem:positiveSpan}
Let $x\in\mathbb{R}^{n_0}$ be a regular vertex
and $s\in\mathcal{S}^C(x)$. For any $v\in\mathbb{R}^{n_0}$ it holds that $v\in \tilde D_s(x)$ if and only if there exist non-negative coefficients $\alpha_{lj}$, $(l,j)\in C(x)$ such that $v=\sum_{(l,j)\in C(x)}\alpha_{lj}\tilde{a}_{s,x,l,j}$. 
\begin{proof}
	First note that $v=\sum_{(l,j)\in C(x)}\alpha_{lj}\tilde{a}_{s,x,l,j}$ for possibly negative coefficients $\alpha_{lj}$, $(l,j)\in C(x)$ by equations~\eqref{eq:tildeKernel}, \eqref{eq:df}, Lemma~\ref{lem:sameKernel} and the assumption $\textnormal{df}^C(x)=0$. By the compatibility assumption, the stability of the critical indices Theorem~\ref{thm:regularity}, Lemma~\ref{lem:tildeV} transforms equation~\eqref{eq:feasibleAlternativeDef} into
\begin{align*}
v\in \tilde D_s(x)&\iff\forall (l,j)\in C(x)\quad\langle v,\tilde v_{s,l,j}\rangle (s_{lj}-0.5)\ge 0\\
&\iff\forall (l,j)\in C(x)\quad\langle v,\tilde u_{s,l,j}\rangle \ge 0\\
&\iff\forall (l,j)\in C(x)\quad \sum_{(l',j')\in C(x)}\alpha_{l'j'}\langle \tilde a_{s,x,l',j'},\tilde u_{s,l,j}\rangle\ge 0
\end{align*} By equation~\eqref{eq:axisCondition} the latter condition evaluates to $\alpha_{l,j}\ge 0$ for $(l,j)\in C(x)$.
\end{proof}
\end{lemma}
The following result shows that different compatible signatures share the same axes at indices with the same activations.
\begin{theorem}
	\label{thm:identicalAxes}
	Let $x\in\mathbb{R}^{n_0}$ be a regular vertex, $(l,j)\in C(x)$ and $s,s'\in\mathcal{S}^C(x)$ with $s_{lj}=s'_{lj}\in\left\{ 0,1 \right\}$. It holds that 
		$\tilde a_{s,x,l,j}=\tilde a_{s',x,l,j}$.
\end{theorem}

The above theorem is an important result that we will use in Section~\ref{sec:updatingFeasibleAxes} to compute the axes $\tilde{\mathcal{A}}_{s'}(x)$ for a new activation pattern $s'\in\mathcal{S}^C(x)$ when we already know the axes $\tilde{\mathcal{A}}_{s}(x)$ for another activation pattern $s\in\mathcal{S}^C(x)$ that only differs at one specific neuron position $(l^*,j^*)\in\mathcal{I}$ from $s'$. By the above result, in this case all but one axis are the same. 
Furthermore, the theorem allows us to define the set 
\begin{equation}
	\mathcal{A}(x)=\left\{ a_{r,x,l,j}\in\mathbb{R}^{n_0}\mid\;r\in\left\{ 0,1 \right\}, (l,j)\in C(x) \right\}.
\end{equation}
of $2n_0$ different \emph{region separating axes} satisfying
\begin{equation}
	\label{eq:feasibleRegionseparating}
	a_{x,s_{lj},l,j}=\tilde a_{x,s,l,j}	\quad \textnormal{ for }(l,j)\in C(x), s\in\mathcal{S}^C(x).
\end{equation}
Note that the feasible axes $\tilde{\mathcal{A}}_s(x)$ are sets of $n_0$ vectors specific to $s\in\mathcal{S}^{C}(x)$ whereas the region separating axes $\mathcal{A}(x)$ are a set of $2n_{0}$ vectors not specific to some activation. 
\begin{example}
\label{ex:axes}
Let $L=2$, $(n_0,n_1,n_2,n_3)=(2,2,1,1)$, $b_1=(0,0)^\intercal$, 
$W_1=
{
\tiny
\begin{pmatrix}
1&0\\
0&1
\end{pmatrix}
}$, $W_2=(1,-1)$, $b_2=-1$, $W_3=(1)$, $b_3=0$. In this case it holds that 
\begin{align}
	h_1(x_1,x_2)&= (\textnormal{ReLU}(x_1),\textnormal{ReLU}(x_2)),\\
A(x)_{2,1}&= \textnormal{ReLU}(x_1)-\textnormal{ReLU}(x_2)-1.
\end{align}
In particular $A(x)_{2,1}=0$ if and only if either $x_2<0$ and $x_1=1$ or $x_2\ge 0$ and $x_1=1+x_2$. This is the reason for the kink in Figure~\ref{fig:axes} at $x^*=(1,0)$. The corresponding critical indices are $C(x^*)=\left\{ (1,2),(2,1) \right\}$ and the region separating axes are given by $a_{x^*,0,1,2}=(0,-1)$, $a_{x^*,1,1,2}=(1,1)$, $a_{x^*,0,2,1}=(-1,0)$, $a_{x^*,1,2,1}=(1,0)$.
\end{example}
\begin{figure}[htpb]
\centering
\includegraphics[width=0.6\linewidth]{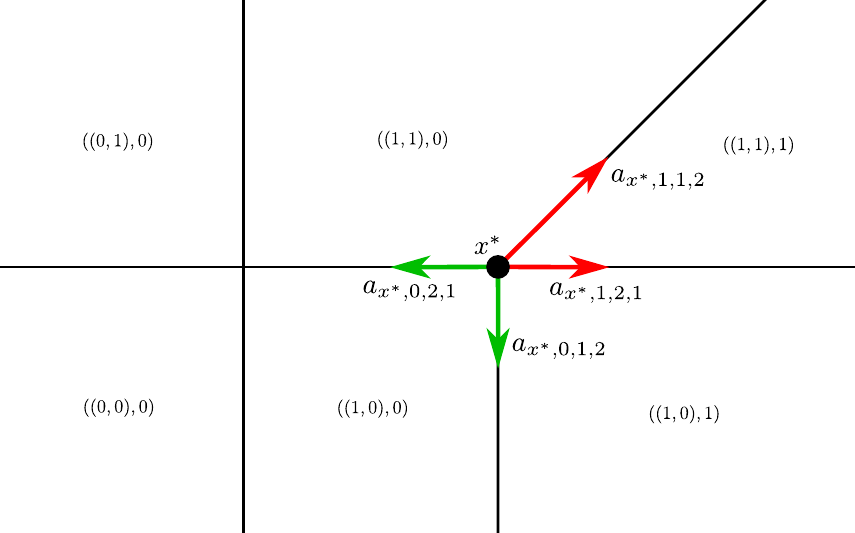}
\caption{Illustration of the region separating axes. The black lines represent the sets $A(x_1,x_2)_{lj}=0$ for $(l,j)\in\left\{ (1,1),(1,2),(2,1) \right\}$. The tuples indicate the activation pattern of affine regions of the neural network. The arrows are the region separating axes at position $x^*=(1,0)$. 
	The red region separating axes represent the feasible axes $\tilde{\mathcal{A}}_s(x^*)$ for the activation pattern $s=((1,1),1)$ whereas the green axes correspond to ${\tilde{\mathcal{A}}}_{ s'}(x^*)$ for $s'=( (1,0),0)$.
}
\label{fig:axes}
\end{figure}

\subsection{Local minima in regular networks}
Gradient descent algorithms are usually based on a computed gradient at the current position in the argument space $\mathbb{R}^{n_0}$. However our considered function $f$ is piece-wise affine and therefore the gradient is constant on each such affine region and not well-defined at the boundaries of these regions. It solely depends on which neurons are active such that we define the gradient specific to the activation pattern $s\in\mathcal{S}$ by 
\begin{equation}
\label{eq:gradient}
\nabla_s=\left(W_{L+1}\textnormal{diag}(s_L)W_L\cdots\textnormal{diag}(s_1)W_1  \right)^\intercal.
\end{equation}
With this definition it holds that
\begin{equation}
\tilde f_s(x)= \langle x,\nabla_s\rangle+\sum_{l=1}^{L+1}W_{L+1}\textnormal{diag}(s_L)\cdots W_{l+1}\textnormal{diag}(s_l)b_l
\end{equation}for $x\in\mathbb{R}^{n_0}$ and a compatible activation pattern $s\in\mathcal{S}^C(x)$ by equation~\eqref{eq:tildef}. In particular by Corollary~\ref{cor:representation} it follows that $\nabla_s$ is indeed equal to the gradient $\nabla f(x)$ at values $x\in\mathbb{R}^{n_0}$ where the gradient is well-defined and the activation pattern is $S(x)=s$.
Note that $\nabla_s$ depends on the weight matrices $W_1,\ldots,W_{L+1}$ but not on the bias vectors $b_1,\ldots,b_{L+1}$ of the neural network $f$.
Below we present a sufficient and necessary condition when a regular vertex is a local minimum of $f$
\begin{proposition}
\label{prop:localMinimum}
Let $x\in\mathbb{R}^{n_0}$ be a regular vertex. Then 
\begin{align*}
&\exists \varepsilon>0: f(x)=\inf\left\{ f(y)\middle\vert\; y\in B_\varepsilon(x) \right\}\\
\iff& \forall a\in\mathcal{A}(x)\;\exists s\in\mathcal{S}^C(x), (l,j)\in \tilde C_s(x) \textnormal{ with }a=\tilde a_{x,s,l,j}\textnormal{ and }\langle \tilde a_{x,s,l,j},\nabla_s\rangle\ge 0.
\end{align*}
\end{proposition}
Intuitively the above result states that a regular vertex $x\in\mathbb{R}^{n_0}$ is a local minimum of $f$ if and only if for any region separating axis $a\in\mathcal{A}(x)$, $f$ is non-decreasing locally in the direction $a$. Furthermore, for every axis $a_{x,r,l,j}$ for $r\in\left\{ 0,1 \right\}, (l,j)\in C(x)$ the test whether $f$ is non-decreasing locally in direction of the axis $a_{x,r,l,j}$ can be carried out by checking for $\langle \tilde a_{x,s,l,j},\nabla_s\rangle\ge0$ where $s\in\mathcal{S}^C(x)$ with $s_{lj}=r$ by equation~\eqref{eq:feasibleRegionseparating}. In particular, this result allows to check for a local minimum using $2n_0$ inner products involving subjective quantities specific to compatible activation patterns.  

We will use Proposition~\ref{prop:localMinimum} in the stopping criterion for our algorithm presented in Section~\ref{sec:drlp}. For regular networks, this condition will exactly identify vertices that are local minima. Note that the piece-wise affinity of the function $f$ implies that only vertices are interesting candidates for local minima. 
\begin{lemma}
\label{lem:minimumVertex}
Every strict local minimum of $f$ is a vertex and for every non-strict local minimum $x\in\mathbb{R}^{n_0}$ there exist a vertex $x'\in\mathbb{R}^{n_0}$ with $f(x)=f(x')$.
\end{lemma}
\subsection{Discussion}
We have analyzed the affine linear structure of ReLU feed-forward neural networks. Our main contribution here are the introduction of precise mathematical quantities and the analysis how these quantities can be obtained from a subjective view on the neural network induced by explicitly specified neuron activations. More precisely we distinguish between objective and subjective versions of such quantities at some input $x$ where the subjective version is specific to an activation pattern compatible with $x$ i.e. where neurons are considered active/inactive when the input of their ReLU activation function is non-negative/non-positive respectively. In case of equality, we call the corresponding quantity \emph{stable}.

At a given input $x\in\mathbb{R}^{n_0}$ such quantities with an objective and a subjective version are for example the critical indices corresponding to those neurons that locally change their activity in every small neighbourhood around $x$, and the critical kernel, i.e. the linear subspace of directions in the input space $\mathbb{R}^{n_0}$ that do not change the argument of the critical indices' ReLU activation functions.
Regularity of a point means that the subjective critical normal vectors are linearly independent for all compatible activation patterns. By our Theorem~\ref{thm:regularity} this automatically guarantees that the objective critical indices and all the subjective the critical indices specific are identical. In particular, the critical indices are stable. Figure~\ref{fig:impossible} depicts two vertex constellations that are impossible as a consequence. This makes the critical indices a general property obtainable from a subjective view on the neural network for a fixed compatible activation pattern.  Note that this property is essential because it shows that the subjective critical indices are unaffected by a compatible change of the activation pattern which allows algorithms to keep track of the subjective critical indices as a state parameter which does not need to be updated for such changes. We will exploit this fact in our Deep ReLU Simplex algorithm explained in Section~\ref{sec:drls}.
\begin{figure}[htpb]
\centering
\includegraphics[width=0.6\linewidth]{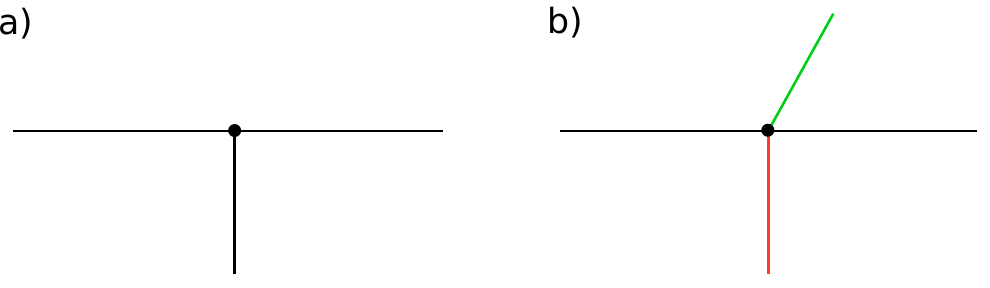}
\caption{
Impossible vertex constellations by the third and fourth row in Table~\ref{tab:stability}. In a) the black dot is not a vertex subjectively seen from the upper adjacent region, in b) the red and green segments separate adjacent affine regions but are meant to correspond to zero sets of different neurons. This is impossible because for regular points, always the same subjective neurons are critical for any compatible activation pattern\protect\footnotemark. 
}
\label{fig:impossible}
\end{figure}
\footnotetext{For non-regular points, the auxiliary Lemma~\ref{lem:collinear} shows that a neuron that is only subjectively critical for some activation patterns, has a normal vector collinear to the normal vector of another neuron such that b) in Figure~\ref{fig:impossible} is still impossible}
\begin{table}[htpb]
	\centering
	\begin{tabular}{|l|l|l|}
		\hline
		Quantity & Stability & Reference\\
		\hline
		ReLU arguments in neurons & $\tilde A_s(x)=A(x)$&Theorem~\ref{thm:argumentRepresentation}\\
		Value of objective function & $\tilde f_s(x)= f(x)$&Corollary~\ref{cor:representation}\\
		Critical kernel &$\tilde{\textnormal{Ker}}^C_s(x)=\textnormal{Ker}^C(x)$ & Lemma~\ref{lem:sameKernel}\\
		Critical indices & $\tilde C_s(x)=C(x)$ for regular $x$&Theorem~\ref{thm:regularity}\\
		\hline
	\end{tabular}
	\caption{Stability of certain concepts for $x\in\mathbb{R}^{n_0}$ and compatible $s\in \mathcal{S}^C(x)$}
	\label{tab:stability}
\end{table}

We have summarized our stability results of several quantities derived in the previous sections in Table~\ref{tab:stability}. These quantities are  specifically interesting for a rigorous analysis and construction of linear programming and active-set related algorithms applied to functions with the structure of a ReLU feed-forward neural network and their stability justifies their deduction based on a subjective view on the network's hyperplanes induced by a fixed imposed activation pattern of the network's ReLU neurons which can thus serve as a state parameter in the corresponding algorithm. In this context, regularity is relevant which we showed is almost surely satisfied for continuous random network parameters and is therefore not a strong requirement. Just as intuition suggests, regular vertices in $\mathbb{R}^{n_0}$ have $2^{n_0}$ adjacent affine regions separated by $2n_0$ axes that ensure a local minimum if their directed derivatives of the objective neural network $f$ are all non-negative.
\section{Algorithmic Primitives for Deep ReLU Programming}
\label{sec:algPrimitives}
In this section we will describe how solver algorithms for deep ReLU programming problems can be implemented. We first describe basic algorithmic routines in Section~\ref{sec:algPrimitives} that will be used as building blocks in Section~\ref{sec:drls} where we describe our DRLSimplex algorithm in detail and provide pseudo-code. With these routines we want to demonstrate how certain quantities can be computed but do not focus on numerical stability.
\subsection{Computation of the gradient}
\label{sec:computationGradient}
For every activation pattern $s\in\mathcal{S}$ we can compute the corresponding gradient $\nabla_s\in\mathbb{R}^{n_0}$ in computational complexity order $\mathcal{O}(\sum_{l=1}^{L}n_ln_{l-1})$ by using equation~\eqref{eq:gradient}.
We will denote the function that computes this gradient by $\textsc{Gradient}$.
\subsection{Computation of an oriented normal vector}
\label{sec:computationNormal}
For an activation pattern $s\in\mathcal{S}$ and a neuron position index $(l,j)\in\mathcal{I}$ the complexity for the computation of the normal vector $\tilde v_{s,l,j}$ is of order $\mathcal{O}(n_0n_1+\ldots+n_{l-1}n_l)$. This follows from equation~\eqref{eq:v} if we compute the matrix product from the left side starting with the $j$-th row of $W_l$. The same holds for the oriented normal vector $\tilde u_{s,l,j}$ by equation~\eqref{eq:orientedNormal}. In our pseudo-code we will denote this procedure by $\textsc{OrientedNormalVec}$ taking as parameters the activation pattern $s$ and the neuron position $(l,j)$.
\subsection{Computing inner products with all oriented normal vectors}
\label{sec:multipleInner}
Assume we have an activation pattern $s\in\mathcal{S}$ and a vector $w\in\mathbb{R}^{n_0}$. We want to compute the inner products $\langle w,\tilde v_{s,l,j}\rangle$ for all $(l,j)\in\mathcal{I}$. Then each of these inner products can be computed as explained in Section~\ref{sec:computationNormal} in $\mathcal{O}(\sum_{\ell=1}^{l}(n_{\ell}n_{\ell-1}))$. However there exists a more efficient way to compute all such inner products at once in the same complexity order. From equation~\eqref{eq:v} it follows that 
\begin{equation}
	\label{eq:normalvectors}
	\langle w,\tilde v_{s,l,j} \rangle=\left( W_l\textnormal{diag}(s_{l-1})W_{l-1}\cdots \textnormal{diag}(s_1)W_1 w \right)_{j}\quad\textnormal{ for }(l,j)\in\mathcal{I}.
\end{equation}In particular the desired inner products can be computed by pushing the vector $w$ through the network with fixed activation pattern $s$ but with bias terms set to zero and observing the computed values at all neuron positions $\mathcal{I}$. More precisely, we successively compute the above matrix vector product from the right side as $u_1:= W_1w$, $u_{l+1}:=W_{l+1}\textnormal{diag}(s_l)u_l$ for $l\in\left\{ 1,\ldots,L-1 \right\}$ and read the corresponding entries $(u_{l})_{j}$ for $(l,j)\in\mathcal{I}$. According to equation~\eqref{eq:orientedNormal}, depending on the activation $s_{lj}$ the sign in equation~\eqref{eq:normalvectors} needs to be flipped to obtain the inner products with the oriented normal vectors. Since this change of the sign, does not increase the overall complexity order, the computation of all inner products is of order $\mathcal{O}(\sum_{l=1}^{L}n_ln_{l-1})$.
\subsection{Projection using the pseudoinverse of some oriented normal vectors}
\label{sec:project}
Assume we have a number $m\le n_0$ of distinct neuron positions $(l_1,j_1),\ldots,(l_m,j_m)\in\mathcal{I}$, an activation pattern $s\in\mathcal{S}$ and the pseudoinverse $A^+\in\mathbb{R}^{m\times n_0}$ of the matrix whose linearly independent columns are the oriented normal vectors $\tilde u_{s,l_1,j_1},\ldots,\tilde u_{s,l_{m},j_m}$. We can project a vector $v\in\mathbb{R}^{n_0}$ onto the linear span of $\tilde u_{s,l_1,j_1},\ldots,\tilde u_{s,l_{m},j_m}$ by computing $(A^+)^{\intercal} w$ where $w=(\langle v, \tilde u_{s,l_1,j_1},\rangle,\ldots,\langle v,\tilde u_{s,l_{m},j_m}\rangle)\in\mathbb{R}^{n_0} $. By Section~\ref{sec:multipleInner} $w$ can be computed in $\mathcal{O}(\sum_{l=1}^L n_ln_{l-1})$ such that the computational complexity for the computation of this projection is of order $\mathcal{O}(\sum_{l=1}^L n_ln_{l-1}+n_0^2)$. We will denote this procedure by $\textsc{Project}(\texttt{A}^+,\texttt{C},\texttt{s},\texttt{v})$ where $\texttt{A}^+$ is the pseudoinverse matrix, $\texttt{C}$ is the list of neuron positions, $\texttt{s}$ is the activation pattern and $\texttt{v}$ is the vector to project.
\subsection{Updating the pseudoinverse for an additional column}
\label{sec:addColumn}
Given $m\in \left\{ 0,\ldots,n_0-1 \right\}$, a matrix $A\in\mathbb{R}^{n_0\times m}$ with $\textnormal{rank}(A)=m$ and its pseudoinverse $A^+=(A^\intercal A)^{-1}A^{\intercal}$ we can add a given additional linear independent column $w\in\mathbb{R}^{n_0}$ to $A$ and compute the corresponding pseudoinverse in computational complexity $\mathcal{O}(mn_0)$. This can be done by first projecting $w$ onto the column space of $A$ by computing $w_{\|}=(A^+)^\intercal A^\intercal w$ and taking the orthogonal part $w_{\perp}=w-w_{\|}$. If we denote by $\tilde A\in\mathbb{R}^{n_0\times (m+1)}$ the matrix that has the same columns as $A$ and an additional $(m+1)$-th column equal to $w$ then its pseudoinverse $\tilde A^+$ can be constructed as follows. The $(m+1)$-th row of $\tilde A^+$ is equal to $\langle w_{\perp},w\rangle^{-1} w_\perp$. Furthermore, if we denote the $i$-th row of $A^+$ by $a^+_i$ then the $i$-th row of $\tilde A^+$ is equal to $a^{+}_i-\langle a^+_i,w\rangle\langle w_{\perp},w\rangle^{-1} w_\perp$ for $i\in\left\{ 1,\ldots,m \right\}$. One easily checks that indeed $\tilde A^+ \tilde A=\textnormal{Id}_{m+1}$ and that the column space of $\tilde A$ and the row space of $\tilde A^+$ are equal. Obviously, given $A$, $A^+$ and $w$, this construction of $\tilde A^+$ is of order $\mathcal{O}(mn_0)$. 

In our pseudo-code we will use a procedure $\textsc{AddPseudorow}$ that takes the following parameters. The first parameter $\texttt{A}^+$ is the pseudoinverse of the matrix $A$ with the linearly independent oriented normal vectors $\tilde u_{\texttt{s},l_1,j_1},\ldots,\tilde u_{\texttt{s},l_{|\texttt{C}|},j_{|\texttt{C}|}}$ as columns where the ordered collection $\texttt{C}=\left( (l_1,j_1,),\ldots,(l_{|\texttt{C}|},j_{|\texttt{C}|})\right)$ of neuron indices and the activation pattern $\texttt{s}\in\mathcal{S}$ are the second and third parameters. The last parameter $\texttt{u}$ is the additional column that shall be added to $A$. The return value is the updated pseudoinverse $\tilde A^+$ computed as described above. Note that we can use the \textsc{Project} procedure from Section~\ref{sec:project} to compute $w_{\|}$ in order $\mathcal{O}(\sum_{l=1}^Ln_ln_{l-1}+n_0^2)$. The rest of the construction is also covered by this computational complexity order.

Upon this procedure we build the $\textsc{AddAxis}$ procedure below in Algorithm~\ref{alg:addAxis} which, instead of an arbitrary an additional column $\texttt{u}$ uses the oriented normal vector at a specified the neuron position $\texttt{c}$. From the pseudo-code and the above discussion it is clear that the whole $\textsc{AddAxis}$ procedure is of computational complexity order $\mathcal{O}(\sum_{l=1}^Ln_ln_{l-1}+n_0^2)$.
\begin{algorithm}
\caption{Add Axis}
\label{alg:addAxis}
\begin{algorithmic}[1]
	\Procedure{AddAxis}{$\texttt{A}^+$: matrix, $\texttt{s}$: activation pattern , $\texttt{c}$: neuron position}
	\State $\texttt{u}\gets  \Call{OrientedNormalVec}{\texttt{s},\texttt{c}}$
	\State $\texttt{A}^+\gets  \Call{AddPseudorow}{\texttt{A}^+,\texttt{C},\texttt{s},\texttt{u}}$
	\State \Return $\texttt{A}^+$
\EndProcedure
\end{algorithmic}
\end{algorithm}
\subsection{Updating the pseudoinverse for an omitted column}
\label{sec:dropColumn}
Given $n\in \left\{ 1,\ldots,n_0 \right\}$, a column index $i\in\left\{ 1,\ldots,n \right\}$, a matrix $A\in\mathbb{R}^{n_0\times n}$ with $\textnormal{rank}(A)=n$ and its pseudoinverse $A^+=(A^\intercal A)^{-1}A^{\intercal}$ we can omit the $i$-th column of $A$ and compute the corresponding pseudoinverse in computational complexity $\mathcal{O}(nn_0)$. If we denote the rows of $A^+$ by $a^+_1,\ldots,a^+_n$ and the matrix that is obtained by omitting the $i$-th column in $A$ by $\tilde A\in\mathbb{R}^{(n-1)\times n_0}$ then for $k\in\left\{ 1,\ldots,n-1 \right\}$ the $k$-th row of its pseudoinverse $\tilde A^+\in \mathbb{R}^{(n-1) \times n_0}$ is equal to $a^+_{j(k)}-\langle a^+_i,a^{+}_{j(k)}\rangle \langle a^+_i,a^+_i\rangle^{-1} a^+_i$ where $ j(k)=k$ if $k<i$, $j(k)=k+1$ if $k\ge i$. Note that this construction of $\tilde A^+$ is of complexity order $\mathcal{O}(n n_0)$ and does not involve $A$ but only its pseudoinverse $A^+$. Again one easily verifies that indeed $\tilde A^+ A=\textnormal{Id}_{n-1}$ and that the column space of $\tilde A$ and the row space of $\tilde A^+$ are equal. 

In our pseudo-code we will denote this procedure by $\textsc{RemovePseudorow}$ taking as parameters a matrix $\texttt{A}^+$ and an index $\texttt{i}$. It returns the modified pseudoinverse as described above where $\texttt{A}^+$ and $\texttt{i}$ take the roles of the pseudoinverse $A^+$ and $i$ respectively.

\subsection{Flip one activation}
\label{sec:flip}
For an activation pattern $s\in\mathcal{S}$ and a neuron index $(l^*,j^*)\in\mathcal{I}$ we define the function that flips the activation of $s$ at position $(l^*,j^*)$ by
$ \textsc{Flip} :\mathcal{S}\times \mathcal{I}\to\mathcal{S}$ with 
\begin{equation}
	\label{eq:flip}
	\textsc{Flip}(s,(l^*,j^*))_{lj}=
	\begin{cases}
		s_{lj}\quad&\textnormal{if } (l,j)\neq(l^*,j^*)\\
		1-s_{lj}\quad&\textnormal{if } (l,j)=(l^*,j^*)
	\end{cases}.
	\quad 
\end{equation}
In an implementation with constant random access time, this function has constant runtime $\mathcal{O}(1)$ as it accesses and changes only one specific bit. 
\subsection{Updating the feasible axes when changing one activation}
\label{sec:updatingFeasibleAxes}
We now explain the essential ingredient that allows the change of the affine region in a computationally efficient way. It exploits the fact that flipping the activation pattern at one specific position only causes one feasible axis to change as depicted in Figure~\ref{fig:adjacentRegionAxes}.

Assume that $x$ is a regular vertex, i.e. $|C(x)|=n_0$ and let $s\in\mathcal{S}^C(x)$ be a compatible activation pattern. Applying Theorem~\ref{thm:regularity} we identify the subjective critical indices $\tilde C_s(x)=C(x)=\left\{(l_1,j_1),\dots,(l_{n_0},j_{n_0})\right\}$. Let $A_s\in\mathbb{R}^{n_0\times n_0}$ be the matrix with columns $\tilde u_{s,l_1,j_1},\dots,\tilde u_{s,l_{n_0},j_{n_0}}$ and denote its pseudoinverse by $A^+_s=(A_s^{\intercal}A_s)^{-1}A_s^{\intercal}$. Now we change the activation at neuron index $(l_{n_0},j_{n_0})$, i.e. let $\tilde s =\textsc{Flip}(s,(l_{n_0},j_{n_0}))$ as defined in equation~\eqref{eq:flip}. We want to find the pseudoinverse $A^+_{\tilde s}\in\mathbb{R}^{n_0\times n_0}$ of the matrix $A_{\tilde s}\in\mathbb{R}^{n_0\times n_0}$ with columns $\tilde u_{\tilde s,l_1,j_1},\dots,\tilde u_{\tilde s,l_{n_0},j_{n_0}}$. The key insight here is the fact that the first $n_0-1$ rows of $A_{\tilde s}^+$ and $A_s^+$ are equal. Since $A_{\tilde s}^+$ is the pseudoinverse of $A_{\tilde s}$, the unknown axis $a^+_{\tilde s,_{n_0}}\in\mathbb{R}^{n_0}$ in the last row of $A_{\tilde s}^+$ satisfies the following three conditions. Firstly, it is an element of the linear span of $\tilde u_{\tilde ,l_1,j_1s},\dots,\tilde u_{\tilde s,l_{n_0},j_{n_0}}$, secondly it is orthogonal to $\tilde{u}_{\tilde s,l_1,j_1s},\dots,\tilde u_{\tilde s,l_{n_0-1},j_{n_0-1}}$ and finally it satisfies $\langle a^{+}_{\tilde s,n_0},\tilde u_{\tilde s,l_{n_0},j_{n_0}}\rangle=1$. These three conditions uniquely determine $a^+_{\tilde s,_{n_0}}\in\mathbb{R}^{n_0}$ and suggest the following computation. First compute 
\begin{equation}
	\label{eq:additionalAxis}
	w:=\tilde u_{\tilde s,l_{n_0},j_{n_0}}-\sum_{i=1}^{n_0-1}\langle \tilde u_{\tilde s,l_i,n_i},\tilde u_{\tilde s,l_{n_0},j_{n_0}} \rangle a^+_{\tilde s,l_{n_i},j_{n_i}}	\in\mathbb{R}^{n_0},
\end{equation}which is then used to compute $a^{+}_{\tilde s,n_0}={\langle w,u_{\tilde s,l_{n_0},j_{n_0}}\rangle}^{-1}w$. One easily checks that indeed the three conditions are satisfied. Since $\tilde u_{\tilde s,l_{n_0},j_{n_0}}$ and the inner products can be computed in $\mathcal{O}(\sum_{l=1}^{L}n_ln_{l-1})$ as explained in Sections~\ref{sec:computationNormal} and \ref{sec:multipleInner} the overall computational effort in equation~\eqref{eq:additionalAxis} is of order $\mathcal{O}(\sum_{i=1}^Ln_in_{i-1}+n_0^2)$ because we need to subtract $n_0-1$ vectors with $n_0$ entries each. In turn this implies that $a^+_{\tilde s,n_0}$ and hence the matrix $A^+_{\tilde s}$ can be constructed in $\mathcal{O}(\sum_{i=1}^Ln_in_{i-1}+n_0^2)$ if we already know $A^+_s$.
Above we assumed for simplicity that the axis that needs to be recomputed corresponds to the last row of the matrix $A^+_{\tilde s}$. We will denote the above procedure by $\textsc{UpdateAxisNewRegion}$ and allow for an arbitrary row index. Its arguments $(\texttt{A}^+,\texttt{i},\texttt{s},\texttt{c})$ specify the matrix $\texttt{A}^+$ whose $\texttt{i}$-th row needs to be updated when the activation pattern $\texttt{s}$ is flipped at neuron index $\texttt{c}\in\mathcal{I}$, where $\texttt{c}$ corresponds to $(l_{n_0},j_{n_0})$ above. This procedure returns the updated matrix $A^+_{\tilde s}$.

We want to stress again that this algorithmic primitive is the important building block which allows to change the considered affine region. It is therefore the main ingredient that distinguishes our algorithm presented in Section~\ref{sec:description} from linear programming algorithms where the considered affine region is the fixed feasible region. The fact that this change of the considered region is computationally not more expensive than the other algorithmic differences as listed in Table~\ref{tab:complexity}, makes deep ReLU programming interesting for practical problems.
\begin{figure}[htpb]
\centering
\includegraphics[width=0.5\linewidth]{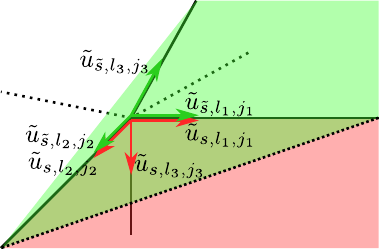}
\caption{
	Illustration of the concept for the algorithmic primitive in Section~\ref{sec:updatingFeasibleAxes} for $n_0=3$. If we know the feasible axes 
	$\tilde u_{s,l_1,j_1}, \tilde u_{s,l_2,j_2},\tilde u_{s,l_3,u_3}$ of the red region with activation pattern $s\in\mathcal{S}$ and want to compute the feasible axes  
	$\tilde u_{\tilde s,l_1,j_1}, \tilde u_{\tilde s,l_2,j_2},\tilde u_{\tilde s,l_3,u_3}$ of the green region 
	with activation pattern $\tilde s$ then we can reuse the $n_0-1=2$ axes $\tilde u_{\tilde s,l_1,j_1}=\tilde u_{s,l_1,j_1}$
	and $\tilde u_{\tilde s,l_2,j_2}=\tilde u_{s,l_2,j_2}$, only one axis $\tilde u _{\tilde s,l_3,j_3}$ needs to be recomputed, see Theorem~\ref{thm:identicalAxes}.
}
\label{fig:adjacentRegionAxes}
\end{figure}
\subsection{Advance maximally without leaving the current affine region}
\label{sec:advanceMax}
Assume we have an argument $x\in\mathbb{R}^{n_0}$ and we want to advance into a direction $v\in\mathbb{R}^{n_0}$ maximally such that the hyperplane pattern does not change. More precisely we want to find 
\begin{equation}
	\label{eq:tstar}
	t^*=\sup\left\{ t\ge0\mid\; H(x+tv)=H(x) \right\}
\end{equation}
 as illustrated in Figure~\ref{fig:11}.\begin{figure}[htpb]
\centering
\includegraphics[width=0.6\linewidth]{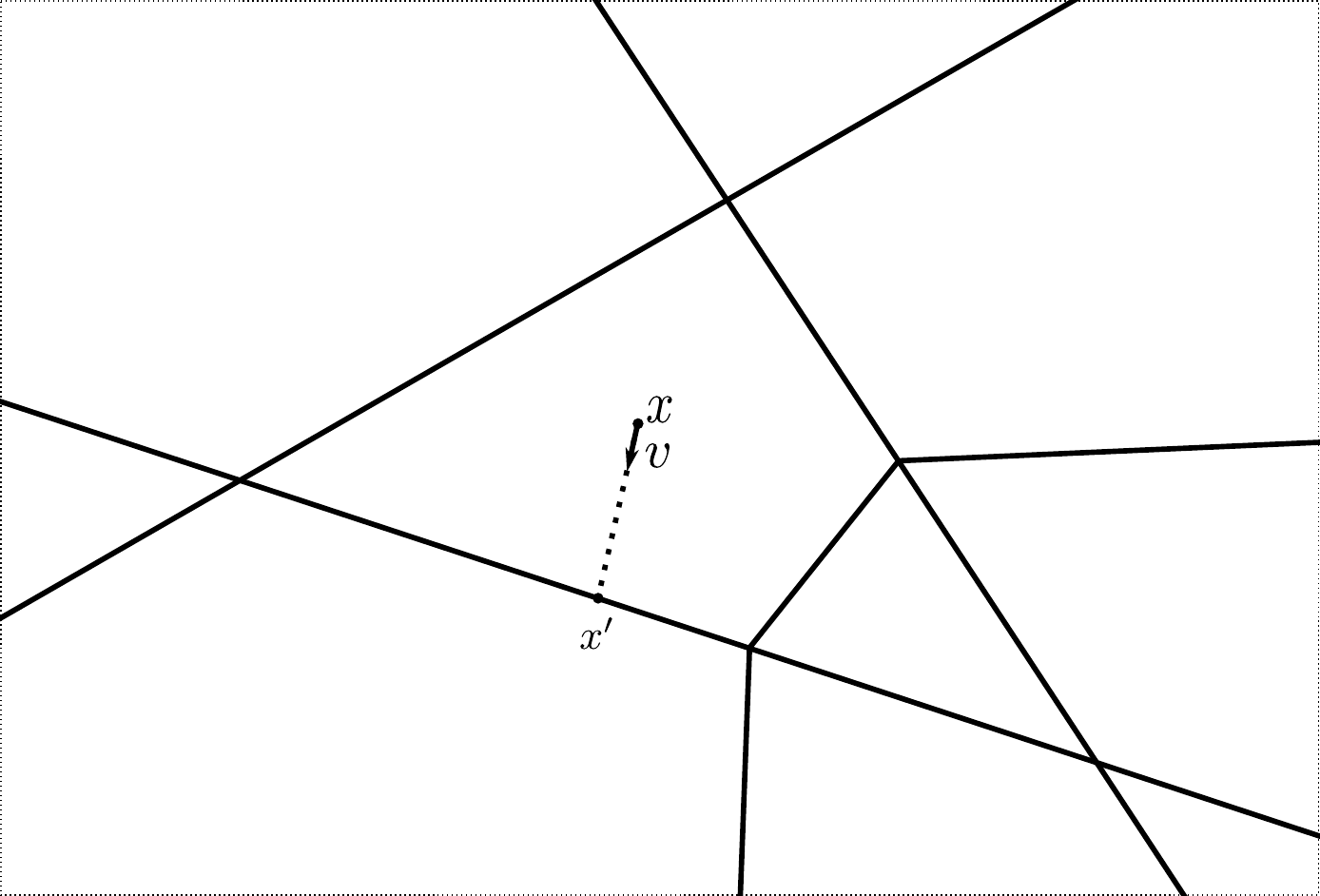}
\caption{Illustration of the setting in ~\ref{sec:advanceMax} for $n_0=2$. Given a starting point $x\in\mathbb{R}^{n_0}$, a compatible activation pattern $s\in\mathcal{S}^C(x)$ and a direction $v\in\mathbb{R}^{n_0}$ we consider the ray $x+tv$, $t\ge 0$ and find $t$ maximally such that the affine region corresponding to $s$ is not left. In the figure, $x'=x+t^*v$.}
\label{fig:11}
\end{figure}
\begin{figure}[htpb]
\centering
\includegraphics[width=\linewidth]{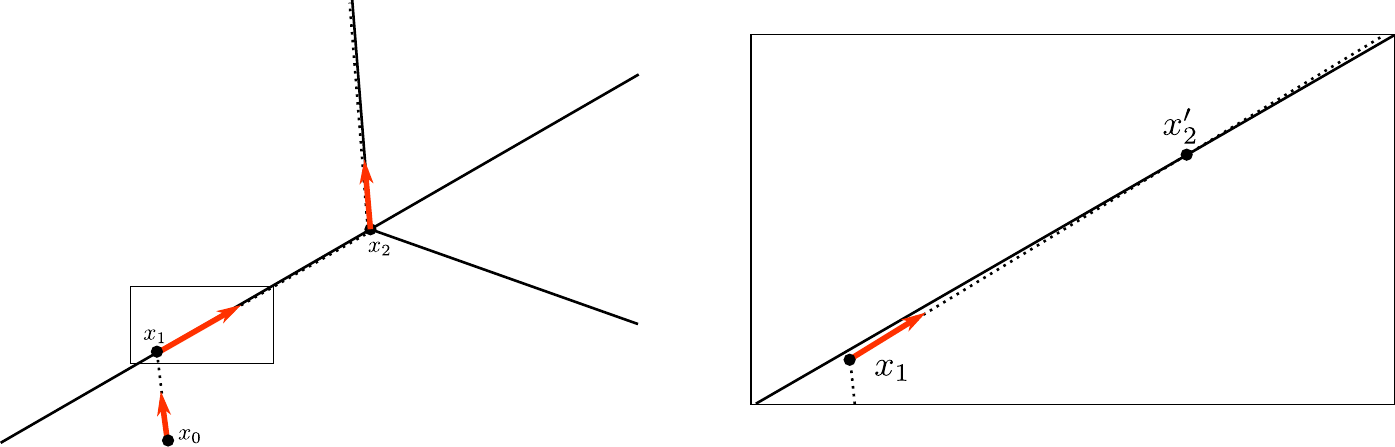}
\caption{
Illustration of a problem caused by numerical uncertainty for a two-dimensional input $n_0=2$. At the position $x_1$ we want to advance in the direction parallel to the boundary of two affine regions until we encounter a change of the affine behaviour at position $x_2$. However, when using finite precision arithmetic, it can happen that $x_1$ itself is not exactly on the boundary and that the direction we take for the next step is not parallel to the boundary line as seen on the magnification of the of the corresponding rectangle section on the left. This causes the dotted line to intersect with the region boundary line such that the next position would be $x_2'$. In order to avoid this artifact, we need to keep book about which region boundaries we are assuming to stay on and ignore activation changes of the corresponding neurons. }
\label{fig:ignore}
\end{figure}
The value $t^*$ above can be computed efficiently as follows. Let $x_v(t)=x+tv$ for $t\in\mathbb{R}$. If we use the notation from equation~\eqref{eq:f} and define $\alpha_1=W_1x+b_1$ and $\beta_1=W_{1}v$ then $g^{(1)}(x_v(t))=\textnormal{ReLU}.(\alpha_1+t\beta_1)$ such that 
\begin{equation}
	\label{eq:t1}
t_1:=\min\left\{ -\frac{(\alpha_1)_j}{(\beta_1)_j}\;\middle|\; j\in\left\{ 1,\dots,n_1 \right\}\land (\beta_1)_j\neq0\land -\frac{(\alpha_1)_j}{(\beta_1)_j}>0 \right\}
\end{equation}
is the infimum $t_1=\inf\left\{ t\ge0\;| \exists\; j\in\left\{ 1,\ldots,n_1 \right\}:\;H(x)_{1,j}\not=H(x_v(t))_{1,j} \right\}$ whenever the set in equation~\eqref{eq:t1} is not empty. Similarly, if we set $s = S(x)$ and recursively define $\alpha_{l+1}=W_{l+1}\textnormal{diag}(s_l)\alpha_l$, $\beta_{l+1}=W_{l+1}\textnormal{diag}(s_l)\beta_l$ for $l\in\left\{ 1,\ldots,L-1 \right\}$, then 
\begin{equation}
	\label{eq:tl}
t_l:=\min
\left\{ 
-\frac{(\alpha_l)_j}{(\beta_l)_j}
\middle\vert\;
j\in\left\{ 1,\dots,n_1 \right\}
\land
(\beta_l)_j\neq0\land
-\frac{(\alpha_l)_j}{(\beta_l)_j}>0
\right\}
\end{equation}
satisfies $t_l =\inf\left\{ t>0\;| \exists\;j\in\left\{ 1,\ldots,n_{l} \right\}:\;H(x)_{l,j}\not=H(x_v(t))_{l,j} \right\}$
for $l\in\left\{ 1,\dots,L \right\}$. In particular it holds that $t^*=\min(t_1,\dots,t_L)$ if at least one of the sets in equations~\eqref{eq:t1} and \eqref{eq:tl} is not empty, otherwise $t^*=\infty$.

This basic idea needs to be extended in some ways to fit our needs. Specifically, we need the following modifications:
\begin{enumerate}
	\item \underline{Explicit specification of activation pattern}: Since we will track the activation pattern in a dedicated state variable in our DRLSimplex algorithm in Section~\ref{sec:description}, instead of using $s=S(x)$ as above, the activation pattern $s$ to be used in the above construction will be explicitly specified as a parameter which is assumed to be compatible with $x$, i.e. to be an element of $\mathcal{S}^C(x)$.
	\item \underline{Information about the corresponding neuron position}: In our implementation below the neuron index at which the activation pattern changes at the point $x_v(t^*)$ is needed. In other words, we also require the argmin neuron position to be returned.
	\item \underline{Controllable insensitivity}: We need to specify some neuron positions where the sign change will be ignored, they shall not be considered in the construction $t^*$. This is necessary to avoid numerical problems illustrated in Figure~\ref{fig:ignore}.
	\item \underline{Mismatch robustness}: Due to numerical uncertainty, it may happen that our specified activation pattern is not compatible with the hyperplane pattern at position $x$, i.e. $s\not\in\mathcal{S}^C(x)$. This means that $x$ is outside the affine region of the objective function $f$ corresponding to $s$. In such a case it makes sense to ignore a sign change or to allow for negative values for $t^*$, see Figure~\ref{fig:mismatchRobustness}. This is achieved in line 10 of Algorithm~\ref{alg:advMaxAdjusted}.
\end{enumerate}
In Algorithm~\ref{alg:advMaxAdjusted} we provide pseudo-code that implements the computation of $t^*$ and these modifications. If the list of neurons $\texttt{C}$ is provided as a boolean array which allows random access, then the check if a specific neuron index $(i,j)\in\mathcal{I}$ is an element of $\texttt{C}$ requires constant time and therefore, the computational complexity of the whole procedure is of order  $\mathcal{O}(\sum_{l=1}^{L}n_ln_{l-1})$. The desired minimum and minimizer are basically computed in a network forward pass.
\begin{figure}[htpb]
\centering
\includegraphics[width=0.6\textwidth]{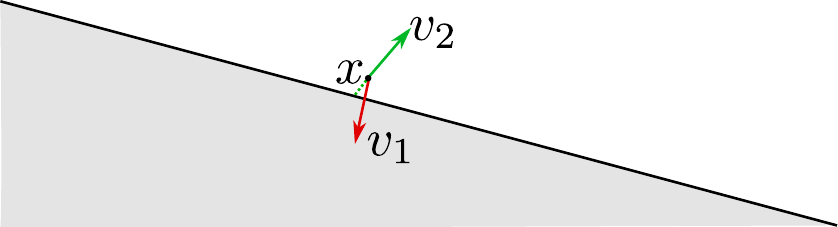}
\caption{Illustration of the mismatch robustness implemented in Algorithm~\ref{alg:advMaxAdjusted}. Assume that the activation pattern $s$ is compatible with arguments $\tilde x$ in the gray area, i.e. $s\in S^C(\tilde x)$ for such $\tilde x$. However the current position $x$ is not in the gray area, $s\not\in S^C(x)$. We distinguish two cases. If the specified direction $v=v_2$ points away from the gray area we allow for negative $t^*$ to get back to the region boundary (black line). If the specified direction $v=v_1$ points towards the gray area, then we ignore this region boundary in the construction of $t^*$. This way we protect against slight disturbances of the specified position $x$ and the assumption $s\in\mathcal{S}^C(x)$.}
\label{fig:mismatchRobustness}
\end{figure}
\begin{algorithm}
\caption{Advance maximally in a specified direction}
\label{alg:advMaxAdjusted}
\begin{algorithmic}[1]
	\Procedure{AdvanceMax}{$\texttt{x}$:position, $\texttt{v}$: direction, $\texttt{s}$:activation pattern, $\texttt{C}$: list of neuron positions to ignore}
	\State $\alpha_1\gets W_{1}\texttt{x}+b_{1}$
	\State $\beta_1\gets W_{1}\texttt{v}$
	\State $c\gets (1,1)$
	\State $t^*\gets\infty$
\For {$l=1,\dots,L$}
\For {$j=1,\dots,n_l$}
\If {$(l,j)\not\in \texttt{C}$ \textbf{and} $\beta_{lj}\neq 0$}
\State $t\gets -\tfrac{\alpha_{lj}}{\beta_{lj}}$
\If{$(\texttt{s}_{l,j}=1\land \beta_j<0)\lor(\texttt{s}_{l,j}=0\land \beta_j>0)$}
\If{$t<t^*$}
\State $t^*\gets t$
\State $\texttt{c}\gets (l,j)$
\EndIf
\EndIf
\EndIf
\EndFor
\State $\alpha_{l+1}\gets W_{l+1}\textnormal{diag}(\texttt{s}_{l})\alpha_l+b_{l+1}$
\State $\beta_{l+1}\gets W_{l+1}\textnormal{diag}(\texttt{s}_{l})\beta_l$
\EndFor
\State\Return $(\texttt{x}+t^*\texttt{v},\texttt{c})$
\EndProcedure
\end{algorithmic}
\end{algorithm}
\section{The Deep ReLU Simplex Algorithm}
\label{sec:drls}
\subsection{Description}
\label{sec:description}
In this section, we describe our \emph{Deep ReLU Simplex (DRLSimplex) algorithm}, an extension of the simplex algorithm from linear programming to deep ReLU programming problems, iterating on vertices of the objective neural network. Where appropriate, we provide pseudo-code. The objective feed-forward neural network $f$ is assumed to be implicitly given. 
We require it to satisfy the following two conditions:
\begin{enumerate}
	\item There need to exist vertices because we aim to first find a vertex and will then iterate on vertices.
	\item The network $f$ needs to be regular in the sense of Section~\ref{sec:regularity} because we rely on the linear independence of the subjective critical normal vectors $\left( \tilde v_{s,l,j} \right)_{(l,j)\in \tilde C_s(x)}$ and on the existence of exactly $2n_0$ region separating axes at every visited vertex $x$ with compatible $s\in\mathcal{S}^C(x)$.
\end{enumerate}
\subsubsection{State Variables}
The state of the algorithm involves the following quantities. The current position in the argument space $\mathbb{R}^{n_0}$ will be saved in the state variable $\texttt{x}\in\mathbb{R}^{n_0}$. Furthermore, the variable $\texttt{s}\in\mathcal{S}$ will always hold an activation pattern that is compatible with $\texttt{x}$, i.e. it shall always hold $\texttt{s}\in\mathcal{S}^C(\texttt{x})$. The variable $\texttt{C}$ is an ordered collection of elements in $\mathcal{I}$. Its purpose is to keep track of the critical indices $\tilde C_{\texttt{s}}$. We assume that $\texttt{C}$ contains at most $n_0$ such indices. Note that this assumption is true by our regularity assumption.

Finally, the algorithm will also keep track of feasible axes in the rows of a matrix $\texttt{A}^+\in\mathbb{R}^{|\texttt{C}|\times n_0}$. This matrix will always be the pseudoinverse of the matrix whose columns are the oriented normal vectors $\tilde u_{\texttt{s},l_1,j_1},\ldots,\tilde u_{\texttt{s},l_{|\texttt{C}|},j_{|\texttt{C}|}}$ where $((l_1,j_1),\ldots,(l_{|\texttt{C}|},j_{|\texttt{C}|}))=\texttt{C}$. The pseudoinverse will not be recomputed in every step but incrementally updated using the algorithmic primitives from Sections~\ref{sec:addColumn}, \ref{sec:dropColumn} and \ref{sec:updatingFeasibleAxes}.

\subsubsection{Initialization}
\label{sec:initialization}
Our DRLSimplex algorithm takes an initial position $\texttt{x}=x_0$ as input. For simplicity, we assume that this initial position is not at the boundary of two affine regions, i.e. $C(x_0)=\left\{  \right\}$. Hence we initialize $\texttt{C}$ to be empty and the matrix $\texttt{A}^+\in\mathbb{R}^{0\times n_0}$ to have no rows. We then compute the activation pattern $\texttt{s}=S(x)$ in a forward pass in $\mathcal{O}(\sum_{l=0}^Ln_ln_{l-1})$. Below we will denote this procedure by \textsc{Initialize(\texttt{x})} and assume it returns the initialized variables $(\texttt{s},\texttt{C},\texttt{A}^+)$.
\subsubsection{Finding a vertex}
After the initialization we need to perform $n_0$ iterations to find a regular vertex of the affine region corresponding to the initial activation pattern $\texttt{s}$. In this phase the activation pattern $\texttt{s}$ and hence affine region is not changed. We first compute the gradient $\nabla_{\texttt{s}}\in\mathbb{R}^n_0$ as explained in Section~\ref{sec:computationGradient} and set the initial direction to $\texttt{v}=-\nabla_{\texttt{s}}$. We then move into the direction opposite to the gradient until we encounter the border of the current affine region. This is done maximally such that the current affine region is not left by using the $\textsc{AdvanceMax}$ subroutine from Section~\ref{sec:advanceMax}. The matrix $\texttt{A}^+$ is updated by using the $\textsc{AddAxis}$ procedure from Section~\ref{sec:addColumn} and the neuron index $\texttt{c}$ returned by $\textsc{AdvanceMax}$ is added to the set $\texttt{C}$ of critical indices in line~\ref{line:addItem} of Algorithm~\ref{alg:FindVertex}. Then the direction for the next step is adjusted by projecting it onto the space orthogonal to span of all normal vectors with indices in $\texttt{C}$. Note that for any $v\in \mathbb{R}^{n_0}$ and subspace $U\subset \mathbb{R}^{n_0}$, the projection $v_U$ of $v$ onto $U$ satisfies $\langle v_U,-v\rangle\le 0$ such that the objective function will not increase in the next iteration. This procedure of moving maximally into a direction and the adjustment of used direction, $\texttt{A}^+$ and $\texttt{C}$ is repeated $n_0$ times until we arrive at a vertex, see Figure~\ref{fig:findVertex} and the pseudo-code of the procedure \textsc{FindVertex} in Algorithm~\ref{alg:FindVertex}.
\begin{figure}[htpb]
\centering
\includegraphics[width=0.5\linewidth]{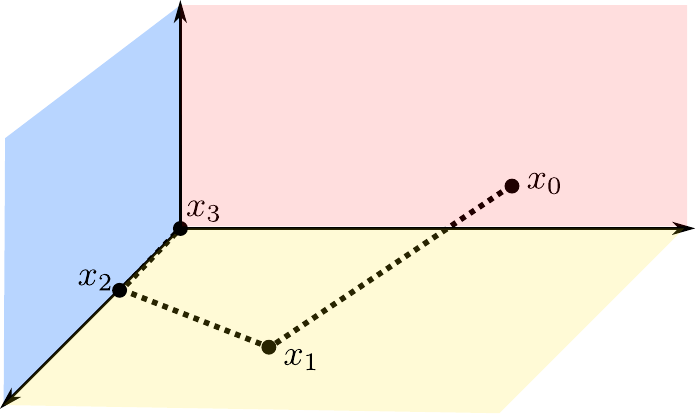}
\caption{
	Illustration of the procedure of finding a vertex $x_3$ for $n_0=3$ starting at position $x_0$. The vector $x_2-x_1$ is orthogonal to the normal vector of the yellow hyperplane, the vector $x_3-x_2$ is orthogonal to normal vectors of both, the yellow and the blue hyperplanes.
}
\label{fig:findVertex}
\end{figure}
\begin{algorithm}
\caption{Find a vertex}
\label{alg:FindVertex}
\begin{algorithmic}[1]
	\Procedure{FindVertex}{$\texttt{x}$: position, $\texttt{s}$: activation pattern, $\texttt{C}$: list of neuron positions, $\texttt{A}^+$: matrix}
	\State $\nabla_{\texttt{s}}\gets \Call{Gradient}{\texttt{s}}$
	\State $\texttt{v}\gets -\nabla_{\texttt{s}}$
\While{$|\texttt{C}|< n_0$}
\State $(\texttt{x},\texttt{c})\gets \Call{AdvanceMax}{\texttt{x},\texttt{v},\texttt{s},\texttt{C}}$
\State $\texttt{A}^+\gets \Call{AddAxis}{\texttt{A}^+,\texttt{s},\texttt{c}}$
\State $\texttt{C}\gets \Call{AddItem}{\texttt{C}^+,\texttt{c}}$\label{line:addItem}
\State $\texttt{v}\gets \texttt{v}-\Call{Project}{\texttt{A}^+,\texttt{C},\texttt{s},\texttt{v}}$
\EndWhile
\State \Return $(\texttt{x},\texttt{s},\texttt{C},\texttt{A}^+)$
\EndProcedure
\end{algorithmic}
\end{algorithm}
\subsubsection{Iterating on vertices and changing the affine region}
\label{sec:drlp}
After a vertex has been found by the \textsc{FindVertex} procedure, the returned position $\texttt{x}$ is a vertex and the rows of the matrix $\texttt{A}^+$ are the feasible axes. Similar to the simplex algorithm in linear programming, we want to choose one of these axes to proceed the iteration on edges. For our purposes we use the procedure $\textsc{ChooseAxis}$ from Algorithm~\ref{alg:chooseAxis} which selects the axis with the strongest correlation with an input direction $\texttt{v}$. 
\begin{algorithm}
\caption{Choose Axis with highest correlation}
\label{alg:chooseAxis}
\begin{algorithmic}[1]
	\Procedure{ChooseAxis}{$\texttt{A}^+$: matrix, $\texttt{v}$: direction}
	\State $\texttt{i}\gets \textnormal{argmax}(\left\{ 1,\ldots,\textnormal{nrows}(\texttt{A}^+) \right\}; i\mapsto \langle \texttt{v},\texttt{A}^+_{i,\cdot}\rangle/\sqrt{\langle \texttt{A}^+_{i,\cdot},\texttt{A}^+_{i,\cdot}\rangle})$\label{line:argmax}
	\State $\texttt{a}^+\gets \texttt{A}^+_{\texttt{i},.}$
	\State \Return $(\texttt{a}^+,\langle \texttt{v},\texttt{a}^+\rangle,\texttt{i})$\label{line:returnAxis}
\EndProcedure
\end{algorithmic}
\end{algorithm}
In line~\ref{line:argmax} the expression ``$\textnormal{nrows}(\texttt{A}^+)$'' returns the number of rows of the matrix $\texttt{A}^+$.
In line~\ref{line:returnAxis} we return the selected axis $\texttt{a}^+$, i.e. row of $\texttt{A}^+$, its inner product $\langle \texttt{v},\texttt{a}^+\rangle$ with the input direction $\texttt{v}$ and the row index $\texttt{i}$ corresponding to $\texttt{a}^+$.  

Now we can compose the subroutines we previously described to form the final DRLSimplex Algorithm~\ref{alg:main}. If we use the negative gradient $-\nabla_{\texttt{s}}$ for the direction $v\in\mathbb{R}^{n_0}$ in Algorithm~\ref{alg:chooseAxis} then the sign of the returned inner product $\langle\texttt{v},\texttt{a}^+\rangle$ tells us if there exists a feasible axis which has a negative inner product with the gradient $\nabla_{\texttt{s}}$ depending on which we are taking the following actions:
\begin{itemize}
	\item If there is such an axis $\texttt{a}^+$, the row with the corresponding index $\texttt{i}$ will be removed from the matrix $\texttt{A}^+$ and we proceed by moving into the direction of this axis using the $\textsc{AdvanceMax}$ procedure. Either there is no neuron that changes its activation in the direction $\texttt{a}^+$. In this case the resulting function can be made arbitrarily small and the algorithm terminates\footnote{For simplicity of the exposition, we did not consider this case where $t^*=\infty$ in Algorithm \ref{alg:advMaxAdjusted}}. Or the procedure returns the new position $\texttt{x}\in\mathbb{R}^{n_0}$ and the index $\texttt{c}\in\mathcal{I}$ of the corresponding neuron that changes its activation at that position. Based on this information we modify the matrix $\texttt{A}^+$ using the $\textsc{AddAxis}$ function from Section~\ref{sec:addColumn} such that its rows are the feasible axes $\tilde{\mathcal{A}}_{\texttt{s}}(\texttt{x})$. The old neuron index corresponding to the hyperplane that was abandoned in the $\textsc{AdvanceMax}$ step is removed from the ordered collection $\texttt{C}$ in line \ref{line:removeAtIndex} and the new index $\texttt{c}$ is appended in line~\ref{line:addItemDRLSimplex}. 
		
		At this stage, the considered affine region of $f$ shall be changed to allow to continue in a different region. To this end the following two state variables need to be adjusted.
		\begin{enumerate}
			\item The state variable $\texttt{s}$ which specifies the activation pattern is flipped at the position $\texttt{c}$ corresponding to the critical neuron. This is achieved by the $\textsc{Flip}$ function from Section~\ref{sec:flip} in line \ref{line:flip} of Algorithm~\ref{alg:main}.
			\item The set of feasible axes need to be updated. Since only one activation changes, we can exploit the fact that only one axis needs to be recomputed as described in Section~\ref{sec:updatingFeasibleAxes} by using the $\textsc{UpdateAxisNewRegion}$ in line~\ref{line:UpdateAxis}.
		\end{enumerate}
		We can now continue with the next iteration as above by computing the gradient $\nabla_\texttt{s}$ for the new $\texttt{s}$. However before the next iteration is made, we assign $1$ to an additional variable $\texttt{nextFlipIndex}$ which is explained below. 
	\item If there is no axis $\texttt{a}^+$ which has negative correlation with the gradient then the current vertex $\texttt{x}$ is a local minimum of the affine region corresponding to the activation pattern $\texttt{s}$. In this case we will change $\texttt{s}$ to a different compatible activation pattern in $\mathcal{S}^C(x)$. After $\texttt{s}$ was changed using the $\textsc{Flip}$ function in line~\ref{line:flip2} we update the axes in $\texttt{A}^+$ for the new activation pattern using $\textsc{UpdateAxisNewRegion}$ in line~\ref{line:UpdateAxis2}. Then we recheck whether there is a feasible axis among the new axes which has negative inner product with the new gradient $\nabla_{\texttt{s}}$. In each retry we flip $\texttt{s}$ at a different neuron. More precisely we subsequently flip at $\texttt{C}[1]$, $\texttt{C}[2]$,\ldots,$\texttt{C}[n_0]$. This is achieved by increasing the index counter $\texttt{nextFlipIndex}$ in line~\ref{line:nextFlipIncrease}. After $n_0$ such flips without finding a new direction to continue with, Proposition~\ref{prop:localMinimum} guarantees that $\texttt{x}$ is a local minimum of $f$ by and the procedure finishes returning $\texttt{x}$ in line~\ref{line:return}.
\end{itemize}
\begin{algorithm}
\caption{DRLSimplex Algorithm}
\label{alg:main}
\begin{algorithmic}[1]
	\Procedure{DRLSimplex-Algorithm}{$\texttt{x}$: starting position}
	\State $(\texttt{s},\texttt{C},\texttt{A}^+)\gets \Call{Initialize}{\texttt{x}}$
\State $(\texttt{x},\texttt{s},\texttt{C},\texttt{A}^+)\gets \Call{FindVertex}{\texttt{x},\texttt{s},\texttt{C},\texttt{A}^+}$
\While{true}
\State $\nabla_{\texttt{s}}\gets \Call{Gradient}{\texttt{s}}$
\State $(\texttt{a}^+,\alpha,\texttt{i})\gets  \Call{ChooseAxis}{\texttt{A}^+, \texttt{C},\texttt{s} ,\nabla_{\texttt{s}}}$
\If{$\alpha<0$}\label{line:alphasmaller}
\State $\texttt{A}^+\gets  \Call{RemovePseudorow}{\texttt{A}^+,\texttt{i}}$
\State $(\texttt{x},\texttt{c})\gets \Call{AdvanceMax}{\texttt{x},\texttt{a}^+,\texttt{s},\texttt{C}}$\label{line:advanceMax}
\State $\texttt{A}^+\gets \Call{AddAxis}{\texttt{A}^+,\texttt{c}}$
\State $\texttt{C}\gets \Call{RemoveAtIndex}{\texttt{C},\texttt{i}}$\label{line:removeAtIndex}
\State $\texttt{C}\gets \Call{AddItem}{\texttt{C},\texttt{c}}$\label{line:addItemDRLSimplex}

\State $\texttt{s}\gets \Call{Flip}{\texttt{s},\texttt{c}}$\label{line:flip}
\State $\texttt{A}^+\gets \Call{UpdateAxisNewRegion}{\texttt{A}^+,\texttt{i},\texttt{s},\texttt{c}}$\label{line:UpdateAxis}
\State $\texttt{nextFlipIndex}\gets 1$
\Else
\If {$\texttt{nextFlipIndex}>n_0$}
\State \Return $\texttt{x}$\label{line:return}
\EndIf
\State $\texttt{c}\gets \texttt{C}[\texttt{nextFlipIndex}]$
\State $\texttt{s}\gets \Call{Flip}{\texttt{s},\texttt{c}}$\label{line:flip2}
\State $\texttt{A}^+\gets \Call{UpdateAxisNewRegion}{\texttt{A}^+,\texttt{nextFlipIndex},\texttt{s},\texttt{c}}$\label{line:UpdateAxis2}
\State $\texttt{nextFlipIndex}\gets \texttt{nextFlipIndex}+1$\label{line:nextFlipIncrease}
\EndIf
\EndWhile
\EndProcedure
\end{algorithmic}
\end{algorithm}
\subsubsection{Computational Complexity per Iteration}
\label{sec:computationalComplexity}
Table~\ref{tab:complexity} summarizes the computational complexity orders for the subroutines used of the previous and this section.
\begin{table}[htpb]
\centering
\caption{Computational complexity of different subroutines used in our solvers}
\label{tab:complexity}
\begin{tabular}{|l|c|c|}
\hline
Algorithm routine & Described in Section & Computational complexity\\
\hline
\textsc{Gradient}& \ref{sec:computationGradient}& $\mathcal{O}(\sum_{l=1}^L n_ln_{l-1})$ \\
\hline
\textsc{Project}& \ref{sec:project} & $\mathcal{O}(\sum_{l=1}^L n_ln_{l-1}+n_0^2)$ \\
\hline
$\textsc{AddAxis}$& \ref{sec:addColumn}& $\mathcal{O}(\sum_{l=1}^L n_ln_{l-1}+n_0^2)$ \\
\hline
$\textsc{DropAxis}$& \ref{sec:dropColumn}& $\mathcal{O}(n_0^2)$ \\
\hline
\textsc{Flip}& \ref{sec:flip}& $\mathcal{O}(1)$ \\
\hline
$\textsc{UpdateAxisNewRegion}$&\ref{sec:updatingFeasibleAxes} & $\mathcal{O}(\sum_{l=1}^L n_ln_{l-1}+n_0^2)$ \\
\hline
$\textsc{AdvanceMax}$& \ref{sec:advanceMax}& $\mathcal{O}(\sum_{l=1}^L n_ln_{l-1})$ \\
\hline
\textsc{Initialize}& \ref{sec:initialization}& $\mathcal{O}(\sum_{l=1}^L n_ln_{l-1})$ \\
\hline
\end{tabular}
\end{table}
Note that all these subroutines can be executed in complexity order $\mathcal{O}(\sum_{l=1}^{L}n_{l}n_{l-1}+n_0^2)$. 
In particular every loop iteration in Algorithms~\ref{alg:FindVertex} and \ref{alg:main} is of the same complexity order. We call these loop iterations \emph{steps} of our algorithm, since either the position of $\texttt{x}$ or the considered activation pattern $\texttt{s}$ changes. Every such step is of order $\mathcal{O}(\sum_{l=1}^{L}n_{l}n_{l-1}+n_0^2)$. 

\subsubsection{Possible modifications}
\label{sec:modifications}
In our pseudo-code we focussed on an easy to understand proof-of-concept implementation and there are plenty possible extensions and improvements.

\paragraph{Position correction}
In our DRLSimplex algorithm, we keep track of the position $\texttt{x}\in\mathbb{R}^{n_0}$ and the activation pattern $\texttt{s}\in\mathcal{S}$ separately, these variables are only synchronized at the beginning in the \textsc{Initialize} function. Due to numerical uncertainty, they can diverge, i.e. $\texttt{s}\not\in \mathcal{S}^C(\texttt{x})$. To avoid such divergence, it may be beneficial to force the position to be exactly in a vertex whenever this is implicitly assumed. More precisely, after the $\textsc{FindVertex}$ procedure in Algorithm \ref{alg:main}, the position $\texttt{x}$ should only iterate on vertices. This means that at the neuron positions $(l,j)\in\texttt{C}\subset \mathcal{I}$ the ReLU arguments need to be zero, i.e. $A(x)_{lj}=0$. If $\texttt{s}\in\mathcal{S}^C(x)$ as assumed in our algorithm, $\tilde A_{\texttt{s}}(x)_{lj}=0$ for $(l,j)\in\texttt{C}$ by Theorem~\ref{thm:argumentRepresentation}. By equation~\eqref{eq:explFormulation} this yields a linear equation system involving inner products with normal vectors corresponding to the neuron positions $\texttt{C}$. Using the inner product computation from Section~\ref{sec:multipleInner}, equation~\eqref{eq:orientedNormal} and the state matrix $\texttt{A}^+$ this equation system can be solved in complexity order $\mathcal{O}(\sum_{l=1}^L n_ln_{l-1}+n_0^2)$. Hence, this position correction does not increase the overall complexity order per step.

\paragraph{Axis correction}
Again due to numerical uncertainty and the fact that the feasible axes in the state matrix $\texttt{A}^+$ are recursively modified, numerical errors can accumulate and it might be necessary to do a fresh computation of $\texttt{A}^+$ as the pseudoinverse of the corresponding oriented normal vectors. However, this step would require a computational complexity of $\mathcal{O}(n_0^3)$.

\paragraph{Other activation function}
Our theoretical and algorithmic considerations of Sections~\ref{sec:analysis} and ~\ref{sec:algPrimitives} can be easily extended to feed-forward neural networks using activation functions of the form
\begin{equation}
	\label{eq:otherActivation}
	\sigma_{a,b}:
	\begin{cases}
		\mathbb{R}&\to\mathbb{R}\\
		x&\mapsto a\textnormal{ReLU}(-x)+ b\textnormal{ReLU}(x)
	\end{cases}\quad \textnormal{ for }a,b\in\mathbb{R}.
\end{equation}To apply our algorithm for networks using this activation function, one can adjust the relevant parts of the algorithm such as the gradient computation. 

Another possibility is the representation of a neuron using this activation functions by a linear combination of two neurons each using a ReLU activation function. These pairs of neurons then induce the same hyperplanes and the overall neural network is not be regular anymore. One therefore has consider only one representative of each such pair in the \textsc{AdvanceMax} procedure and flip their activation synchronously in the \textsc{Flip} procedure. We use this technique in our Julia implementation described in Section~\ref{sec:juliaImplementation}. This way we are able to successfully demonstrate our algorithm in quantile regression, where  final layer's activation function is of the form~\eqref{eq:otherActivation}.

\paragraph{Alternatives to the iteration on vertices}
Our DRLSimplex algorithm first finds a vertex and then proceeds on vertices in subsequent iterations. The theory in Section~\ref{sec:feasibleAxes} on feasible axes for regular vertices can easily be extended for regular points. More precisely, the definition of the feasible axes are also valid for regular points $x\in\mathbb{R}^{n_0}$. Furthermore, from the proof of Lemma~\ref{lem:positiveSpan} it is clear that
\begin{equation}
	\label{eq:feasibleRegionAlternative}
	v\in \tilde D_s(x) \iff \forall (l,j)\in\tilde C_s(x)\; \langle v, \tilde u_{s,l,j}\rangle\ge 0 \quad \textnormal{ for } s\in\mathcal{S}^C(x)
\end{equation} such that $v\in \tilde D_s(x)$ if and only if $v=v_{\perp}+v_{\parallel}$ with $v_{\parallel}=\sum_{(l,j)\in \tilde C_s(x)}\alpha_{lj} \tilde a_{s,x,l,j}$ for non-negative coefficients $\alpha_{lj}$, $(l,j)\in \tilde C_s(x)$ and $\langle v_{\perp},\tilde u_{s,l,j}\rangle =0$ for $(l,j)\in \tilde C_s(x)$. This decomposition allows to find directions $v\in \tilde D_s(x)$ that have negative inner product with the gradient $\nabla_s$ which can then be used instead of $\texttt{a}^+$ in line~\ref{line:advanceMax} of Algorithm~\ref{alg:main}.

We can then drop the first condition on $f$ we required at the beginning of Section~\ref{sec:description}, since we do not need iterate on vertices anymore.
\paragraph{Quadratic Deep ReLU Programming}
For $m\in\mathbb{R}$, a positive definite $m$-ary quadratic form $q:\mathbb{R}^{m}\to\mathbb{R}_+$ is a degree 2 polynomial of the form $q(x)=x^\intercal Ax$ for a positive definite matrix $A\in\mathbb{R}^{m\times m}$. In particular for $x,v\in\mathbb{R}^m$ it holds that
\begin{equation}
\label{eq:quadraticForm}
q(x+tv)=at^2+ bt+c \textnormal{ for }t\in\mathbb{R}
\end{equation}
with real constants $a=(v^\intercal A v)$, $b=(x^\intercal (A+A^\intercal) v)$ and $c=x^{\intercal}A x$. If $v$ is non-zero, then the minimum $t^*_q$ of the parabola~\eqref{eq:quadraticForm} is given by the linear equation 
\begin{equation}
	\label{eq:quadraticFormMinimum}
	2a t^*_q+b=0.
\end{equation}
One can therefore extend our algorithm to be applicable to a combination of definite quadratic forms and ReLU feed-forward neural networks. For example let $f$ be as in equation~\eqref{eq:f}, $q$ a definite quadratic form with $m=n_0$ and define the objective
\begin{equation*}
	\mathfrak{f}:
	\begin{cases}
		\mathbb{R}^{n_0}&\to\mathbb{R}\\
		x&\mapsto f(x)+q(x)
	\end{cases}.
\end{equation*}For a position and a direction $x',v\in\mathbb{R}^{n_0}$ with $\lim_{h\to 0} (\mathfrak{f}(x+hv)-\mathfrak{f}(x))/h<0$ we can use the \textsc{AdvanceMax} procedure to obtain $t^*>0$ maximally such that $t\mapsto f(x+tv)$ is affine on $[0,t^*]$. Then by equation~\eqref{eq:quadraticForm}, $\mathfrak{f}$ is a parabola on $[0,t^*]$ whose extreme value position $\tilde t$ restricted to this interval can be easily computed by the linear equation~\eqref{eq:quadraticFormMinimum} and a comparison against $0$ and $t^*$. By the assumption on $v$, $\tilde t>0$ and $\mathfrak{f}(x)>\mathfrak{f}(x+\tilde t v)$. This shows how to automatically and efficiently select the step size in this case. At position $x\in\mathbb{R}^{n_0}$ with compatible activation pattern $s\in\mathcal{S}^C(x)$ the direction $v$ for the next iteration needs to be selected to have negative inner product with the gradient 
\begin{equation}
	\label{eq:frakGradient}
	\nabla_{\mathfrak{f},s}(x):=\nabla q(x)+\nabla_s
\end{equation}
within $ \tilde D_s(x)$ defined as in~\eqref{eq:feasibleRegionAlternative} for the state variable $s\in\mathcal{S}(x)$. If this is not possible, the compatible activation pattern $s\in\mathcal{S}(x)$ has to be changed to consider a different adjacent affine region of $f$ as described in Section~\ref{sec:drlp}. Note that local minima of $\mathfrak{f}$ are not necessarily on the vertices such that we have to resort to alternative iterations as described above. 

More generally, we can apply these modifications to allow for objective functions $Q$ of the form 
	\begin{equation}
		\label{eq:quadraticObjectiveFunction}
		Q:
		\begin{cases}
			\mathbb{R}^{n_0}&\to\mathbb{R}\\
			x&\mapsto q(x_1,\ldots,x_{n_0},f(x))
		\end{cases}
	\end{equation} for an $m=(n_0+1)$-ary quadratic form $q$. An example of such a \emph{Quadratic Deep ReLU Programming} problem is the LASSO optimization which we will cover in Section~\ref{sec:examples} below.

\subsection{Julia Implementation}
\label{sec:juliaImplementation}
We provide a simple proof-of-concept implementation of our DRLSimplex algorithm in the Julia programming language in the form of a GitHub repository\footnote{\url{https://github.com/hinzstatmathethzch/DRLP}}. It incorporates some of the extensions discussed in Section~\ref{sec:modifications}, specifically the use of two ReLU units to replicate absolute value function and the ability to allow for quadratic forms. However, we rather focussed on simple and instructive code and not on fast execution. In particular, we did not implement the incremental pseudoinverse matrix updating procedure described in Section~\ref{sec:dropColumn} and \ref{sec:dropColumn} because without a proper pivoting strategy as common in linear programming algorithms, this method suffers from numerical instability. Instead, for simplicity we use a full matrix inversion in every step which. This leads to a suboptimal computational complexity per step of order $\mathcal{O}(\sum_{l=1}^Ln_ln_{l-1}+n_{0}^3)$ in our simple implementation. In the next section, we demonstrate several applications which are also available in the code repository.
\subsection{Examples}
\label{sec:examples}
\subsubsection{Local minimum of random weight network}
We first want to apply our algorithm to a neural network with random weight sampled from a continuous distribution. By Theorem~\ref{thm:regularNetwork}, the resulting neural network $f$ will be almost surely regular\footnote{Here we ignore the fact, that finite precision floating point arithmetic only allows discrete distributions.}. As noted in the introduction, the function $f$ may attain infinitely small values and in this case the DRLSimplex algorithm would terminate in the \textsc{AdvanceMax} step. Otherwise it will stop after a finite number of iterations because the function values at different visited vertices will strictly decrease by line~\ref{line:alphasmaller} of Algorithm~\ref{alg:main}. In Figures~\ref{fig:localMinDim1} and \ref{fig:localMin} we depict the objective function centered around the local minimum found by our algorithm for input dimensions $n_0=1$ and $n_0=2$.
\begin{figure}[htpb]
\centering
\includegraphics[width=0.6\linewidth]{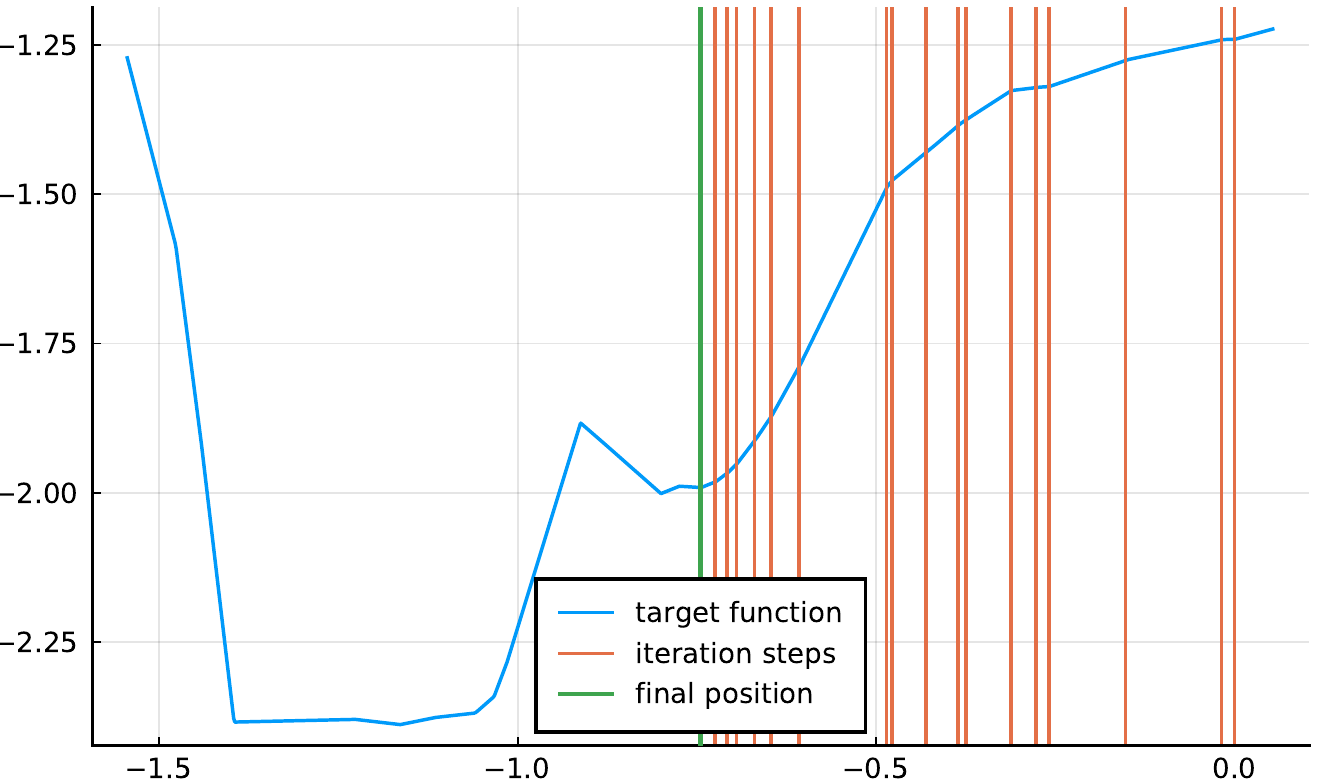}
\caption{
	After 18 iterations, our DRLSimplex algorithm successfully finds a local minimum of a feed-forward ReLU neural network with uniform random parameters on $[-1,1]$ and topology $(n_0,\ldots,n_{L+1})=(1,50,10,10,10,10,10,1)$ starting at position $x_0=0$. The red vertical lines represent the iteration steps. The green vertical line indicates the final position at termination.
}
\label{fig:localMinDim1}
\end{figure}
\begin{figure}[htpb]
\centering
\includegraphics[width=0.6\linewidth]{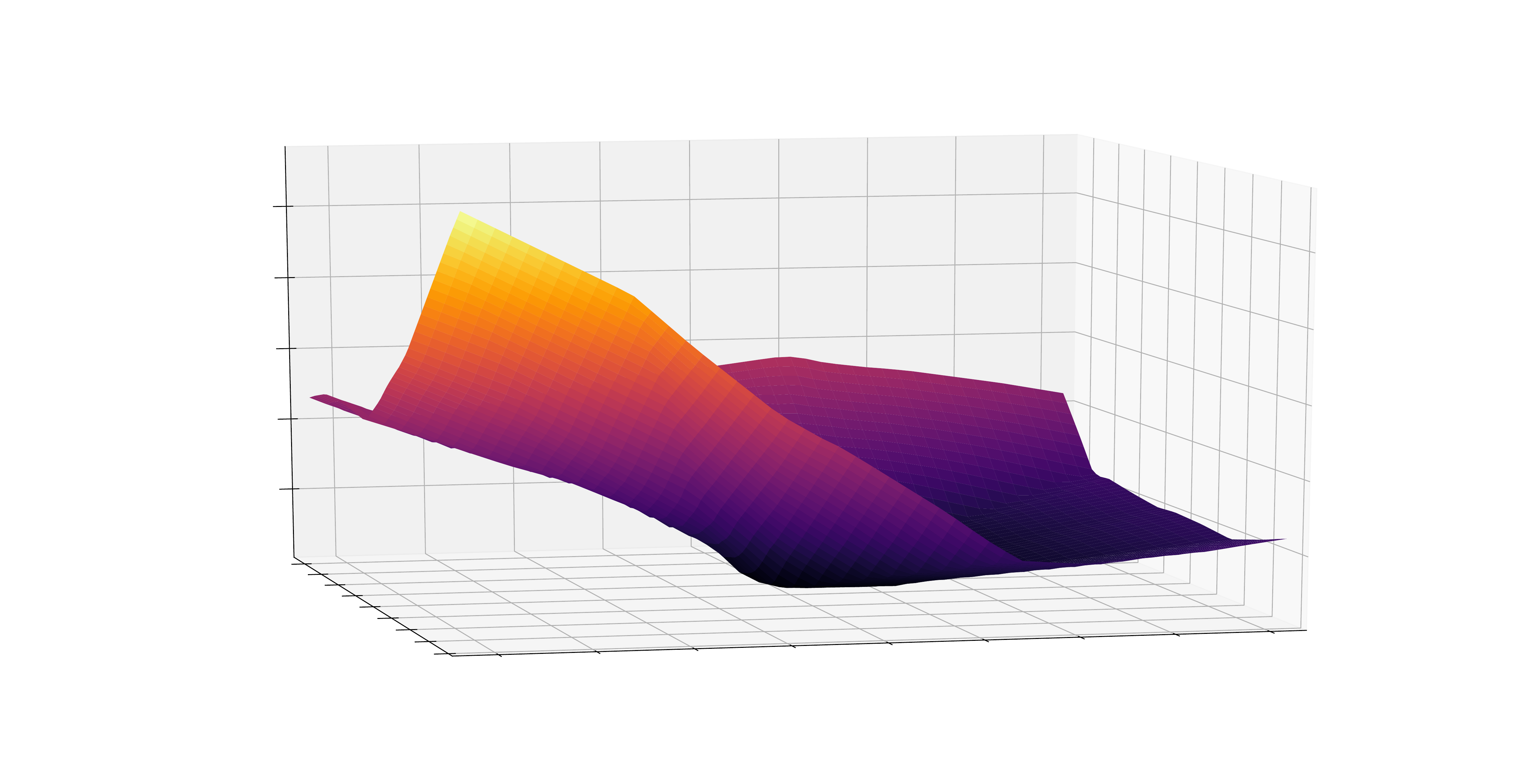}
\caption{
	Surface plot of the objective function around a local minimum found with our DRLSimplex algorithm applied to a random feed-forward ReLU neural network with topology $(n_0,\ldots,n_{L+1})=(2,10,10,10,10,10,1)$.
}
\label{fig:localMin}
\end{figure}

While this demonstrates our algorithm in practice, it is not immediately clear why it is useful to minimize a real-valued feed-forward neural network in its argument. In the next sections below we will transform loss functions of common machine-learning problems into such networks and thus, our algorithm allows provides a training method in these cases. Existing algorithms can then be seen as special cases of our DRLSimplex algorithm. 

\subsubsection{LASSO penalized quantile regression}
\label{sec:quantileReg}
For $p,N\in\mathbb{N}$, consider predictor and response samples $x_1,\dots,x_N\in\mathbb{N}^{p}$, $y_1,\dots,y_N\in\mathbb{R}$. Quantile regression as introduced in \cite{koenker1978regression} aims to minimize $\theta\mapsto\sum_{i=1}^N \rho_\alpha(y-\theta_0-\sum_{j=1}^p\theta_j x_{ij})$ over $\theta\in\mathbb{R}^{p+1}$ with $\rho_\alpha:x\mapsto \alpha\max(x,0)+(1-\alpha)\max(-x,0)$, $\alpha\in[0,1]$. For more generality, we also incorporate a LASSO penalty term weighted by $\lambda\ge 0$ in our loss function
\begin{equation}
	\label{eq:quantreg}
	\mathcal{L}:
\begin{cases}
\mathbb{R}^{p+1}&\to\mathbb{R}_{\ge 0}\\
\theta&\mapsto\sum_{i=1}^{N}\rho_\alpha( y_i-\sum_{j=1}^{p}\theta_j x_{ij}-\theta_0) +\lambda\sum_{j=1}^p\vert\theta_j\vert
\end{cases}.
\end{equation}
For $\lambda=0$ and $\alpha=0.5$ we obtain the loss function of \emph{Least Absolute Deviation} (LAD) regression \cite{Bloomfield1983}. Such optimization problems are usually solved using specialized linear programming algorithms such as the Barrodale-Roberts algorithm~\cite{barrodale1973improved}. 

Note that $\mathcal{L}(x)= W_{2}\textnormal{ReLU.}(W_1 x+ b_1)$ with matrices 
\begin{equation*}
{\tiny
W_1=
\begin{pmatrix}
-\alpha\mathds{X}\\
(1-\alpha)\mathds{X}\\
\lambda\textnormal{Id}_p\\
-\lambda\textnormal{Id}_p
\end{pmatrix}_{2(N+p)\times p+1},  
\mathds{X}=
\begin{pmatrix}
1& x_1^\intercal\\
\vdots& \vdots\\
1& x_N^\intercal
\end{pmatrix}_{N\times p+1}, 
b_1=
\begin{pmatrix}
\alpha Y\\
-(1-\alpha)Y\\
0_{2p\times 1}
\end{pmatrix}_{2(N+p)\times 1},
Y=
\begin{pmatrix}
y_1\\
\vdots\\
y_N
\end{pmatrix}_{N\times 1}
}
\end{equation*}and $W_2=(1,\ldots,1)\in\mathbb{R}^{1\times 2 (N+p)}$ is a one hidden layer feed-forward neural network with layer widths $(n_0,n_1,n_2)=(p+1,2(N+p),1)$. Each absolute value and $\rho_\alpha$ function is replicated by a linear combination of two ReLU neurons. They are inducing the same hyperplanes such that the resulting network is not regular. However our algorithm can be applied if we only consider the activation of one representative for each such pair as described in Section~\ref{sec:modifications}. 

We have incorporated this extension in our Julia implementation and for a sample data set we have tested our algorithm starting from a random position and compared its final position to the output of the \texttt{rq.fit.lasso} R function. Indeed our algorithm terminates with the correct minimizer. We have included this example in our code repository, see Section~\ref{sec:juliaImplementation}. While our implementation will not provide an advantage over the existing specialized optimization methods, this demonstrates the universality of our algorithm. In fact, the Barrodale-Roberts algorithm can be seen as a special case of our DRLSimplex-Algorithm as we will describe next. 
\paragraph{Barrodale-Roberts (BR) as a special case of DRLSimplex} Similarly to our DRLSimplex algorithm, the BR algorithm iterates on vertices and uses the same stopping criterion. The iteration terminates when there is no region separating axis that has negative inner product with the gradient corresponding the respective adjacent affine regions. However, as we have shown above the loss function $\mathcal{L}$ can be represented as a neural network with only one hidden layer such that there is no composition but only one layer involving the ReLU activation function. In consequence, the partitioning of the input space into regions of affine behaviour is based on $N+p$ hyperplanes and this induces a special vertex structure. More precisely, the region separating axes from equation~\eqref{eq:feasibleRegionseparating} at a vertex $x\in\mathbb{R}^{n_0}$ satisfy $a_{x,1,l,j}=-\alpha a_{x,0,l,j}$ for compatible $(l,j)\in\mathcal{I}$ and $\alpha>0$. Exploiting this fact, the BR algorithm follows ``lines of vertices'' given by the intersection of $n_0-1$ hyperplanes until a vertex is found such that the the objective function cannot further be decreased in the same direction. In particular, the updating process of new feasible axes as described Sections~\ref{sec:addColumn} and \ref{sec:updatingFeasibleAxes} is not necessary at every visited vertex in the DRLSimplex algorithm and the BR algorithm intelligently saves computation time by continuing with the same previous direction until it encounters a vertex where the this direction cannot be used anymore to further decrease the objective function. Only at such vertices, the feasible axes are recomputed to find a direction to continue with or to stop the iteration.

Therefore in its essential structure, the BR algorithm can be seen as a special case of the DRLSimplex algorithm for one-hidden-layer ReLU networks which saves computation time by exploiting the fact that there are no composition of nonlinearities such that the vertex structure is induced by hyperplane intersections. In contrast, the DRLSimplex algorithm can be used for ReLU networks with more layers and therefore has to compute the feasible axes at every visited vertex.

\subsubsection{Censored Least Absolute Deviation}
\label{sec:censoredLAD}
The \emph{censored least absolute deviation} (CLAD) estimator was introduced 1984 by J. Powell in \cite{powell1984least}. For $p,N\in\mathbb{N}$, and predictor and response samples $x_1,\dots,x_N\in\mathbb{R}^{p}$, $y_1,\dots,y_N\in\mathbb{R}$ it aims to minimize the loss
\begin{equation}
	\label{eq:clad}
	\mathcal{L}:
	\begin{cases}
		\mathbb{R}^{p}&\to\mathbb{R}\\
		\theta&\mapsto \sum_{i=1}^N|y_i-\max(\langle \theta, x_i\rangle,0)|
	\end{cases}.
\end{equation}
Quantile regression from Section~\ref{sec:quantileReg} involved no composition of nonlinearities and such that the structure can be easily understood and efficient linear programming based solver algorithms similar to the Barrodale-Roberts algorithm can be formulated. In contrast for the above loss function, a rewrite as a neural network requires two hidden layers since there is a composition of the absolute value function and the nonlinear mapping induced by the maximum. Furthermore the loss function $\mathcal{L}$ is non-convex such that practical solvers for this problem are not apparent.

Powell originally proposed direct programming solvers which only consider function evaluations and ignore the underlying structure of the loss function. 
Later, Buchinsky \cite{buchinsky} proposed his Iterative Linear Programming Algorithm (ILPA) which iteratively performs firstly a standard quantile regression for the previously selected uncensored observations and then checks for the newly estimated parameter $\theta$ which observations are actually uncensored and will be considered in the next iteration, starting with all observations uncensored in the first iteration. Unfortunately, this procedure does not always converge and does not necessarily provide a local minimum as shown by Fitzenberger \cite{Fitzenberger}. In this work, he also presents his BRCENS algorithm, an adaptation of the Barrodale-Roberts algorithm to CLAD regression iterating on the vertices induced by the hyperplanes
corresponding to the absolute value and maximum functions in equation \eqref{eq:clad}. 

\paragraph{BRCENS as a special case of DRLSimplex}
We can obtain BRCENS as a special case of our DRLSimplex algorithm by rewriting the loss function $\mathcal{L}$ as a neural network. For the data matrix $\mathds{X}:=\left( x_1,\dots,x_N \right)^{\intercal}\in\mathbb{R}^{N\times p}$ and the response vector $Y:=(y_1,\dots,y_N)^\intercal\in\mathbb{R}^{N}$ it holds that
\begin{eqnarray*}
	\mathcal{L}(\theta)&=& \sum_{i=1}^{N}\vert y-\max\left( \langle x_i,\theta\rangle, 0 \right)\vert \\
	&=& \sum_{i=1}^{N}\textnormal{ReLU}\left( y_i-\textnormal{ReLU}(\langle x_i,\theta\rangle \right) + \sum_{i=1}^{N}\textnormal{ReLU}\left( -y_i+\textnormal{ReLU}(\langle x_i,\theta\rangle \right)\\
&=& \left( 1,\dots,1 \right)_{1\times N}\textnormal{ReLU.}\left( Y-\textnormal{ReLU.}(\mathds{X}\theta ) \right) + \left( 1,\dots,1 \right)_{1\times N}\textnormal{ReLU.}\left( -Y+\textnormal{ReLU.}(\mathds{X}\theta ) \right)\\
&=& \left( 1,\dots,1 \right)_{1\times 2N}\textnormal{ReLU.}\left( W_2 \textnormal{ReLU.}\left( W_1 \theta\right)+b_2  \right),
\end{eqnarray*}
where
\begin{equation*}
{\scriptsize
W_2=
\begin{pmatrix}
\mathds{1}_N \\
-\mathds{1}_{N}\\
\end{pmatrix}_{2N\times N}, 
W_1= \mathds{X}
,\;
b_2=
\begin{pmatrix}
-Y\\
Y\\
\end{pmatrix}_{2N\times 1}
}.
\end{equation*}
In particular, $\mathcal{L}$ can be written as a $L=2$ hidden layer neural network in the form \eqref{eq:f} with $(n_0,n_1,n_2,n_3)=(p,N,2N,1)$ such that we can apply our DRLSimplex algorithm. 

Note that for $\theta\in\mathbb{R}^{n_0}$, the function $x\mapsto \langle \theta,x\rangle$ is a feed-forward neural network of the form~\eqref{eq:f} with $L=1$, $W_1=\theta^\intercal$, $W_2=1$, $b_1=0$ and $b_2=0$. In this sense the application of the DRLSimplex algorithm to the above problem can be seen as training the first layer parameters of a simple neural network with absolute deviation loss. We will leverage this idea to train the first layer of more general, possibly deep ReLU networks in the next section with our DRLSimplex algorithm.
\subsubsection{First layer L1 loss-optimization for feed-forward neural networks}
In this section, we want to use our DRLSimplex algorithm to train the first layer parameters of a feed-forward neural network $f\in\mathcal{F}$ as in equation~\eqref{eq:f}. To express the variability of the first layer parameters $\theta=(W_\theta,b_\theta)\in\mathbb{R}^{n_1\times n_0}\times \mathbb{R}^{n_0}=:\Theta$ we define 
\begin{equation*}
	f_{\theta}(x)=W_{L+1}g^{(L)}\circ\cdots\circ g^{(2)}\left(\textnormal{ReLU}.\left( W_\theta x+b_\theta\right)  \right)+b_{L+1}
\end{equation*}for $x\in\mathbb{R}^{n_0}$. Note that compared to $f$, the only difference is that the first layer weight and bias parameters are replaced by those specified by $\theta$. In particular $f_\theta=f$ for $\theta=(W_1,b_1)$. Using this notation, the $L1$ loss for fixed training samples $(x_1,y_1),\dots,(x_N,y_N)\in\mathbb{R}^{n_0}\times\mathbb{R}$, $N\in\mathbb{N}$ is given by
\begin{equation*}
	\mathcal{L}:
\begin{cases}
\Theta&\to\mathbb{R}\\
\theta&\mapsto \sum_{i=1}^{N}\vert f_{\theta}(x_i)-y_i\vert
\end{cases}.
\end{equation*}
We now want to construct a feed-forward neural network $f_{\mathcal{L}}:\mathbb{R}^{n_{1}(n_0+1)}\to\mathbb{R}$ that computes this loss $f_L(\tilde\theta)=\mathcal{L}(\theta)$ for $\theta\in\Theta$. Here we denote by $\tilde\theta \in\mathbb{R}^{n_1(n_0+1)}$ the rearrangement of the entries of $\theta=(W_\theta,b_{\theta})\in\mathbb{R}^{n_1\times n_0}\times \mathbb{R}^{n_1}$ into a vector according to the rule 
\begin{equation*}
	\tilde \theta_{k} =
	\begin{cases}
		(W_{\theta})_{ij}\quad &\textnormal{ if } k\le n_1n_0\textnormal{ with }(i,j)\in\left\{ 1,\ldots,n_1 \right\}\times \left\{ 1,\ldots,n_0 \right\} \textnormal{ such that } k=in_0+j\\
		b_{k-n_1n_0}\quad& \textnormal{ if } k>n_1n_0
	\end{cases}
\end{equation*}for indices $k\in\left\{ 1,\ldots,n_1(n_0+1) \right\}$. This means that we first fill the parameters of the weight matrix $W_{\theta}$ row-wise into the rearrangement $\tilde \theta$, followed by the entries of the bias vector $b_\theta$. We also need to introduce the corresponding $(Nn_1)\times (n_0+1)$ data matrix
$M_{X}=(M(x_1)^\intercal\cdots M(x_N)^\intercal)^\intercal$, where
\begin{equation*}
	{
	\tiny
M(x_i)=
\begin{pmatrix}
x_i^\intercal&0&\dots&0&1&0&\dots&0\\
0&\ddots&\ddots&\vdots&0&\ddots&\ddots&\vdots\\
\vdots&\ddots&\ddots&0&\vdots&\ddots&\ddots&0\\
0&\dots&0&x_i^\intercal&0&\dots&0&1\\
\end{pmatrix}}\quad \textnormal{ for }i\in\left\{ 1,\cdots,N \right\}.
\end{equation*}
Note that for $i\in\left\{ 1,\ldots,N \right\}$, $M(x_i)\tilde \theta=W_\theta x_i+b_\theta$ expresses the first layer affine transformation of $f_\theta$ applied to the predictor $x_i$ such that $M_X\tilde\theta$ is a long vector containing groups of the entries $M(x_1)\tilde\theta,\ldots,M(x_N)\tilde\theta$. To each of these groups, also the remaining transformations of $f_\theta$ need to be applied. In order to express this notationally conveniently, for $a,b,l\in\mathbb{N}$ we define 
the $l$-fold diagonal replication $D_l(W)$ of $W\in\mathbb{R}^{a\times b}$ and $l$-fold stacked vector $T_l(b)$ of $b\in\mathbb{R}^{a}$ by
\begin{equation*}
	\tiny
D_{l}(W)=
\begin{pmatrix}
W&0&\dots&0\\
0&\ddots&\ddots&\vdots\\
\vdots&\ddots&\ddots&0\\
0&\dots&0&W
\end{pmatrix}
\in\mathbb{R}^{la\times lb}
,\quad
T_{l}(b)=
\begin{pmatrix}
b\\
\vdots\\
b\\
\end{pmatrix}
\in\mathbb{R}^{la}.
\end{equation*}

Using the notation 
$g_{N}^{(i)}: \mathbb{R}^{Nn_{i-1}}\to\mathbb{R}^{Nn_{i}},\; x\mapsto\textnormal{ReLU.}\left( D_N(W_i)x+T_N(b_i) \right)$
for $i\in\left\{ 2,\dots,L \right\}$ and $Y=(y_1,\dots,y_N)\in\mathbb{R}^{N}$ it follows that
\begin{align*}
	\scriptsize\tiny
	&\mathcal{L}(\theta)=\sum_{i=1}^{N}\vert f_{\theta}(x_i)-y_i\vert = \sum_{i=1}^{N}\vert W_{L+1}g_N^{(L)}\circ\dots\circ g_{N}^{(2)}(  \textnormal{ReLU.}(W_\theta x_i+ b_\theta))+b_{L+1}-y_i\vert\\
=& \left( 1,\dots,1 \right)_{1\times N} \textnormal{ReLU.}\Big( -Y+T_N\left( b_{L+1}\right)+ 
D_N(W_{L+1})g_N^{(L)}\circ \dots\circ g_{N}^{(2)} \left( \textnormal{ReLU.}(M_X\tilde \theta) \right)\Big)\\
&+ \left( 1,\dots,1 \right)_{1\times N} \textnormal{ReLU.}\Big( Y-T_N\left( b_{L+1} \right)- D_N(W_{L+1})g_N^{(L)}\circ 
	\dots\circ g_{N}^{(2)}
\left( \textnormal{ReLU.}(M_X\tilde \theta )\right)\Big)\\
=& \left( 1,\dots,1 \right)_{1\times 2N} \textnormal{ReLU.}\Big( 
\begin{pmatrix}
-Y+T_N\left( b_{L+1}\right)\\
Y-T_N\left( b_{L+1}\right)
\end{pmatrix}+ 
{\scriptsize \begin{pmatrix}
D_N(W_{L+1})\\
-D_N(W_{L+1})
\end{pmatrix}}
g_N^{(L)}\circ 
\dots\circ g_{N}^{(2)}
\left( \textnormal{ReLU.}(M_X\tilde \theta) \right)\Big)
\qedhere
\end{align*}
It follows that a network $f_{\mathcal{L}}$ that computes $\mathcal{L}$ can be realized using $L+1$ hidden layers of widths $(n_0,\cdots,n_{L+1})=(n_1(n_0+1),Nn_1,\cdots,Nn_{L},2N)$. The predictor and response samples affect the first layer matrix and the $L+1$-th layer bias vector respectively.

Note that for $Nn_1>n_1(n_0+1)$ the origin $0\in\mathbb{R}^{n_1(n_0+1)}$ of the above constructed network $f_{\mathcal{L}}$ will generally be a vertex that is \emph{not} regular because the first layer bias vector is $0$ such that every of the $Nn_1$ neurons in the first layer may change their activation locally around the origin. Thus, our algorithm will only be applicable as long as the iterated positions are regular vertices.

Furthermore, note that similar to the above construction, also the L1 loss as a function of the weight and bias parameters of any other layer with index $\ell\in\left\{ 1,\ldots,L \right\}$ can be written as a feed forward neural for fixed parameters in the remaining layers. In this case, the transformed predictors $g^{(\ell-1)}\circ\cdots\circ g^{(1)}(x_i)$, $i\in\left\{ 1,\ldots,N \right\}$ take the role of the predictors above and induce the matrix $M_X$. However, the $\textnormal{ReLU}$ activation function applied during this transformation often causes a linear dependence structure among the transformed predictors such that whenever the ReLU arguments are zero for some neurons, automatically those of other neurons are also zero, for example in the case multiple transformed predictors being completely zero. This causes corresponding vertices to be non-regular, thus rendering our algorithm unusable. Due to this reason, we restricted to first layer's parameter training here.
\begin{example}
	\label{ex:training}
	For the topology $L=3$, $(n_0,n_1,n_2,n_3,n_4)=(4,5,4,2,1)$ we sample the network weights and bias vectors from a uniform distribution and apply our DRLSimplex algorithm to the function $\mathcal{L}$ above rewritten as a neural network for 500 training samples also sampled from a uniform distribution. Our algorithm produces a sequence $n_1(n_0+1)=25$-dimensional estimates for the first layer network parameters with strictly decreasing loss depicted in Figure~\ref{fig:nntraining}, a property that most other training methods do not guarantee. After a finite number of steps convergence is reached. 
	
	Interestingly, the decay in training loss is approximately exponential. Note that since our algorithm iterates on vertices, this might empirically provide information about the density of vertices around local minima of the loss $\mathcal{L}$. 
	An interesting question to ask is whether this decay graph provides information about the quality of the converged local minimum: Assuming an exponential decay, maybe a overly fast convergence might indicate premature trapping in a suboptimal undesirable local minimum. A further question for future analysis is whether such vertex density information can be used to improve the step size control in standard gradient descent in the sense that a high density requires small step sizes and vice versa.
\begin{figure}[htpb]
\centering
\includegraphics[width=0.6\linewidth]{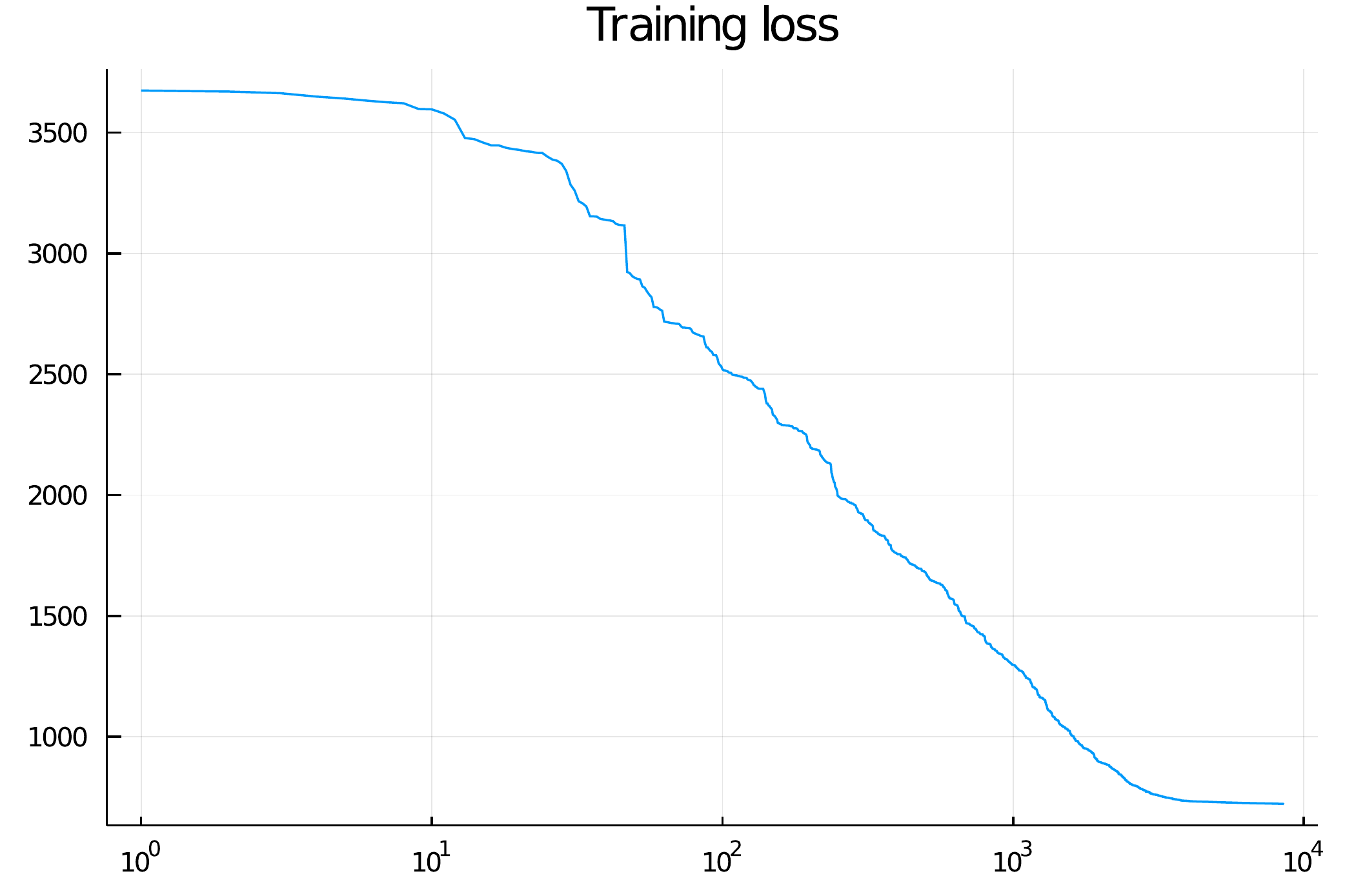}
\caption{
	Training loss for the setup explained in Example~\ref{ex:training} across different epochs.
}
\label{fig:nntraining}
\end{figure}
\end{example}

The widths of the hidden layers of $f_{\mathcal{L}}$ scale with $N$ such that one step of our algorithm is of order $\mathcal{O}(N^2)$ for increasing $N$, see Section~\ref{sec:computationalComplexity}. This undesired blowup is caused by the large replication matrices and stacked vectors used in the network $f_{\mathcal{L}}$ above which compute the parallel evaluation of the original network $f$ on the predictor samples. Instead, one could consider a modification of our algorithm adapted to such multiple evaluations of the same function $f$ on different samples. This way, the quadratic dependence in $N$ can be improved to a linear dependence. We will pursue this idea in future work.

\subsubsection{LASSO optimization}
\label{sec:lassoOptimization}
	For a data matrix $\mathbb{X}\in\mathbb{R}^{N\times p}$ and response vector $Y\in\mathbb{R}^{N}$ as in Section~\ref{sec:censoredLAD} and a penalty parameter $\lambda>0$, the classical LASSO~\cite{tibshirani1996regression} aims to optimize the loss function
\begin{equation*}
\mathcal{L}(\theta):
\begin{cases}
	\mathbb{R}^{p}&\to\mathbb{R}\\
\theta&\mapsto \|Y-X\theta\|_2+\lambda\sum_{i=1}^p|\theta_i|
\end{cases}.
\end{equation*}
where $\lambda>0$ is a hyper parameter.
In Section~\ref{sec:modifications} we explained how our algorithm can be modified such that mixtures \eqref{eq:quadraticObjectiveFunction} of quadratic forms and feed-forward neural networks can be optimized. 
This modification then no longer iterates on vertices of the neural network but is rather an specialization of an active-set method from quadratic programming with the same ideas as our DRLSimplex algorithm. For brevity and clarity of our exposition we do not explicitly provide pseudo-code for this modification. Instead, we provide a Julia implementation in our code repository, see Section~\ref{sec:juliaImplementation}. 

We do not aim to compete with existing implementations but rather want to demonstrate how our methodology and theoretical results may be used for the implementation of new algorithms for piece-wise quadratic optimization of objective functions whose structure is of the form~\eqref{eq:quadraticObjectiveFunction}.
For a random sample data set indeed, our implementation based on the above discussed modification of the DRLSimplex algorithm converges to the same solution as the corresponding R LASSO estimator of the \texttt{glmnet} package.

\subsection{Discussion}
\label{sec:discussion}
In this section we want to first discuss our DRLSimplex algorithm in the context of linear programming and gradient descent. Then we point out the benefits that arise from our contribution.
\subsubsection{Relation to the standard simplex algorithm}
As noted in Section~\ref{sec:drlpproblem}, deep ReLU programming can be seen as an extension of linear programming because the objective function $f$ is piece-wise affine on convex domains
\begin{equation}
	R_s:=\big\{x\in\mathbb{R}^{n_0}\mid\; s\in \mathcal{S}^C(x)  \big\}, \quad s\in\mathcal{S}
\end{equation}
each of which has the form of a feasible region in linear programming. In this sense linear programming focusses only on an affine objective on the feasible region while deep ReLU linear programming considers multiple regions $R_s$ with their own affine objective function $f|_{R_s}$ for $s\in\mathcal{S}$. At the boundaries of these regions their objective functions are equal by the continuity of the feed-forward ReLU neural network $f$. This means that deep ReLU programming can be seen as a patchwork of many linear programming problems which satisfy a continuity condition and the algorithmic difference of our DRLSimplex algorithm compared to the standard simplex algorithm is the ability to change the considered region $R_s$ by switching $s\in\mathcal{S}$ in a computationally efficient way that allows to reuse previously computed axes.

Section~\ref{sec:drlpproblem} also showed that every linear program can be rewritten as a deep ReLU program by constructing a ReLU neural network that has the solution of the linear program as its global minimum. However, this is only a theoretical statement and such a rewrite is of no practical use since our algorithm does not provide any computational runtime advantage over existing solvers for linear programs. Instead, the advantage of our DRLSimplex algorithm is the ability to iterate through different affine regions to find local minima of feed-forward ReLU neural networks, whereas the standard simplex algorithm does not leave its hard-coded feasible region.

For linear programming problems there exist inner point methods which have polynomial runtime for the required number of iterations in terms of the number of conditions and the dimensionality. Their iteration steps are within the feasible region of the considered linear program. In contrast, for deep ReLU programming problems it is very unlikely that the  currently considered affine region of the objective function has a local minimum of $f$ as one the vertices on its boundary. Instead probably many different affine regions have to be traversed such that it is questionable whether a transfer of ideas from inner point methods from linear programming to deep ReLU programming is beneficial.

For the variants of the standard simplex algorithm, it is generally known that the worst case number of iterations until convergence is exponential in the number of variables and inequality conditions. The reformulation of a linear programming problem as a deep ReLU programming problem presented in Section~\ref{sec:drlpproblem} shows that our DRLSimplex algorithm inherits this exponential worst case number of iterations as it also iterates on vertices.
\subsubsection{Relation to gradient descent procedures}
Given a differentiable function $\varphi:\mathbb{R}^{n_0}\to\mathbb{R}$ and a starting point $x_0\in\mathbb{R}^{n_0}$, gradient descent-like algorithms iteratively compute a sequence of points $x_0, x_1,\ldots\in\mathbb{R}^{n_0}$ with an update rule of the form
\begin{equation}
	\label{eq:updateRule}
	x_{i+1}=x_i-\eta_i\nabla \varphi(x_i),\quad i\in\mathbb{N}
\end{equation} for positive step size parameters $\eta_0,\eta_1,\ldots\in\mathbb{R}$. This is motivated by the fact that for every $x\in\mathbb{R}^{n_0}$
\begin{equation*}
	\mathbb{R}^n_0\setminus \left\{ 0 \right\}\to\mathbb{R}, v\mapsto \lim_{t\to 0}\frac{\varphi(x+tv)-\varphi(x)}{t\|v\|_2}=\langle \frac{v}{\|v\|_2},\nabla\varphi(x)\rangle
\end{equation*}
is minimized by $v=-\alpha \nabla\varphi(x)$, $\alpha>0$ by Cauchy-Schwartz inequality. However, this is only the optimal direction to decrease $\varphi$ in an infinitely small neighbourhood around $x$ such that the step size in equation~\eqref{eq:updateRule} has to be chosen with care. The DRLSimplex algorithm has a similar update rule, however the arguments generated in its iterations are placed on vertices which are the boundaries between affine-regions. Here the gradient is not defined and instead of the gradient, an axis separating such regions is used as the direction for the next iteration step. 

The step size is automatically determined maximally, such that the assumed affine behaviour is still valid. This is achieved by exploiting the structure of feed-forward ReLU neural networks. 
In contrast to standard gradient descent-like algorithms this adaptation to such functions allows the update procedure to decrease the function value at every iteration. This is a very strong property since every new iteration will either strictly decrease the objective function or our algorithm stops. For regular networks with vertices, this stopping condition explained in Section~\ref{sec:drlp} exactly determines local minima such that for these networks, our algorithm finds the exact position of a local minimum after a finite number of iterations. Usually, other gradient descent-like algorithms do not have these properties but they are not restricted to piece-wise affine functions.

Whenever it is defined, the gradient $\nabla f(x)$ for a feed forward neural network $f$ as in equation~\eqref{eq:f} at position $x\in\mathbb{R}^{n_0}$ can be computed in $\mathcal{O}(\sum_{l=0}^L n_{l+1}n_l)$ by the chain rule in calculus and also the update rule~\eqref{eq:updateRule} is of this order. In contrast, we have shown in Section~\ref{sec:computationalComplexity} that one iteration in our DRLSimplex algorithm is of order $\mathcal{O}(\sum_{l=0}^Ln_{l+1}n_{l}+n_0^2)$. In particular, when $n_0^2$ is dominated by one of $n_{i+1}n_{i}$, $i\in\left\{ 0,\ldots,L \right\}$, one iteration in our algorithm is of the same order as in gradient descent for $\max(n_0,\ldots,n_{L+1})\to\infty$. 

Concerning the required number of iterations until convergence it is plausible that for large $n_0,\ldots,n_{L}$, the automatically selected step size in deep ReLU programming is small because the input space $\mathbb{R}^{n_0}$ is split into many small regions, each with its own affine behaviour. This can slow down the minimization progress, especially at the beginning when large step sizes in equation~\eqref{eq:updateRule} are appropriate. In contrast, at later stages of the minimization process, when a small step size has to be chosen in gradient descent-like algorithms to further minimize the objective function, the automatic step size-control of the DRLSimplex algorithm can be beneficial such that a hybrid algorithm may be interesting, especially because of the exponential number of iterations in a worst case scenario inherited from the standard simplex algorithm from linear programming.

Furthermore the use of an axis that is separating affine regions of $f$ instead of its gradient can reduce the number iteration steps needed in situations as depicted in Figure~\ref{fig:axisVsGradient}. In this case, gradient descent-like algorithms will oscillate between two affine regions because for every new iteration, they will follow direction of the gradient and will eventually overjump the region boundary. While modifications such as the momentum method~\cite{qian1999momentum} can circumvent this problem to some extend, the DRLSimplex algorithm would first find the intersection point and then follow the region separating axis in the direction that decreases the function value until a neuron activation changes, thus completely avoiding unnecessary oscillations. 
\begin{figure}[htpb]
\centering
\includegraphics[width=0.5\textwidth]{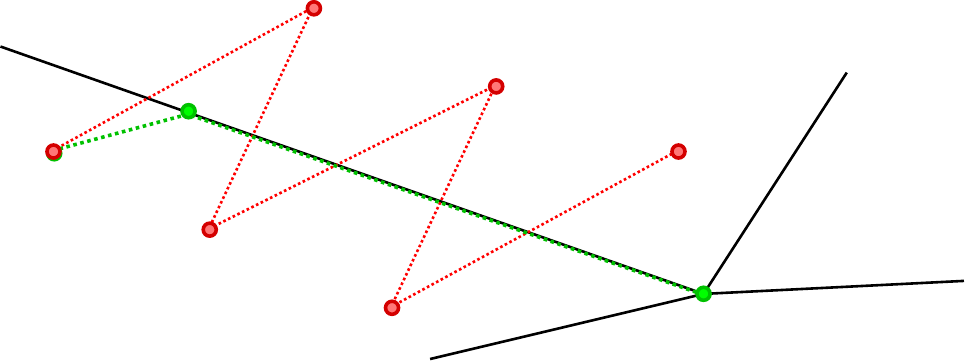}
\caption{Qualitative comparison of possible behaviours of the DRLSimplex algorithm (green) versus other gradient descent-like algorithms (red) in three dimensions (two-dimensional input space) in starting at the same position $x_0\in\mathbb{R}^2$. The dotted lines connect the points corresponding to subsequent iteration steps. The black lines represent the boundaries of the affine regions of $f$.}
\label{fig:axisVsGradient}
\end{figure}

\subsubsection{Possible use cases and contribution}

Theoretical foundations such as the stability concept of several quantities are the justification for the use these quantities as state parameters in our algorithm and guarantee that derived quantities like the critical neurons are not specific to the choice of the compatible activation pattern at the current position. In this sense, our theoretical concepts and results are interesting in their own right because they are useful for the construction of similar linear-programming or active set related algorithms for the minimization of feed-forward ReLU neural networks. In this context also the algorithmic primitives we presented in Section~\ref{sec:algPrimitives} may be relevant.

Our DRLSimplex algorithm generalizes the standard simplex algorithm from linear programming because it can iterate on vertices generated by a feed-forward neural network and not only on vertices given by the intersection of hyperplanes. Compared to the simplex algorithm, it can change the feasible region and does this in a computationally efficient way such that the computational complexity $\mathcal{O}(\sum_{l=1}^{L}n_{l}n_{l-1}+n_0^2)$ per iteration still reflects the unavoidable terms $\sum_{l=1}^L n_ln_{l-1}$ (number of network parameters) and $n_0^2$ (simplex tableau) and is therefore not negatively affected.

Important use cases of our DRLSimplex algorithm emerge from the insight that several L1 loss functions relevant for machine-learning can be rewritten as a minimization problem of a ReLU feed-forward neural network. As we have shown in Section~\ref{sec:examples}, corresponding special purpose solver algorithms such as the Barrodale-Roberts algorithm for quantile regression of the BRCENS algorithm for CLAD regression can be seen as incarnations of our DRLSimplex algorithm for one- and two-hidden-layer networks respectively. More generally, our algorithm provides a novel training method for the first layer parameters of a feed forward neural network to find a local minimum of the L1 loss with fixed training data. Compared to other training methods we get exact positions of vertices and local minima and furthermore obtain a strictly decreasing training loss at every new vertex returned by our algorithm. A hybrid combination with other gradient-descent algorithms might provide additional benefits in an early training phase where fast training is more important than guaranteed strict training loss decrease.

Furthermore, the iteration on vertices might provide a new tool for an empirical vertex density analysis in the case where the training loss can be written as a ReLU feed-forward neural network. For example as discussed previously in Example~\ref{ex:training}, the loss decay behaviour there seems to behave exponential in the number of visited vertices such that this training loss curve might contain information on the quality of a local minimum converged to. In addition, local vertex density could be estimated by our DRLSimplex algorithm using the distance of vertex positions which in turn could serve as an input to tune or improve the step size control in other gradient descent training algorithms.
\section{Summary}
We have introduced the class of deep ReLU programming problems as non-convex optimization problems of finding a local minimum of a real-valued feed-forward neural network $f$ with ReLU activation functions. These functions are piece-wise affine on domains which are given by sets of linear inequalities. In particular, if restricted to one such domain, the resulting minimization problem has the form of a linear programming problem where the restricted domain takes the role of the feasible region. Hence, deep ReLU programming can be seen as a generalization of linear programming with the difference that we allow multiple feasible regions, each with its own affine objective function such that the overall function is continuous.

We developed notation and theory suitable for the analysis of these functions in Section~\ref{sec:analysis}. 
For an input value $x$ we defined activation patterns $s$ to be compatible if replacing the ReLU activation functions by hard-coded activities specified in $s$ preserves the output and all intermediate layer values of $f$, i.e. if the network $f$ can be described by fixed affine layer maps induced by $s$. We call quantities \emph{objective} when they are defined based on the true network $f$ and \emph{subjective} when they are defined based on this fixed affine layer network induced by some activation pattern $s$. Quantities that have an objective and meaningful subjective versions are \emph{stable}, if the objective version coincides with the subjective versions for compatible $s$. We have introduced several local network quantities such as the critical kernel and the critical indices and proved stability. We further formulated the concept of \emph{regularity} which is a requirement in our stability result for the critical indices. For a random continuous independent distribution on its parameters the resulting network is almost surely regular which shows that regularity is not a strong requirement. Our stability results lay the foundation for algorithms that use the activation of the neurons as a state parameters because for they identify subjective quantities that do not need to be updated when the state parameter $s$ is changed from one compatible activation pattern to another, thus a change of the ``feasible region'' identified by $s$ without the need to update or recompute such quantities. We then formalized the notion of a vertex, its corresponding axes that separate affine regions of $f$ and formulated a necessary and sufficient condition for $x$ being a local minimum of $f$.
	
In addition we provided algorithmic subroutines suitable for the construction of iterative optimization algorithms for deep ReLU programming problems in Section~\ref{sec:algPrimitives}. In Section~\ref{sec:description} we combined these subroutines with our theoretical concepts and results to formalize our DRLSimplex algorithm, a generalization of the standard simplex algorithm to deep ReLU programming that can be used to iteratively optimize regular feed-forward ReLU neural networks. In its first phase it first finds a vertex and then iterates on vertices while changing the affine region and keeping track of a compatible activation pattern and a set of critical indices that define the region separating axes to proceed with. The automatic step size selection exploits the piece-wise constant gradient and ensures that the objective function cannot increase across iterations. Our algorithm stops based on the local minimum condition for regular networks we formulated previously. For an $L$ layer network of widths $n_0,\ldots,n_L$ we showed that every iteration is of computational complexity $\mathcal{O}(\sum_{i=1}^Ln_i n_{i-1}+n_0^2)$, where the first and the second terms reflect the number of parameters and simplex tableau size, such that this complexity order is optimal for an extension of the simplex algorithm to ReLU neural networks. We further proposed modifications for better numeric stability, other activation functions, alternatives to vertex iterations and generalizations for mixtures of feed forward ReLU nets and quadratic forms as objective functions.

For our DRLSimplex algorithm we provided a proof-of-concept implementation in the Julia programming language in Section~\ref{sec:juliaImplementation} which is demonstrated in Section~\ref{sec:examples} at optimization problems that can be rewritten in the form of a deep ReLU programming problem. We showed that our algorithm can be applied in LASSO penalized quantile regression, censored least absolute deviation and even for training of the first layer of a deep ReLU network using L1 loss. For the first and second applications in this list, there exist specialized linear programming solver algorithms that can be seen as special cases of our DRLSimplex algorithm, for the third example one usually resorts to standard neural network training algorithms not specifically designed to take advantage of the piece-wise affine structure of the ReLU network to be trained. Here our algorithm provides a simplex-like iteration on vertices of the loss function such that the decreases in every step. In future work it will be interesting to apply deep linear programming for training of not only the first layer of ReLU neural networks but also other layers iteratively. This would yield a training procedure which monotonically decreases the L1 training loss across iterations, a major advantage compared to other optimization methods. We also demonstrated how our DRLSimplex method can be generalized for objective functions that are a mixture of a feed-forward ReLU neural network and a positive definite form: Our implementation is able to successfully perform LASSO optimization.

We finally discussed our DRLSimplex algorithm and compared it to linear programming and gradient descent variants in Section~\ref{sec:discussion}.
The overall computational complexity for an iteration in our algorithm of $\mathcal{O}(n_0^2+\sum_{i=1}^Ln_in_{i-1})$ allows it to be considered as a new competitive optimization tool. Especially for loss functions that can be rewritten in the form of a feed-forward ReLU neural network, in a combination other gradient-descent based algorithms could benefit from the strictly decreasing loss guaranteed by our algorithm's automatic step size control or the fact that exact vertex positions are returned. In particular the latter property could be used to empirically analyze the vertex density of such loss functions which in turn could provide new insight in their structure or allow for improved step-size control of gradient-descent algorithms based on local vertex density. 

We hope that our theoretical framework, our algorithmic building blocks and our DRLSimplex algorithm open the door for the application of linear programming and active-set related techniques to feed-forward ReLU neural networks. This could help to analyze and better understand the vertex structure of loss functions related to ReLU neural networks and might eventually lead to improved neural network training algorithms.
\label{sec:summary}

\newpage
\appendix

\section{Proofs and auxiliary results}
\label{app:proofs}
In the sequel, we use the notation introduced in Sections~\ref{sec:introduction} and \ref{sec:analysis}.
\subsection{Compatibility}
\label{app:compatibility}
\begin{proof}[Proof of Theorem~\ref{thm:argumentRepresentation}]
Fix $s\in \mathcal{S}^C(x)$. For $(l,j)\in \mathcal{I}$, if $s_{lj}=0$ then by equation~\eqref{eq:compatibleSignatures} $\textnormal{sign}(A(x)_{lj})=H(x)_{lj}\in\left\{ 0,-1 \right\}$ and therefore $\textnormal{ReLU}(A(x)_{lj})=0$. Similarly $s_{lj}=1$ implies $\textnormal{ReLU}(A(x)_{lj})=A(x)_{lj}$. Hence $g^{(1)}(x)=\tilde g^{(1)}_{s_1}(x)$ and for $l\in\left\{ 2,\ldots,L \right\}$,
\begin{equation*}
g^{(l)}\circ\cdots\circ g^{(1)}(x)=\textnormal{ReLU.}(A(x)_{l})
=\textnormal{diag}(s_l)A(x)_l=\tilde g^{(l)}_{s_l}\circ g^{(l-1)}\circ\cdots\circ g^{(1)}(x).
\end{equation*}A recursive application of the above formula yields $A(x)=\tilde{A}_s(x)$.

Now fix $s\in\tilde{\mathcal{S}}^C(x)$. By equation~\eqref{def:representingSignaturesFixed} $s$ is compatible with $\tilde{H}_s(x)$. In particular for $(l,j)\in\mathcal{I}$,  $s_{lj}=0$ implies $\textnormal{sign}(\tilde A_s(x)_{lj})=\tilde H_s(x)_{lj}\in\left\{ -1,0 \right\}$ and $\textnormal{ReLU}(\tilde A_s(x)_{lj})=0$. Similarly $s_{lj}=1$ implies $\textnormal{ReLU}(\tilde A_s(x)_{lj})=\tilde{A}_s(x)_{lj}$. This implies
\begin{equation}
\label{eq:reluDiag}
\forall l\in\left\{ 1,\ldots,L \right\}\quad \textnormal{ReLU.}(\tilde A_s(x)_l)=\textnormal{diag}(s_l)\tilde A_s(x)_l 
\end{equation}
We now prove  $A(x)=\tilde A_s(x)$ by induction.  Note that $A(x)_1=W_1x+b_1=\tilde A_s(x)_1$.  
For the induction step assume $A(x)_l=\tilde A_s (x)_{l}$ for some $l\in\left\{ 1,\ldots,L-1 \right\}$. Then by equation~\eqref{eq:reluDiag}
\begin{equation*}
	g^{(l)}\circ\cdots\circ g^{(1)}(x)=\textnormal{ReLU.}(A(x)_{l})=\textnormal{ReLU.}(\tilde A_{s}(x)_{l})=\textnormal{diag}(s_{l})(\tilde A_s(x)_{l})=\tilde g_{s_l}^{(l)}\circ\cdots\circ \tilde g_{s_1}^{(1)}(x).
\end{equation*} This implies that $A(x)_{l+1}=W_{l+1}g^{(l)}\circ\cdots\circ g^{(1)}(x)+b_{l+1}=W_{l+1} \tilde{g}_{s_l}^{(l)}\circ\cdots\circ \tilde{g}_{s_1}^{(1)}+b_{l+1}=\tilde A_s(x)_{l+1}$. An induction on $l$ yields $A(x)=\tilde A_s(x)$.

Therefore we have shown that $s\in\mathcal{S}^C(x)\cup\tilde{\mathcal{S}}^C(x)$ implies $A(x)=\tilde A_s(x)$, hence $H(x)=\tilde H_s(x)$ and therefore $\mathcal{S}^C(x)=\tilde{\mathcal{S}}^C(x)$.
\end{proof}
\begin{proof}[Proof of Corollary~\ref{cor:representation}]
This follows from Theorem~\ref{thm:argumentRepresentation} because $\tilde f_s(x)= W_{L+1}\textnormal{ReLU}.(\tilde A_{s}(x)_L)+b_{L+1}$ and $f(x)=W_{L+1}\textnormal{ReLU}(A(x)_L)+b_{L+1}$.
\end{proof}
\begin{proof}[Proof of Corollary~\ref{cor:sameHyperplanePatternSet}]
Let $y\in \mathbb{R}^{n_0}$ with $H(y)=H(x)$. Then $H(y)$ and $H(x)$ are compatible with $s$ such that Theorem~\ref{thm:argumentRepresentation} implies that $\tilde H_s(y)=\textnormal{sign.}(\tilde{A}_s(y))=\textnormal{sign.}(A(y))=H(y)$ and similarly $\tilde H_s(x)=H(x)$. In particular $\tilde H_s(y)=\tilde H_s(x)$. For the other direction assume $y\in\mathbb{R}^{n_0}$ with $\tilde H_s(y)=\tilde H_s(x)$. By assumption and again Theorem~\ref{thm:argumentRepresentation} it holds that $s\in\tilde{\mathcal{S}}^C(x)$, i.e. $s$ is compatible with $\tilde H_s(x)$ and hence also with $\tilde H_s(y)$ such that $s\in\tilde{\mathcal{S}}^{C}(y)$. Like above, it follows from the same theorem that $H(y)=\tilde H_s(y)=\tilde H_s(x)=H(x)$.
\end{proof}
\subsection{Local Behaviour}
\begin{lemma}
\label{lem:localCompat}
For all $x\in\mathbb{R}^{n_0}$, there exists $\varepsilon>0$ such that for all $y\in B_\varepsilon(x)$ it holds that $\mathcal{S}^C(y)\subset\mathcal{S}^C(x)$.
\begin{proof}
Let $\varepsilon>0$ such that $H(x)_{lj}=H(y)_{lj}$ for $y\in B_\varepsilon(x)$ and $(l,j)\in\mathcal{I}\setminus C(x)$. By Lemma~\ref{lem:zero} $H(x)_{lj}=0$ for $(l,j)\in C(x)$ such that for all $y\in B_\varepsilon(x)$ any $s\in\mathcal{S}$ compatible with $H(y)$ is also compatible with $H(x)$.
\end{proof}
\end{lemma}
This means that  for every input $x\in\mathbb{R}^{n_0}$ there exists a neighbourhood such that for all $y$ in this neighbourhood all $s$ compatible with $H(y)$ are also compatible with $H(x)$.
\begin{proof}[Proof of Lemma~\ref{lem:tildeKernel}]
	If $v^*\in\tilde{\textnormal{Ker}}{}^C_s(x)$, then there exists $\varepsilon^*>0$ such that for all $(l,j)\in \tilde C_s(x)$ and $t\in(-\varepsilon,\varepsilon)$ $\tilde H_s(x+tv^*)_{lj}=\tilde H_s(x)_{lj}=0$. In particular for all $t\in(-\varepsilon,\varepsilon)$ it holds that $\tilde A_s(x+tv^*)_{lj}=0$. By affinity of $z\mapsto \tilde A_s(z)_{lj}$, it holds that $\tilde A_s(x+v^*)_{lj}=0$. 
For the other direction assume that $v^*\in\mathbb{R}^{n_0}$ with $\tilde A_s(x+v^*)_{lj}=0$ for $(l,j)\in\tilde C_s(x)$. Let $\varepsilon^*>0$ such that for all $v\in B_{\varepsilon^*}(0)$ and $(l,j)\mathcal{I}\setminus\tilde C_s(x)$ it holds that $\tilde H_s(x+v)_{lj}=\tilde H_{s}(x)_{lj}$. This is possible by equation~\eqref{eq:subjectiveC}. Then for all $t\in\mathbb{R}^{n_0}$ such that $t v^*\in B_{\varepsilon^*}(0)$ it holds that $\tilde H_s(x+tv^*)=\tilde H_s(x)$ since also for indices $(l,j)\in \tilde C_s(x)$ it holds that $\tilde H_s(x+tv)=0=\tilde H_s(x)$ by Lemma~\ref{lem:zero}.
\end{proof}
\begin{proof} [Proof of Lemma~\ref{lem:hyperKernelEquality}]
By Corollary~\ref{cor:sameHyperplanePatternSet} $\big\{ v\in\mathbb{R}^{n_0}\mid H(x+v)=H(x) \big\}=\big\{ v\in\mathbb{R}^{n_0}\mid \tilde H_s(x+v)=\tilde H_s(x) \big\}$. Let $\varepsilon^*>0$ small enough such that $H_s(x+v)_{ij}$ is constant for indices $(i,j)\in \mathcal{I}\setminus \tilde C_s(x)$ and $v\in B_{\varepsilon^*}(0)$. This is possible by equation~\eqref{eq:subjectiveC}. For all other indices $(i,j)\in \tilde C_s(x)$, the subjective hyperplane pattern satisfies $\tilde H_s(x)_{ij}=0$ by Lemma~\ref{lem:zero}. It follows that $B_{\varepsilon^*}(x)\cap \big\{ v\in\mathbb{R}^{n_0}\mid \tilde H_s(x+v)=\tilde H_s(x) \big\}= B_{\varepsilon^*}(x)\bigcap_{(l,j)\in \tilde C_s(x)}\left\{ v\in\mathbb{R}^{n_0}\mid \tilde A_s(x+v)_{lj}=0 \right\}$. Lemmas~\ref{lem:tildeKernel} and \ref{lem:sameKernel} complete the proof.
\end{proof}

\subsection{Regularity}
\label{app:regularity}
\subsubsection{Basic results}
\begin{lemma}
	\label{lem:collinear}
	Assume that $x\in\mathbb{R}^{n_0}$, $(l^*,j^*)\in \mathcal{I}$ and $s,s'\in \mathcal{S}^C(x)$ with $s_{lj}=s'_{lj}$ for $(l,j)\in\mathcal{I}\setminus \left\{ (l^*,j^*) \right\}$ and $s_{l^*j^*}\neq s'_{l^*j^*}$. If $v_{s,\tilde l,\tilde j}=0$ and $v_{s',\tilde l,\tilde j}\neq 0$ for some $(\tilde l, \tilde j)\in \mathcal{I}$ then $v_{s',\tilde l,\tilde j}=\alpha v_{s,l^*,j^*}$ for some $\alpha\in\mathbb{R}$.
	\begin{proof}
		By equation~\eqref{eq:v} $\tilde l>l^*$ and 
		\begin{equation*}
		v_{s',\tilde l,\tilde j}=v_{s',\tilde l,\tilde j}-v_{s,\tilde l,\tilde j}=\left( W_{\tilde l}\textnormal{diag}(s_{\tilde l-1})W_{l-1}\cdots \textnormal{diag}(s_{l^*+1})W_{l^*+1} \right)_{\tilde j j^*} v_{s,l^*,j^*}
	\end{equation*} for $s'_{j^*l^*}=1$. Similarly, for $s'_{l^*j^*}=0$ we obtain the negative right-hand side.
	\end{proof}
\end{lemma}
\begin{proof}[Proof of Theorem~\ref{thm:regularity}]
		For two compatible activation patterns $s,s'\in \mathcal{S}^C(x)$ that differ only at one position $(l^*,j^*)\in \mathcal{I}$, it holds that $(l^*,j^*)\in \tilde C_s(x)\cap \tilde C_{s'}(x)$ by equations~\eqref{eq:subjectiveArgs} and \eqref{eq:subjectiveC}. If $\tilde C_s(x)\neq \tilde C_{s'}(x)$ then there exists $(l',j')\in \tilde C_{s}(x)\cup \tilde C_{s'}(x)$ with both, $v_{s,l^*,j^*}=0$ and $v_{s',l^*,j^*}\neq 0$ or both, $v_{s',l^*,j^*}=0$ and $v_{s,l^*,j^*}\neq 0$. But by Lemma~\ref{lem:collinear} this violates the regularity assumption such that $\tilde C_s(x)=\tilde C_{s'}(x)$.

		For two arbitrary compatible activation patterns $s,s'\in \mathcal{S}^C(x)$ that differ at $k>1$ positions we can find a transformation path $s=:s^{(1)},\ldots,s^{(k+1)}:=s'\in \mathcal{S}^C(x)$ such that two subsequent activation patterns $s^{(i)},s^{(i+1)}$, $i\in\left\{ 1,\ldots,k \right\}$ differ only at one position $(l_i,j_i)\in \mathcal{I}$. Now the above result implies $\tilde C_{s^{(1)}}(x)=\ldots = \tilde C_{s^{(k)}}(x)$, in particular $\tilde C_s(x)=\tilde C_{s'}(x)$. 

		We now have shown that there exists $C\subset \mathcal{I}$ with $\tilde C_s(x)=C^*$ for all $s\in\mathcal{S}^C(x)$. Now equations~\eqref{eq:C} and $\eqref{eq:subjectiveC}$ imply that this set $C^*$ of indices must be exactly $C(x)$.
	\end{proof}
\subsubsection{Probabilistic results on regular vertices}
Below we use the notation and concepts of Section~\ref{sec:analysis}. For $s\in\mathcal{S}$, $(l,j)\in\mathcal{I}$ let
\begin{equation}
	\label{eq:hyperplanes}
	H^s_{lj}:=\left\{ x\in\mathbb{R}^{n_0}\middle\vert\;\langle x,\tilde v_{s,l,j}\rangle+
		\left( \sum_{i=1}^{l-1} W_l\textnormal{diag}(s_{l-1})\cdots W_{i+1}\textnormal{diag}(s_i)b_i \right)_j+b_{lj} 
	=0\right\}.
\end{equation}
Note that $H^s_{lj}=\big\{ x\in\mathbb{R}^{n_0}\mid\; \tilde H_s(x)_{lj}=0 \big\}$.
\begin{lemma}
Assume the $\sum_{i=2}^{L+1}n_ln_{l-1}$ weight parameters of the matrices $W_1,\ldots,W_{L}$ and the $\sum_{i=2}^{L+1}n_i$ bias parameters of the vectors $b_1,\ldots,b_L$ of the neural network $f$ are sampled from a distribution such that conditionally on the weight parameters, the bias parameters are independent with a conditional marginal distribution that assigns probability zero to all finite sets. 
Then the following holds almost surely: For all $n\in\mathbb{N}$, $s_1,\ldots,s_n\in\mathcal{S}$, $(l_1,j_1),\ldots,(l_n,j_n)\in\mathcal{I}$ with non-zero normal vectors $\tilde v_{s_1,l_1,j_1},\ldots,\tilde v_{s_n,l_n,j_n}$ the hyperplanes $H^{s_1}_{l_1,j_1},\ldots,H^{s_n}_{l_n,j_n}$ defined by equation~\eqref{eq:hyperplanes} satisfy either $\cap_{i=1}^n H^{s_i}_{l_i,j_i}= \left\{  \right\}$ or the normal vectors $\tilde v_{s_1,l_1,j_1},\ldots,\tilde v_{s_n,l_n,j_n}$ are linearly independent.
	\label{lem:randomHyperplanes}
	\begin{proof}
		We proceed by induction. For $n=1$ the claim is obviously true. Now assume the claim is true for $n=n^*\in\mathbb{N}$ and let $s_1,\ldots,s_{n^*+1}\in\mathcal{S}$, $(l_1,j_1),\ldots,(l_{n^*+1},j_{n^{*}+1})\in\mathcal{I}$. Conditionally on the weight parameters, the normal vectors 
		\begin{equation}
			\label{eq:normalvectorsProof}
			\tilde v_{s_1,l_1,j_1},\ldots,\tilde v_{s_{n^*+1},l_{n^*+1},j_{n^*+1}}
		\end{equation} are almost surely constant and by assumption the bias parameters are independent with a continuous distribution. We distinguish two cases:
\begin{enumerate}
	\item Conditionally on the weight parameters, at least one of the above normal vectors is equal to zero.
	\item Conditionally on the weight parameters, all of the above normal vectors are nonzero. In this case let $k\in\left\{ 1,\ldots,n^*+1 \right\}$ with $l_k=\max_{i\in\left\{ 1,\ldots,n^* \right\}}l_i$ and let $J:=\left\{ 1,\ldots,n^*+1 \right\}\setminus\{k\}$. Now we condition on the weight parameters and on $b_{l_i,j_i}$, $i\in J$. In this case, the bias vector $b_{l_k,j_k}$ has a continuous distribution. If $\cap_{i=1}^{n^*+1}H^{s_i}_{l_i,j_i}\neq \{\}$ then by induction assumption the vectors $\tilde v_{s_i,l_i,j_i}$, $i\in J$ are linearly independent. If $\tilde v_{s_i,l_i,j_i}$, $i\in\left\{ 1,\ldots,n^* \right\}$ are linearly dependent, $\cap_{i=1}^{n^*+1}H^{s_i}_{l_i,j_i}\neq \{\}$ implies that $\cap_{i\in J}H^{s_i}_{l_i,j_i}\subset H^{s_k}_{l_k,j_k}$ and this is only possible when $b_{l_k,j_k}$ attains a specific value $b^*\in\mathbb{R}^*$ which almost surely does not happen.
\end{enumerate}
The above discussion shows that with probability zero the normal vectors in equation \eqref{eq:normalvectorsProof} are all non-zero, linearly dependent and the hyperplanes satisfy $\cap_{i=1}^{n^*+1}H^{s_i}_{l_i,j_i}\neq \{\}$. Since there are only finitely many choices for $s_1,\ldots,s_{n^*+1}\in\mathcal{S}$, $(l_1,j_1),\ldots,(l_{n^*+1},j_{n^{*}+1})\in\mathcal{I}$ the claim holds almost surely for $n=n^*+1$.
	\end{proof}
\end{lemma}

\begin{proof}[Proof of Theorem~\ref{thm:regularNetwork}]
	The claim is just a reformulation of Lemma~\ref{lem:randomHyperplanes}.
\end{proof}

\subsection{Feasible directions}
\label{app:feasibleDirections}
\begin{lemma}
\label{lem:compatibleReluArgumetns}
For $x\in\mathbb{R}^{n_0}$ it holds that
$s\in \mathcal{S}^C(x)$ if and only if $\tilde A_s(x)_{lj}(s_{lj}-0.5)\ge 0$ for all $(l,j)\in\mathcal{I}$.
\begin{proof}
This follows from equations~\eqref{eq:compatible} and \eqref{eq:subjectiveHPattern}.
\end{proof}
\end{lemma}
\begin{proof}[Proof of Lemma~\ref{lem:feasibleFormulation}]
Let $v^*\in\mathbb{R}^{n_0}$ with $\tilde A_s(x+v^*)_{lj}(s_{lj}-0.5)\ge 0$ for all $\tilde C_s(x)$. Choose $\varepsilon^*>0$ small enough such that for indices $(l,j)\in\mathcal{I}\setminus \tilde C_s(x)$, the subjective hyperplane pattern is constant, i.e. such that for all $v\in B_{\varepsilon^*}(0)$ it holds that $\tilde H_s(x+v)_{lj}=\tilde H_s(x)_{lj}$. This is possible by equation~\eqref{eq:subjectiveC}. Now let $\varepsilon>0$ such that $\varepsilon v^*\in B_{\varepsilon^*}(0)$. In particular at indices $(l,j)\in\mathcal{I}\setminus \tilde C_s(x)$ $\tilde H_s(x+\varepsilon v^*)_{lj}=\tilde H_s(x)_{lj}$ is compatible with $s$. For the other indices $(l,j)\in \tilde C_s(x)$ it holds by assumption and Lemma~\ref{lem:tildeV} that 
\begin{equation*}
\tilde A_s(x+\varepsilon v^*)_{lj}(s_{lj}-0.5)=\varepsilon\langle v^*,\tilde v_{s,l,j}\rangle (s_{lj}-0.5)=\varepsilon\tilde A_s(x+ v^*)_{lj}(s_{lj}-0.5)\ge 0.
\end{equation*}
Using Lemma~\ref{lem:compatibleReluArgumetns} it follows that $\tilde H_s(x+\varepsilon v^*)$ is compatible with $s$.  
To prove the other direction let $v\in\mathbb{R}^{n_0}$ such that for some $(l,j)\in\tilde C_s(x+v)$ the condition $A_s(v)_{lj}(s_{lj}-0.5)\ge 0$ does \emph{not} hold. Similar to above it follows that $H(x+\varepsilon v)$ cannot be compatible with $s$ for any $\varepsilon>0$.
\end{proof}
\begin{corollary}
\label{cor:feasibleCompatible}
For an input $x\in\mathbb\mathbb{R}^{n_0}$ and a compatible activation pattern $s\in\mathcal{S}^{C}(x)$ it holds that all directions $v\in\mathbb{R}^{n_0}$ such that $s\in\mathcal{S}^C(x+v)$ satisfy $v\in\tilde D_s(x)$.
\begin{proof}[Proof of Corollary~\ref{cor:feasibleCompatible}]
This follows immediately from Lemmas~\ref{lem:compatibleReluArgumetns} and \ref{lem:feasibleFormulation}.
\end{proof}
\end{corollary}
The above corollary gives a sufficient criterion for directions being feasible directions in terms of compatibility.
\begin{proof}[Proof of Theorem~\ref{thm:identicalAxes}]
	For $r:=s_{lj}=s'_{lj}$, let
\begin{equation*}
\mathcal{W}_s:=\left\{v\in\mathbb{R}^{n_0}\mid \forall (l,j')\in \mathcal{I}\setminus \left\{ (l,j) \right\}\;\tilde H_s(x+v)_{l'j'}=\tilde H_s(x)_{l'j'} \land  (r-0.5)\tilde A_s(x+v)_{lj}>0 \right\}
\end{equation*} and define $\mathcal{W}_{s'}$ similarly. Corollary \ref{cor:sameHyperplanePatternSet} implies $\mathcal{W}_s=\mathcal{W}_{\tilde s}$.
By equation~\eqref{eq:subjectiveC} and Theorem~\ref{thm:regularity} it is possible to find $\varepsilon_s,\varepsilon_{s'}>0$ such that for $y_s\in B_{\varepsilon_s}(x)$, $y_{s'}\in B_{\varepsilon_{s'}}(x)$ the hyperplane patterns $\tilde H_s(y_s)_{lj}$, $\tilde H_{s'}(y_{s'})_{lj}$ are constant at indices $(l,j)\not\in C(x)=\tilde C_{s}(x)=\tilde C_{s'}(x)$.
Since $B_{\min(\varepsilon_s,\varepsilon_{s'})}(x)$ contains an open ball around $x$, it follows that
\begin{align*}
	&\; \left\{v\in\mathbb{R}^{n_0}\mid \forall (l,j')\in \tilde C_s(x)\setminus \left\{ (l,j) \right\}\;\tilde H_s(x+v)_{l'j'}=\tilde H_s(x)_{l'j'} \land  (r-0.5)\tilde A_s(x+v)_{lj}>0\right\}=\\
	&\; \left\{v\in\mathbb{R}^{n_0}\mid \forall (l,j')\in \tilde C_{s'}(x)\setminus \left\{ (l,j) \right\}\;\tilde H_{s'}(x+v)_{l'j'}=\tilde H_{s'}(x)_{l'j'} \land  (r-0.5)\tilde A_{s'}(x+v)_{lj}>0\right\}.
\end{align*} 
Now note that by equation~\eqref{eq:axisCondition}, $\tilde a_{s,x,l,j}$ is an element of the first of the above sets and by equivalence also of the second. In particular, for $\varepsilon^*>0$ small enough $x+\varepsilon^*\tilde a_{s,x,l,j}\in \mathcal{W}_s=\mathcal{W}_{\tilde s}$ such that $s,s'\in \mathcal{S}^C(x+\varepsilon^*\tilde a_{s,x,l,j})$. Now Lemma~\ref{lem:tildeV} and Theorem~\ref{thm:argumentRepresentation} imply
\begin{align*}
	\varepsilon^*\langle \tilde a_{s,x,l,j},\tilde v_{s',l',j'}\rangle&= \tilde A_{s'}(x+\varepsilon^* a_{s,x,l,j})_{l'j'} = A(x+\varepsilon^* \tilde a_{s,x,l,j})_{l'j'} = \tilde A_{s}(x+\varepsilon^* \tilde a_{s,x,l,j})_{l'j'}\\&= \varepsilon^*\langle \tilde a_{s,x,l,j},\tilde v_{s,l',j'}\rangle=\varepsilon^*\delta_{(l,j),(l',j')}.
\end{align*}
for all critical indices $(l',j')\in C(x)$.
By equation~\eqref{eq:axisCondition} this means 
$\tilde a_{s,x,l,j}=\tilde a_{s',x,l,j}$.
\end{proof}
\subsection{Regular vertices and local minima}
\label{app:regularVerticesLocalMinima}
\begin{proof}[Proof of Lemma~\ref{lem:minimumVertex}]
	If a strict local minimum $x\in\mathbb{R}^{n_0}$ is not a vertex then there exists a non-zero $v\in\textnormal{Ker}^C(x)$. But by equation~\eqref{eq:hyperKernel} cannot be a strict local minimum. For the second claim let $x^*$ be a non-strict local minimum of $f$, i.e. there exists $\varepsilon^*>0$ such that $f(x^*)=\inf\big\{f(x)\mid\; x\in B_{\varepsilon}(x^*)  \big\}$. If $\textnormal{df}^C(x^*)>0$ then choose a non-zero $v\in \textnormal{Ker}^C(x^*)$. The map $t\mapsto f(x^*+tv)$ is an affine function for $t\in \left\{ t\in\mathbb{R}\mid\; H(x^*+tc)=H(x^*) \right\}$. By equation~\eqref{eq:hyperKernel} there exists $\varepsilon>0$ such that $(-\varepsilon,\varepsilon)\subset \left\{ t\in\mathbb{R}\mid\; H(x^*+tv)=H(x^*) \right\}$. This implies that, the function $\mathbb{R}\to\mathbb{R}, t\mapsto f(x^*+tv)$ is differentiable at $t=0$ with derivative $0$. Then for $t^*=\min\left\{ t>0\middle\vert\; H(x^*+tv)\neq H(x^*) \right\}$ it holds that $f(x^*)=f(x^*+tv)$ and $\textnormal{df}^C(x^*+tv)<\textnormal{df}^C(x^*)$. Repetition of this procedure leads to a point $x^{**}$ with $\textnormal{df}^C(x^{**})=0$, i.e. a vertex with $f(x^{**})=f(x^*)$ after a finite number of steps.
\end{proof}
\begin{lemma}
\label{lem:oneSided}
For $x\in\mathbb{R}^{n_0}$ and $v\in\mathbb{R}^{n_0}\setminus\left\{ 0 \right\}$ the limit
$\lim_{t\searrow 0}\frac{1}{t}( f(x+tv)-f(x) )$
exists. Furthermore assume that $\varepsilon^*>0$ such that for all $(l'j')\in \mathcal{I}\setminus C(x)$, $H(y)_{lj}$ is constant for $y\in B_{\varepsilon^*}(x)$ and let $t^*>0$ small enough such that $x+t^*v\in B_{\varepsilon^*}$. For $s\in \mathcal{S}^{C}(x+t^*v)$ it holds that
\begin{equation*}
\lim_{t\searrow 0}\frac{1}{t}\left( f(x+tv)-f(x) \right)= \lim_{t\searrow 0}\frac{1}{t}\left( \tilde f_s(x+tv)-\tilde f_s(x) \right)=\langle v ,\nabla_s\rangle
\end{equation*}
\begin{proof}
Note that $H(x)_{lj}=H(x+t^*v)_{lj}$ for $(l,j)\in \mathcal{I}\setminus C(x)$ and $H(x)_{lj}=0$ for $(l,j)\in C(x)$. Hence $s$ is compatible with $H(x)$ and by assumption $s$ is compatible with $H(x+t^*v)$. By Theorem~\ref{thm:argumentRepresentation} it follows that $s$ is compatible with both $\tilde H_s(x)$ and $\tilde H_s(x+t^*v)$. By equation~\eqref{eq:subjectiveHPattern} and the fact that all coordinates of $y\mapsto \tilde A_s(y)$ are affine functions it follows that $s$ is compatible with all $\tilde H_s(x+tv)$ for $t\in[0,t^*]$ and hence again by Theorem~\ref{thm:argumentRepresentation} with all $H(x+tv)$ for $t\in[0,t^*]$. Applying the same theorem again yields $A(x+tv)=\tilde A_s(x+tv)$ and hence $f(x+tv)=\tilde f_s(x+tv)$ for all $t\in[0,t^*]$. The result now follows from the fact that $\tilde f_s$ is affine with gradient exactly $\nabla_s$ as in equation~\eqref{eq:gradient}.
\end{proof}
\end{lemma}
\begin{corollary}
\label{cor:localexpansion}
Let $x\in\mathbb{R}^{n_0}$ be a regular vertex and $\varepsilon^*>0$ such that for all $(l'j')\in \mathcal{I}\setminus C(x)$, $ H(y)$ is constant for $y\in B_{\varepsilon^*}(x)$. For every $v\in B_{\varepsilon^*}(0)$ there exist non-negative coefficients $\alpha_1,\ldots,\alpha_{n_0}\in [0,\infty)$ and $n_0$ region separating axes $a_1,\ldots,a_{n_0}\in \mathcal{A}(x)$ such that
\begin{equation*}
	f(x+v)= f(x)+\sum_{i=1}^{n_0}\alpha_{i}\lim_{t\searrow 0}\frac{1}{t}\left( f(x+ta_i)-f(x) \right).
\end{equation*}
\begin{proof}
	Let $v\in B_{\varepsilon^*}(0)$ and $s\in \mathcal{S}^C(x)$ such that $s$ is compatible with $H(x+v)$. Then $v\in \tilde D_s(x)$ by Corollary~\ref{cor:feasibleCompatible}. Lemma~\ref{lem:positiveSpan} and equation~\eqref{eq:feasibleRegionseparating} imply that there exist coefficients $\alpha_{lj}\in[0,\infty)$, $(l,j)\in C(x)$ such that
\begin{equation}
\label{eq:proofDirectionConvex}
v=\sum_{(l,j)\in \tilde C_s(x)}\alpha_{lj} \tilde a_{x,s,l,j}=\sum_{(l,j)\in \tilde C_s(x)}\alpha_{lj} a_{x,s_{lj},l,j}
\end{equation}This implies that
\begin{align*}
	f(x+v)&\overset{(1)}{=} \tilde f_s(x+v)\overset{(2)}{=}\tilde f_s(x+\sum_{(l,j)\in \tilde C_s(x)}\alpha_{lj} a_{x,s_{lj},l,j})
	\overset{(3)}{=}\tilde f_s(x)+\langle \sum_{(l,j)\in \tilde C_s(x)}\alpha_{lj} a_{x,s_{lj},l,j},\nabla_s\rangle\\
	&\overset{}{=} f(x)+\sum_{(l,j)\in \tilde C_s(x)}\alpha_{lj}\langle  a_{x,s_{lj},l,j},\nabla_s\rangle 
	\overset{(4)}{=} f(x)+\sum_{(l,j)\in \tilde C_s(x)}\alpha_{lj}\left( \tilde f_s(x+a_{x,s_{lj,l,j}})-\tilde f_s(x) \right) \\
	&\overset{(5)}{=} f(x)+\sum_{(l,j)\in \tilde C_s(x)}\alpha_{lj}\lim_{t\searrow 0}\frac{1}{t}\left( \tilde f_s(x+ta_{x,s_{lj},l,j})-\tilde f_s(x) \right) \\
	&\overset{(6)}{=} f(x)+\sum_{(l,j)\in \tilde C_s(x)}\alpha_{lj}\lim_{t\searrow 0}\frac{1}{t}\left( f(x+ta_{x,s_{lj},l,j})-f(x) \right) .
\end{align*}
The justification of the steps is as follows: (1) follows from Corollary~\ref{cor:representation}, (2) from equation~\eqref{eq:proofDirectionConvex}, (3), (4) and (5) from the fact that the functions $f_{s'}$, $s'\in\mathcal{S}$ are affine and (6) is an implication of Lemma~\ref{lem:oneSided}. Furthermore $|\tilde C_s(x)|=n_0$ by the regularity assumption.
\end{proof}
\end{corollary}
\begin{proof}[Proof of Proposition~\ref{prop:localMinimum}]
The claim follows from Lemmas~\ref{lem:oneSided} and Corollary~\ref{cor:localexpansion}.
\end{proof}
\section{Bounds on the number of affine regions}

\label{sec:NumReg}
In the sequel, let $N(n_0,\dots,n_L)$ denote the maximal number of affine regions that a ReLU feed-forward neural network of depth $L$ and widths $n_0,\dots,n_L$ can have. We will now present some bounds on $N(n_0,\dots,n_L)$.

The work \cite{Montufar:2014:NLR:2969033.2969153} establishes the following bound on the number of affine regions:
\begin{lemma}[Mont\'ufar]
Let $L\in\mathbb{N}$. A neural network with $L+2$ layers ($L$ hidden layers) of widths $n_0,\dots,n_{L+1}$ is a piece-wise affine function with at most 

\begin{equation}
\label{eq:Montufar}
N(n_0,\dots,n_L)\le \prod_{i=1}^L\sum_{j=0}^{\min(n_0,\dots,n_i)}{n_i\choose j}
\end{equation}
affine regions.
\label{lem:counting}
\end{lemma}
It uses a result from 1943 by R.C.Buck, which bounds the number of regions that are carved out of $\mathbb{R}^{n_0}$ by $n_1$ hyperplanes, see~\cite{buck1943partition}:
\begin{lemma}[Buck]
In $\mathbb{R}^d$, $N \in\mathbb{N}$ hyperplanes cut the space $\mathbb{R}^d$ into at most $\sum_{j=0}^{d}{N\choose j}$ regions.
\label{lem:zaslavsky}
\end{lemma}
This bound was improved in~\cite{DBLP:BoundingCounting} to
\begin{equation}
\label{eq:BoundingCounting}
N(n_0,\dots,n_L)\le \sum_{j_1,\dots,j_L\in J}^{}\prod_{i=1}^L{n_{i}\choose j},
\end{equation}where $J=\left\{ (j_1,\dots,j_L)\in\mathbb{N}^L\vert\;\forall i\in\left\{ 1,\dots,L \right\}:\;j_i\le\min(n_0,n_1-j_1,\dots,n_{L-1}-j_{L-1},n_L)\right\}$. Later, an abstract framework for the construction of such upper bounds was presented in~\cite{hinz2018framework}. It yields bounds formulated as the L1 norm of a product of matrices
\begin{equation}
\label{eq:framework}
N(n_0,\dots,n_L)\le\|B^{(\gamma)}_{n_L}M_{n_{L-1},n_L}\dots B^{(\gamma)}_{n_1}M_{n_{0},n_{1}}e^{(n_{0}+1)}_{n_0+1}\|_1,
\end{equation}
where $M_{n_0,n_1},\dots,M_{n_{L-1},n_L}$ are defined by $M_{a,b}=\left(\delta_{i,\min(j,a+1)}\right)_{i\in\left\{ 1,\dots,b+1 \right\},j\in\left\{ 1,\dots,a+1 \right\}}\in\mathbb{N}^{b+1\times a+1}$ for $a,b\in\mathbb{N}_+$, $e^{(n_{0}+1)}_{n_0+1}=\left(\delta_{n_0+1,i}\right)_{i\in\{1,\dots,n_0+1\}}\in\mathbb{N}^{n_0+1}$ and $B^{(\gamma)}_{a}\in\mathbb{N}^{a+1\times a+1}$ are upper triangular matrices depending on some family of vectors $(\gamma_{ij})_{i,j\in\mathbb{N}}$. The bounds~\eqref{eq:Montufar} and~\eqref{eq:BoundingCounting} can be derived from~\eqref{eq:framework}, see~\cite{hinz2018framework}.

\bibliography{DRLProgramming} 
\bibliographystyle{ieeetr}
\end{document}

%% file: tex/packages.tex
\usepackage[margin=1in]{geometry} 
\usepackage[T1]{fontenc}
\usepackage{beramono}
\usepackage{listings}
\usepackage[usenames,dvipsnames]{xcolor}
\usepackage{amsmath,amsthm,amssymb}
\usepackage{dsfont}
\usepackage{algorithm}
\usepackage{algpseudocode}
\usepackage{tabularx}
\usepackage{makecell}

\usepackage{mathdots}
\usepackage{graphicx}
\usepackage{hyperref}
\usepackage{abstract}
\usepackage{enumitem}

\allowdisplaybreaks
\theoremstyle{theorem}
\newtheorem{theorem}{Theorem}
\newtheorem{corollary}[theorem]{Corollary}
\newtheorem{lemma}[theorem]{Lemma}
\newtheorem{proposition}[theorem]{Proposition}

\theoremstyle{definition}

\newtheorem{example}[theorem]{Example}

\lstdefinelanguage{Julia}%
  {morekeywords={abstract,break,case,catch,const,continue,do,else,elseif,%
      end,export,false,for,function,immutable,import,importall,if,in,%
      macro,module,otherwise,quote,return,switch,true,try,type,typealias,%
      using,while},%
   sensitive=true,%
   alsoother={$},%
   morecomment=[l]\#,%
   morecomment=[n]{\#=}{=\#},%
   morestring=[s]{"}{"},%
   morestring=[m]{'}{'},%
}[keywords,comments,strings]%

%% file: tex/abstract.tex
\begin{abstract}
	Feed-forward ReLU neural networks partition their input domain into finitely many ``affine regions'' of constant neuron activation pattern and affine behaviour. We analyze their mathematical structure and provide algorithmic primitives for an efficient application of linear programming related techniques for iterative minimization of such non-convex functions. In particular, we propose an extension of the Simplex algorithm which is iterating on induced vertices but, in addition, is able to change its feasible region computationally efficiently to adjacent ``affine regions''. This way, we obtain the Barrodale-Roberts algorithm for LAD regression as a special case, but also are able to train the first layer of neural networks with L1 training loss decreasing in every step.
\end{abstract}